\documentclass{amsart}
\usepackage[margin=1in]{geometry} 
\usepackage{lipsum}
\usepackage{tikz-cd}
\usepackage[unicode]{hyperref}
\usepackage{graphicx,color}
\usepackage{amssymb,amsmath}
\usepackage{mathtools}
\usepackage{mathrsfs}
\usepackage{enumerate}
\usepackage{ragged2e}
\usepackage{enumitem,pifont}
\usepackage{xypic}
\usepackage{tikz}
\usetikzlibrary{arrows,decorations.pathmorphing,automata,backgrounds}
\usetikzlibrary{backgrounds,positioning}

\usepackage{caption}
\usepackage{subcaption}
\usepackage{float}

\usepackage{egothic}
\usepackage[T1]{fontenc}
\usepackage{yfonts}

\usepackage{calc}

\usepackage{dbnsymb}

\definecolor{ao}{rgb}{0.0, 0.5, 0.0}

\newcommand{\vertex}{\includegraphics[width=0.2in]{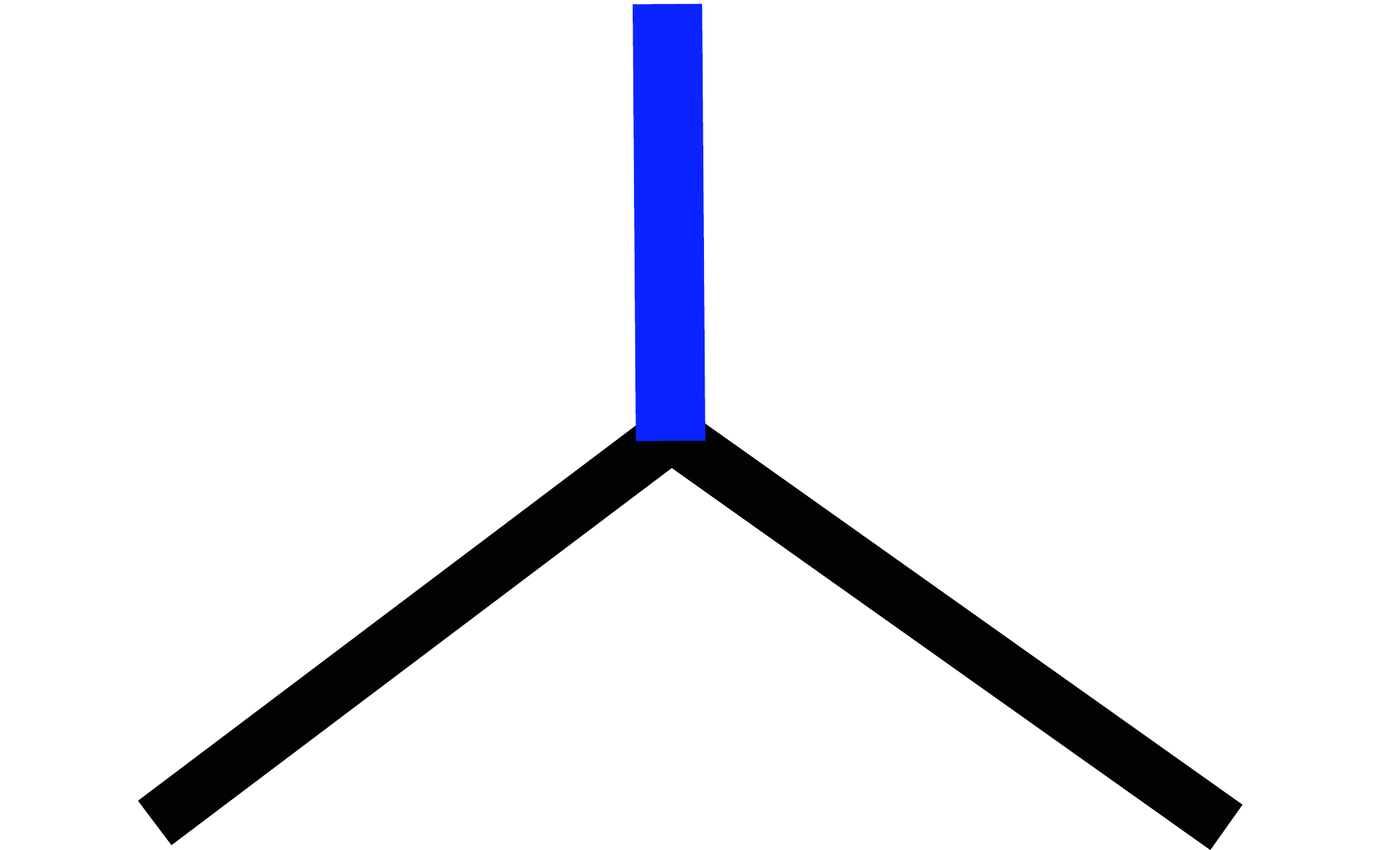}}
\newcommand{\chord}{\includegraphics[height=.1 in, width=0.1in]{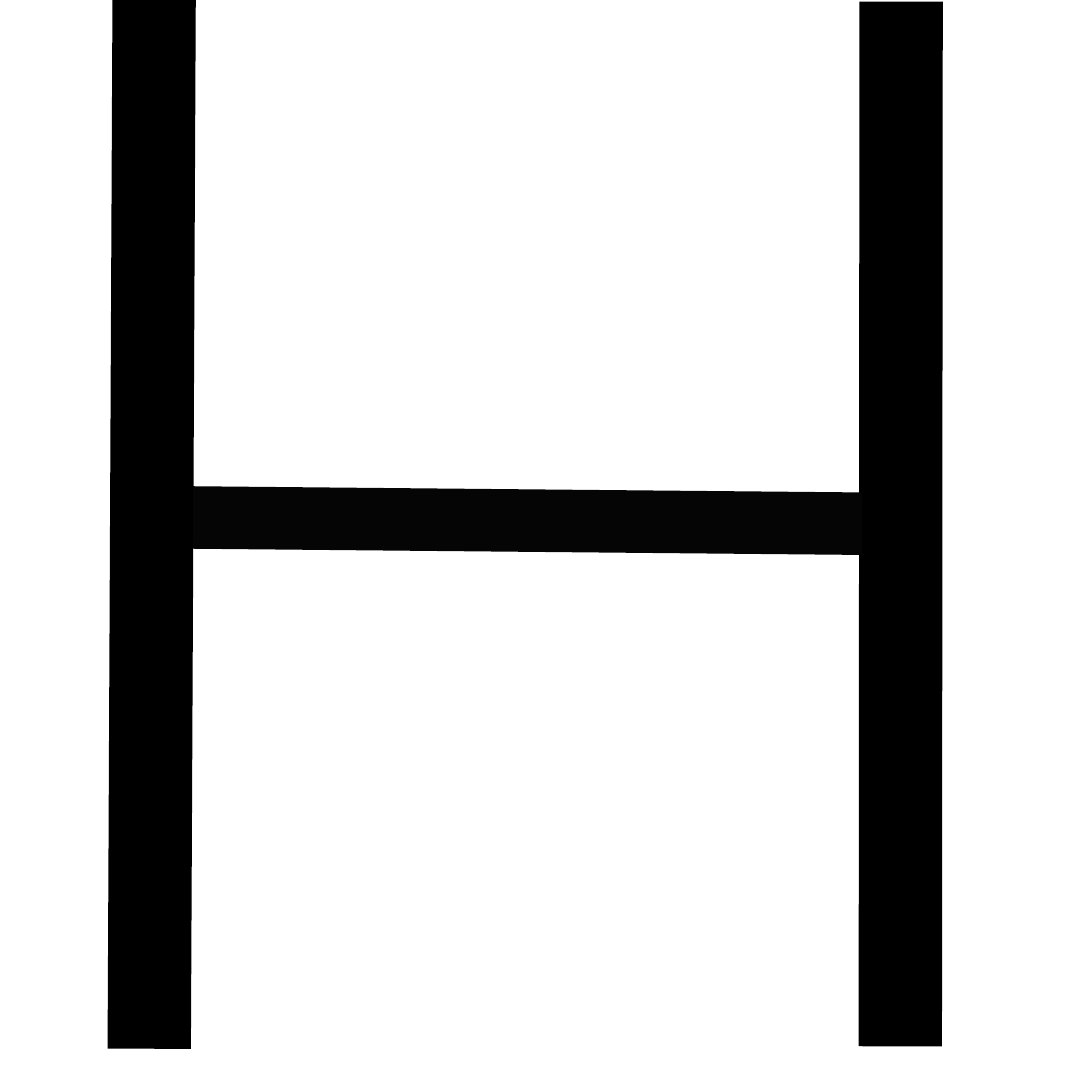}}
\newcommand{\larrow}{\includegraphics[height=.1 in, width=0.1in]{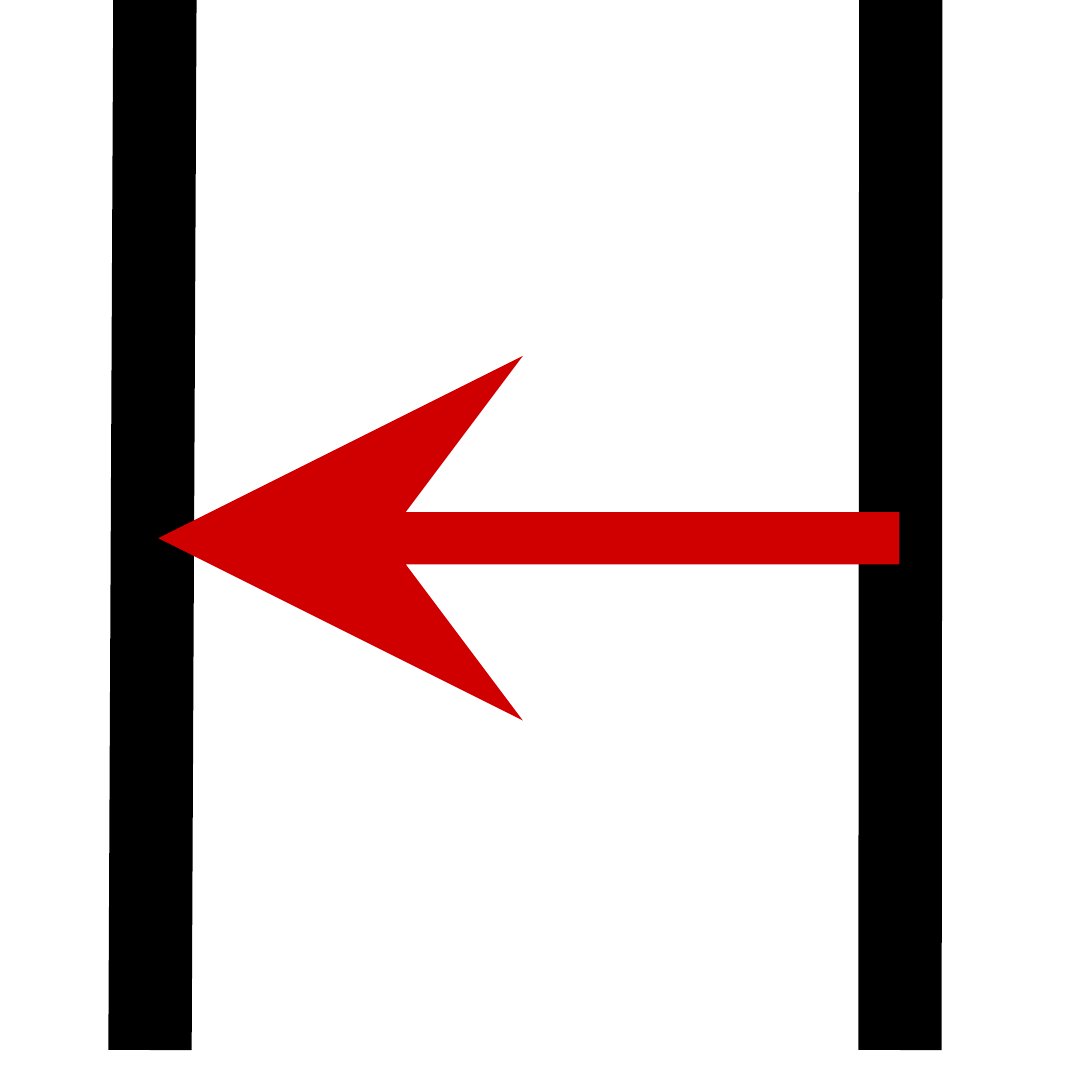}}
\newcommand{\rarrow}{\includegraphics[height=.1 in, width=0.1in]{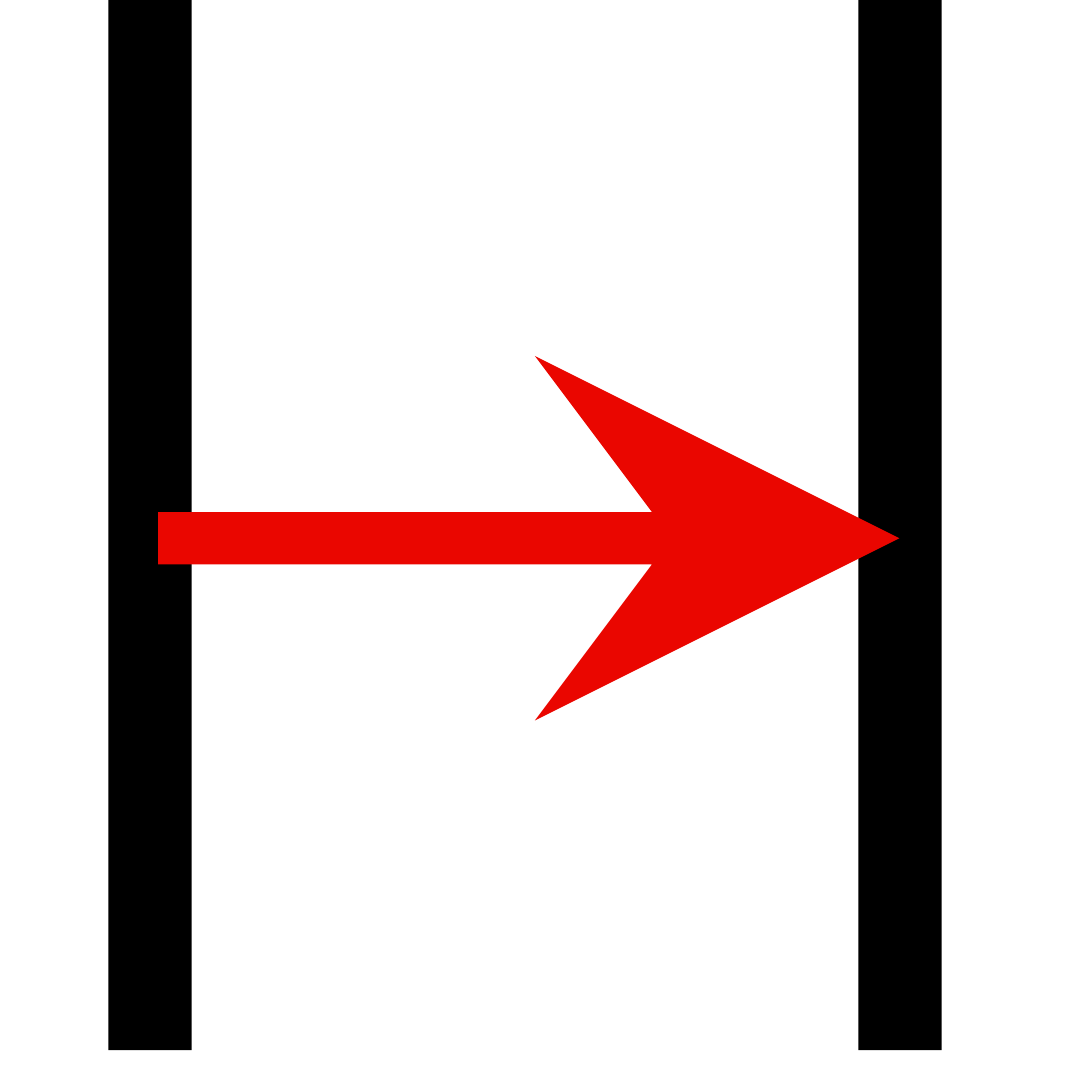}}

\newcommand{\twist}{\includegraphics[height=0.1 in, width=0.15in]{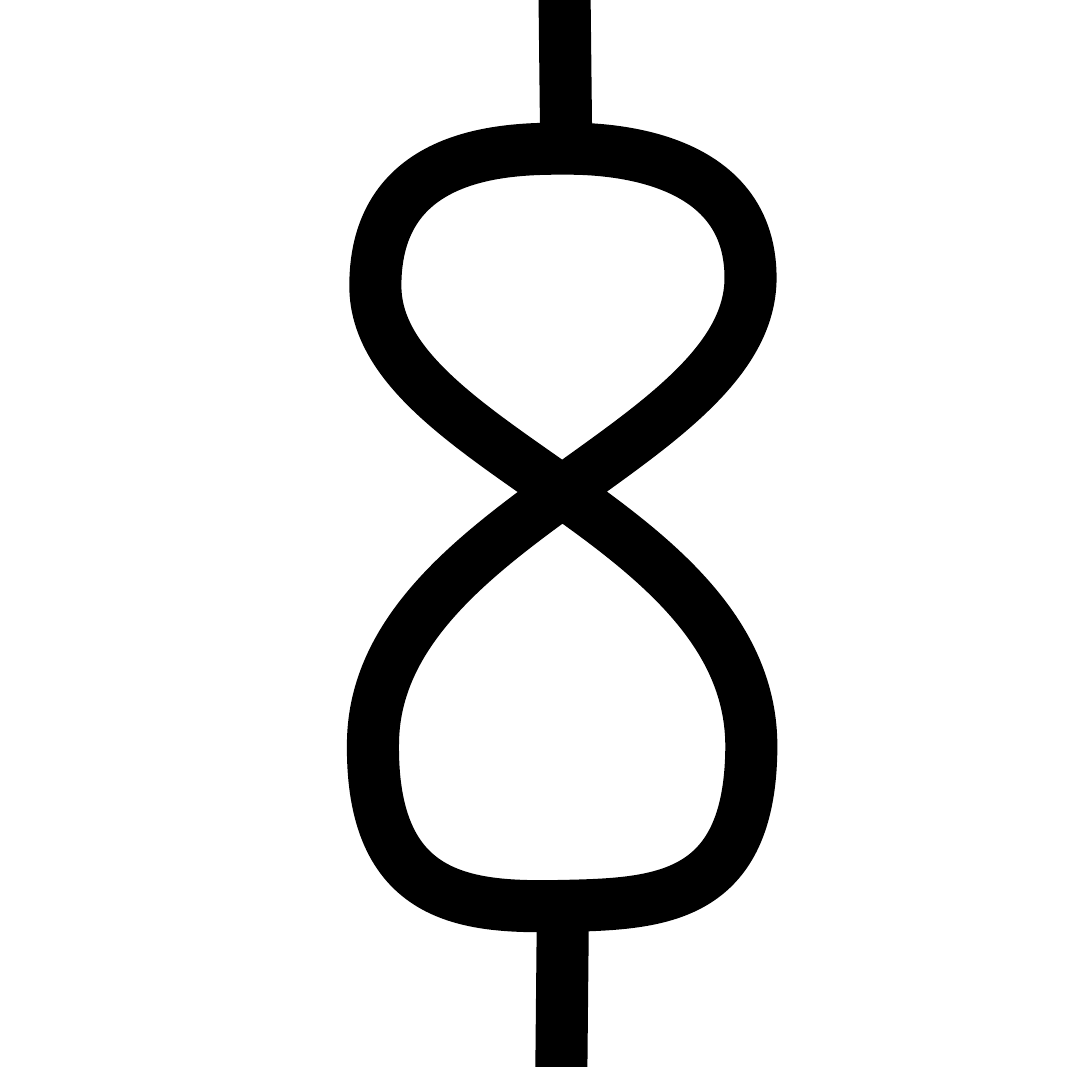}}
\newcommand{\bubble}{\includegraphics[height=0.1 in, width=0.15in]{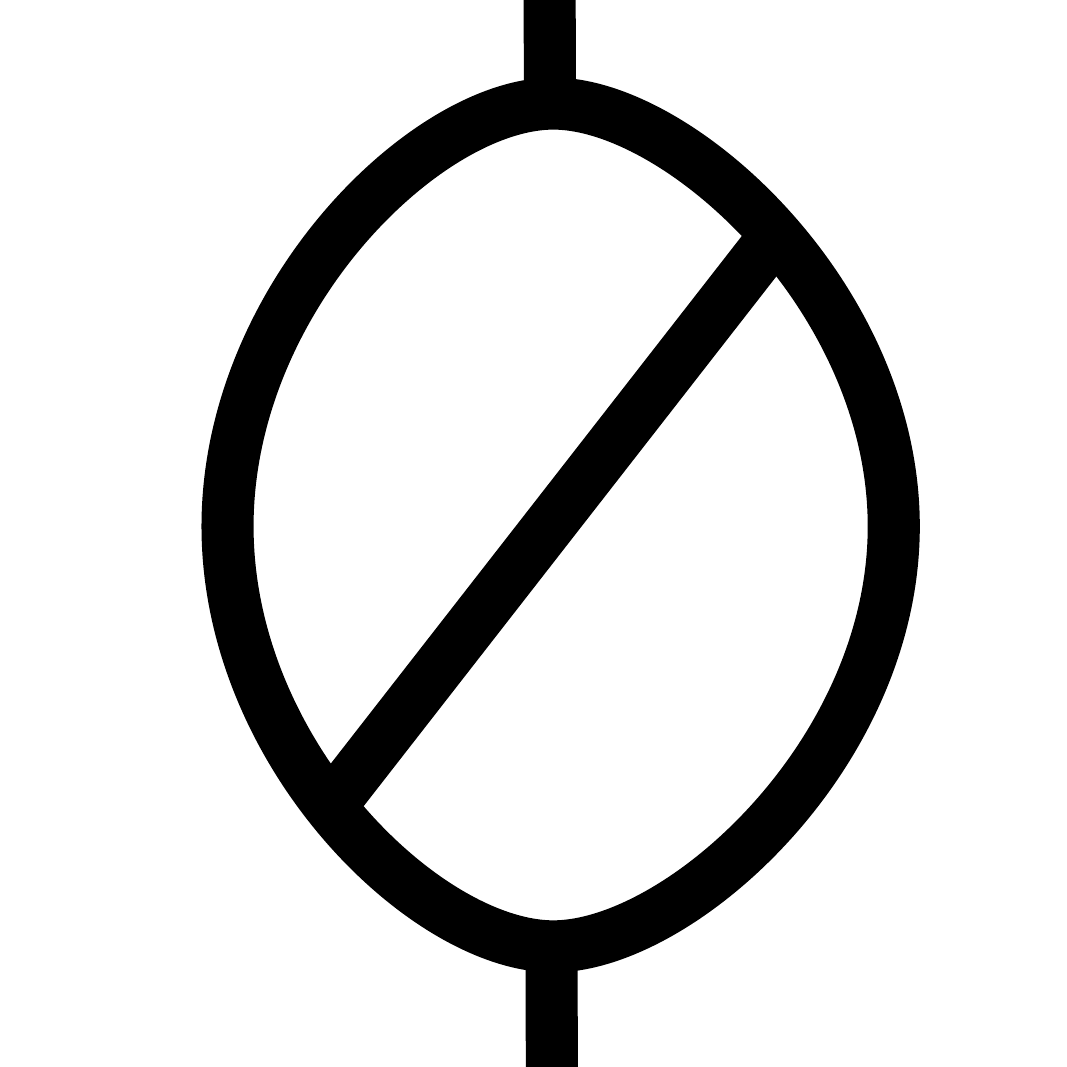}}

\newcommand{\Z}{\mathbb{Z}}

\newcommand{\V}{\mathsf{V}}
\newcommand{\W}{\mathsf{W}}

\newcommand{\WD}{\mathfrak{D}}

\newcommand{\trace}{\textrm{tr}}

\newcommand{\calI}{\mathcal{I}}
\newcommand{\calL}{\mathcal{L}}

\newcommand{\calS}{\mathcal{S}}

\newcommand{\id}{\operatorname{id}}
\newcommand{\Id}{\operatorname{Id}}

\newcommand{\Hom}{\operatorname{Hom}}

\newcommand{\Aut}{\operatorname{Aut}}
\newcommand{\TAut}{\operatorname{TAut}}

\newcommand{\ad}{\operatorname{Ad}}

\newcounter{dummy} \numberwithin{dummy}{section}
 \newtheorem{theorem}[dummy]{Theorem}
\newtheorem{thm}[dummy]{Theorem}
\newtheorem{prop}[dummy]{Proposition}
\newtheorem{lemma}[dummy]{Lemma}

\newtheorem*{thm*}{Theorem}
\newtheorem*{prop*}{Proposition}

\theoremstyle{definition}
\newtheorem{definition}[dummy]{Definition}

\newtheorem{example}[dummy]{Example}
\newtheorem{remark}[dummy]{Remark}
\newtheorem{notation}[dummy]{Notation}
\numberwithin{equation}{section}

\usepackage{multicol}

\newcommand{\gt}{\mathsf{GT}}
\newcommand{\grt}{\mathsf{GRT}}
\newcommand{\kv}{\mathsf{KV}}
\newcommand{\krv}{\mathsf{KRV}}
\newcommand{\EB}{\mathsf{E}}

\newcommand{\PaCD}{\mathsf{PaCD}}
\newcommand{\PaB}{\mathsf{PaB}}

\newcommand{\lie}{\widehat{\mathfrak{lie}}}
\newcommand{\tder}{\mathfrak{tder}}

\newcommand{\cyc}{\mathrm{cyc}}
\newcommand{\ass}{\mathfrak{ass}}

\newcommand{\wf}{\mathsf{wF}}
\newcommand{\hatwf}{\widehat{\wf}}
\newcommand{\arrows}{\mathsf{A}}
\newcommand{\J}{\mathsf{J}}
\newcommand{\A}{\arrows}
\newcommand{\gr}{\textrm{gr}}

\title{A topological characterisation of the Kashiwara--Vergne groups}

\author[Z. Dancso]{Zsuzsanna Dancso}
\address{School of Mathematics and Statistics\\ The University of Sydney\\ Sydney, NSW, Australia}
\email{zsuzsanna.dancso@sydney.edu.au}

\author[I. Halacheva]{Iva Halacheva}
\address{Department of Mathematics \\ Northeastern University \\ Boston, Massachusetts, USA}
\email{i.halacheva@northeastern.edu}

\author[M. Robertson]{Marcy Robertson}
\address{School of Mathematics and Statistics \\ The University of Melbourne \\ Melbourne, Victoria, Australia}
\email{marcy.robertson@unimelb.edu.au}

\date{\today}

\begin{document}
\maketitle\bibliographystyle{amsalpha}
	
\begin{abstract} 
In \cite{BND:WKO2} Bar-Natan and the first author show that solutions to the Kashiwara--Vergne equations are in bijection with certain knot invariants: {\em homomorphic expansions} of {\em welded foams}. 
Welded foams are a class of knotted tubes in $\mathbb{R}^4$, which can be finitely presented algebraically as a {\em circuit algebra}, or equivalently, a {\em wheeled prop}. In this paper we describe the Kashiwara-Vergne groups $\kv$ and $\krv$ -- the symmetry groups of Kashiwara-Vergne solutions -- as automorphisms of the completed circuit algebras of welded foams, and their associated graded circuit algebras of arrow diagrams, respectively. Finally, we provide a description of the graded Grothendieck-Teichm\"uller group $\grt_1$ as automorphisms of arrow diagrams. 
	
\end{abstract}

\tableofcontents

\section{Introduction} 
Universal finite type invariants, or \emph{homomorphic expansions}, are powerful tools in the study of knots and three-manifolds. Given a class of knotted objects $\mathcal{K}$, an expansion $Z:\mathcal{K}\rightarrow \arrows$ takes values in the associated graded space of $\mathcal{K}$ with respect to a {\em Vassiliev filtration}, and satisfies the {\em universality property} that the associated graded map of $Z$ is the identity map of $\arrows$. An expansion is {\em homomorphic} if it respects any operations defined on $\mathcal K$, such as braid or tangle composition, knot connected sum or cabling. The space $\arrows$ is usually combinatorially described as Jacobi or Feynman diagrams. The most famous class of examples is the Kontsevich integral of knots \cite{MR1237836}, and its many variants \cite{BN-Intro}, \cite{BN_Survey_Knot_Invariants}, \cite{LM96}, \cite{LMO98}.

In practice, producing homomorphic expansions is often difficult, and in many cases (e.g. for tangles) they don't exist. When $\mathcal K$ can be finitely presented as an algebraic structure (e.g. parenthesised braids viewed as a category), the Vassiliev filtration coincides with the $I$-adic filtration by powers of the augmentation ideal, and finding a homomorphic expansion is equivalent to solving a set of equations in the associated graded space $\arrows$. In this vein, Bar-Natan  \cite{BNGT} showed that there is a bijection between homomorphic expansions for parenthesised braids and {\em Drinfeld associators}: objects in quantum algebra defined as the set of solutions to the {\em pentagon} and {\em hexagon} equations.  The existence of homomorphic expansions for parenthesised braids can thus be inferred from the existence of Drinfeld associators. 

In \cite{Drin90}, Drinfeld introduced two groups,  the Grothendieck--Teichm\"uller group $\gt$ and its graded version $\grt_1$, which act freely and transitively on the set of associators. Bar-Natan's homomorphic expansions induce isomorphisms between the prounipotent completion of parenthesised braids, $\widehat{\PaB}$, and their associated graded parenthesised chord diagrams, $\PaCD$.  As such, Bar-Natan \cite{BNGT} showed that it is natural to identify these symmetry groups with automorphisms of the source and target so that \[\gt\cong \Aut(\widehat{\PaB}) \quad \text{and} \quad \grt_1 \cong \Aut(\PaCD),\] which act freely and transitively on homomorphic expansions by pre- and post-composition respectively. Later, several authors noticed that the collections of parenthesised braids and parenthesised chord diagrams could be concisely described as operads, and that homomorphic expansions, and therefore Drinfeld associators, can be identified with the set of operad isomorphisms $Z:\widehat{\PaB}\rightarrow \PaCD$ (\cite{tamarkin1998proof}, \cite{FresseVol1}).

\medskip 
A higher-dimensional version of Bar-Natan's story arises in the Lie theoretic context of the Kashiwara--Vergne problem. Informally, the Kashiwara--Vergne conjecture asks when, in the non-commutative setting, one can simplify the exponentiation rule \[e^xe^y=e^{\mathfrak{bch}(x,y)}.\] Here, we write \[\mathfrak{bch}(x,y)=x+y+\frac{1}{2}[x,y] +\frac{1}{12}[x,[x,y]]+\dots\] for the Baker--Campbell--Hausdorff series. While the Baker--Campbell--Hausdorff series provides a formula for $e^xe^y$ as an infinite Lie series, this is not always useful in applications. The original form of the Kashiwara--Vergne (KV) conjecture \cite{KV78} asks if the Baker-Campbell-Hausdorff series can be expressed in terms of a convergent power series and adjoint endomorphisms.  A solution to the KV conjecture can be formulated as an automorphism $F$ of the (degree completed) free Lie algebra $\lie_2$ with generators $x$ and $y$, such that $F$ satisfies the {\em first KV equation} \[F(e^x e^y)=e^{x+y}, \]  as well as several other properties which we omit here (complete details in Section~\ref{sec: KRV}). 

The existence of general solutions to the KV conjecture was shown in 2006 \cite{AM06} and a deep relationship between the KV conjecture and Drinfeld associators was established in \cite{AT12, ATE10}.  As in the case of Drinfeld associators, there exist symmetry groups, called the Kashiwara--Vergne groups and denoted $\kv$ and $\krv$, which act freely and transitively on the set of solutions to the KV equations.

As with Drinfeld associators, the set of KV solutions has a topological description. \emph{Welded foams}, heretofore called $w$-foams, are a class of knotted surfaces in $\mathbb{R}^4$ \cite{Satoh}. They have much in common with parenthesised braids, but live in a higher dimension. In a series of papers \cite{BND:WKO1, BND:WKO2}, Bar-Natan and the first author show that homomorphic expansions of $w$-foams are in bijection with the solutions to the KV conjecture~\cite[Theorem 4.9]{BND:WKO2}. The existence of a homomorphic expansion for $w$-foams can therefore be deduced from the existence of solutions to the KV conjecture. In this paper, we build on this topological interpretation of the KV conjecture to give a simultaneously topological and ``operadic'' interpretation of the Kashiwara--Vergne groups. 

A key feature of $w$-foams is that they are finitely presented as a \emph{circuit algebra} with additional cabling-type operations called {\em unzips}. Circuit algebras, reviewed in Section~\ref{sec: ca}, are a generalisation of Jones's planar algebras \cite{Jones:PA}, in which one drops the planarity condition on ``connection diagrams''. We describe circuit algebras as algebras over the coloured operad of \emph{wiring diagrams} in Definition~\ref{def:CA}.   Alternatively, in \cite[Theorem 5.5]{DHR1} we showed that circuit algebras are equivalent to a type of tensor category called a {\em wheeled prop}. The preliminary sections of this paper describe the circuit algebra of $w$-foams, denoted $\wf$, as well as its associated graded circuit algebra $\arrows$, where homomorphic expansions take values. $\arrows$ admits a combinatorial description in terms of \emph{arrow diagrams} (Definition~\ref{def: arrow diagram circuit alg}): an oriented version of the better known space of chord diagrams. The first main theorem of this paper is the following (Theorem~\ref{thm: aut(A)=KRV}):

\begin{thm*}
There is an isomorphism of groups $\Aut_{v}(\arrows)\cong \krv$. 
\end{thm*}

In order to develop the corresponding theorem for the group $\kv$, we construct a prounipotent completion of the circuit algebra of $w$-foams. We show explicitly that completion intertwines with the circuit algebra operations, as well as unzips and the further auxiliary operations. In this way, we avoid a full development of the rational homotopy theory of circuit algebras which arises in the operadic approach to the Grothendieck-Teichm\"uller group (e.g. Fresse~\cite{FresseVol2}).  The decision to avoid rational homotopy theory for circuit algebras is due to the fact that $w$-foams have a richer structure than that of just a circuit algebra: the most notable additional structure is embodied by the unzip operations. Full details are in Sections~\ref{sec:wf} and \ref{sec:completion}. 

After taking the prounipotent completion of $w$-foams, we show that homomorphic expansions correspond to isomorphisms of circuit algebras (Theorem~\ref{thm: expansions are isomorphisms}).
\begin{thm*}
For any homomorphic expansion $Z:\wf\xrightarrow{}\A$, the induced map $\widehat{Z}: \hatwf \to \A$ is an isomorphism of filtered, complete circuit algebras. 
\end{thm*}  This is a new contribution to the literature on homomorphic expansions of $w$-foams, where completions so far have not been discussed. Current work-in-progress by Bar-Natan~\cite{BN:Complete} further explores the interactions between prounipotent completion and homomorphic expansions for groups. A direct consequence of this result is that we can identify KV solutions with isomorphisms of circuit algebras.  Another consequence of Theorem~\ref{thm: expansions are isomorphisms} is that we can describe the symmetry group $\kv$ as automorphisms of (completed) $w$-foams by combining a (non-canonical) isomorphism $\widehat{Z}$ and the isomorphism $\Aut_v(\arrows)\cong \krv$ in Theorem~\ref{thm: ZGZ is in KV}:
\begin{thm*} 
There is an isomorphism of groups $\Aut_{v}(\hatwf)\cong \kv$. 
\end{thm*}

These results complete the $4$-dimensional topological interpretation of the Kashiwara--Vergne conjecture initiated by Bar-Natan and the first author. At the same time, combining these results with the wheeled prop description of circuit algebras \cite{DHR1}, we have given an ``operadic'' interpretation of the Kashiwara-Vergne groups. Namely, solutions of the Kashiwara--Vergne conjecture give rise to isomorphisms of completed wheeled props, and the symmetry groups $\kv$ an $\krv$ are identified with automorphism groups of completed wheeled props. 

\medskip 

The parallels between Drinfeld associators and Kashiwara--Vergne solutions are more than a coincidence.  In a series of breakthrough articles Alekseev, Enriquez and Torossian \cite{AT12, ATE10} show that each Drinfeld associator $\Phi$ (\cite{Drin90, Drin89}) gives rise to a KV solution $F_{\Phi}$. Conversely, a KV solution $F$ gives rise to a ``KV associator'' $\Phi_F$.  The main distinction is that KV associators live in a different space than Drinfeld associators: KV associators are automorphisms of free Lie algebras and Drinfeld associators are automorphisms of Lie algebras of infinitesimal braids.

This close relationship extends to a relationship between the graded Grothendieck--Teichm\"uller group $\grt_1$ and the graded Kashiwara--Vergne group $\krv$.  Alekseev and Torossian construct a group homomorphism $\varrho:\grt_1\rightarrow\krv$ and conjecture that $\krv \cong \grt_1\times \mathbb{Q}$ \cite[Remark 9.14]{AT12}.  It is natural to wish to interpret this relationship in the topological and operadic context.  In Section~\ref{sec:krvgrt}, we describe the relationship between the operad of parenthesised chord diagrams $\PaCD$ and the circuit algebra of arrow diagrams $\arrows$.  Because the algebraic structures are different-- $\PaCD$ is an operad and $\arrows$ a circuit algebra-- the relationship is subtle. Nonetheless, we construct an image of $\PaCD$ in $\arrows$, and exhibit $\grt_1$ as automorphisms of arrow diagrams in Theorem~\ref{thm: GRT as Aut(A)}. 

\begin{remark}
The topological interpretation of the analogous map $\varrho:\gt_1\rightarrow \kv$ which maps the prounipotent radical of the Grothendieck--Teichm\"uller group into the Kashiwara--Vergne group $\kv$ (\cite{ATE10}) is not included in this paper, but is the topic of separate, future work. 
\end{remark}

\subsection*{Acknowledgements} 
The authors would like to thank Dror Bar-Natan, who has been generous with his insights throughout the writing of this paper, and contributed important ideas to several proofs.
We also thank Anton Alekseev, Tamara Hogan, Arun Ram, and Chris Rogers for their interest, suggestions and helpful mathematical discussions. The first and third author would like to thank the Mathematical Sciences Research Institute (MSRI) for their support via the 2020 program "Higher Categories and Categorification" and the Sydney Mathematical Research Institute (SMRI) for providing visitor funding to the third author, and a constructive working environment where much of this paper was written.

\section{Preliminaries} 

\subsection{Circuit algebras}\label{sec: ca}
A \emph{circuit algebra} is an algebraic structure analogous to Jones's planar algebras ~\cite{Jones:PA} used to describe virtual and welded tangles in low-dimensional topology. Informally, an oriented circuit algebra is a bi-graded sequence of sets or vector spaces, together with a family of operations parametrised by \emph{wiring diagrams} $\WD$. More formally, the collection of all wiring diagrams forms a coloured operad and circuit algebras are algebras over this operad. Here we present the basic details needed in this paper, and refer the reader to \cite{DHR1} for a complete introduction and an alternate, equivalent, description of circuit algebras as wheeled props. 

\medskip

Throughout this section, let $\calI$ denote a countable alphabet, the set of labels. 
\begin{definition}\label{def:WD}  
An \emph{oriented wiring diagram} is a triple $\WD=(\calL,M,f)$ consisting of:
\begin{enumerate}
\item 


A set of sets of labels
$$\mathcal{L}=\{L_0^{-}, L_0^{+}; L_1^{-},L_1^{+},\hdots,L_r^{-},L_r^{+}\}\subseteq \calI.$$   The elements of the set $L^{-}_i$ are referred to as the $i$th set of \emph{outgoing labels} and the elements of $L^{+}_i$ are the $i$th set of \emph{incoming labels}.  The sets $L_0^{-}$ and $L_{0}^{+}$ play a distinguished role: their elements are called the \emph{output labels} of the diagram $\WD$, while $L_i^+$ and $L_i^-$ for $i>0$ are called {\em input labels}.  We write $L_{i}^{\pm}$ to mean  ``$L_{i}^{-}$ and $L_{i}^{+}$, respectively'' and we may refer to the set $L_i^{\pm}$ as the \emph{$i$th label set} of $\WD$. 
			
\item An abstract, oriented, compact $1$-manifold $M$ with boundary, $\partial M$, regarded up to orientation-preserving homeomorphism. We write $\partial M^{-}$ for the set of {\em beginning} boundary points of $M$, and  $\partial M^{+}$ for the set of {\em ending}  boundary points, so $\partial M = \partial M^{-} \sqcup \partial M^{+}$.
			
\item Set bijections\footnote{If the sets $\{L_i^{\pm}\}$ are not pairwise disjoint, replace the unions $(\cup_{i=0}^r L_i^{-})$ and  $(\cup_{i=0}^r L_i^{+})$ by the set of triples $\{(a,i,\pm)\, | \, a\in L_i^{\pm}, 0\leq i \leq r\}$.} \label{fn:Triples}
$$\partial M^{-} \xrightarrow{f} \cup^r_{i=0} {L^{-}_i} \quad \textrm{and}\quad  \partial M^{+} \xrightarrow{f} \cup^r_{i=0} {L^{+}_i}.$$
\end{enumerate}
\end{definition}

Wiring diagrams have a convenient pictorial representation as tangle diagrams, as in Figure~\ref{fig:WiringDiagramOri}. In this representation, we arrange the elements of $\calL$ on the boundary of a disc with $r$ disjoint holes $$D_0 \setminus (\mathring{D}_1 \sqcup \mathring{D}_2 \sqcup  \hdots \sqcup \mathring{D}_r).$$ The set $L_0^{\pm}$ gives a set of distinguished points on the boundary of the big disc $D_0$ and each $L_i^{\pm}$ represents a set of distinguished points on the boundary of the $i$th interior disc.  The abstract $1$-manifold $M$ is pictured as an immersed $1$-manifold, pairing up the input and output labels from $\calL$.  While this pictorial representation is convenient, and we will use it throughout the paper, we emphasise that wiring diagrams are \emph{combinatorial}, not topological objects. In particular, the chosen immersion of the $1$-manifold $M$ depicted in Figure~\ref{fig:WiringDiagramOri} is not part of the wiring diagram data. See~\cite[Proposition 2.3]{DHR1} for more details.  

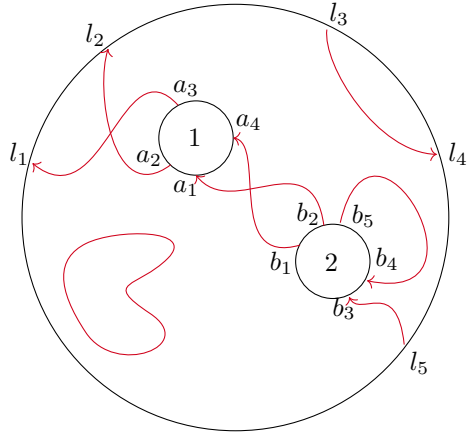
\begin{figure}[h]
	\begin{tikzpicture}[x=0.75pt,y=0.75pt,yscale=-.75,xscale=.75]

\draw   (100,151.5) .. controls (100,72.25) and (164.25,8) .. (243.5,8) .. controls (322.75,8) and (387,72.25) .. (387,151.5) .. controls (387,230.75) and (322.75,295) .. (243.5,295) .. controls (164.25,295) and (100,230.75) .. (100,151.5) -- cycle ;
\draw   (192,98) .. controls (192,84.19) and (203.19,73) .. (217,73) .. controls (230.81,73) and (242,84.19) .. (242,98) .. controls (242,111.81) and (230.81,123) .. (217,123) .. controls (203.19,123) and (192,111.81) .. (192,98) -- cycle ;
\draw   (284,181) .. controls (284,167.19) and (295.19,156) .. (309,156) .. controls (322.81,156) and (334,167.19) .. (334,181) .. controls (334,194.81) and (322.81,206) .. (309,206) .. controls (295.19,206) and (284,194.81) .. (284,181) -- cycle ;
\draw  [color={rgb, 255:red, 208; green, 2; blue, 27 }  ,draw opacity=1 ] (143,169) .. controls (163,159) and (218,159) .. (198,179) .. controls (178,199) and (153,194) .. (186,216) .. controls (219,238) and (163,259) .. (143,229) .. controls (123,199) and (123,179) .. (143,169) -- cycle ;
\draw[<-] [color={rgb, 255:red, 208; green, 2; blue, 27 }  ,draw opacity=1 ]   (158,38) .. controls (149,64) and (160,146) .. (200,116) ;
\draw[<-] [color={rgb, 255:red, 208; green, 2; blue, 27 }  ,draw opacity=1 ]   (107,115) .. controls (155,154) and (167,34) .. (205,76) ;
\draw[->] [color={rgb, 255:red, 208; green, 2; blue, 27 }  ,draw opacity=1 ]   (303,157) .. controls (295,98) and (242,155) .. (217,123) ;
\draw[->] [color={rgb, 255:red, 208; green, 2; blue, 27 }  ,draw opacity=1 ]   (287,170) .. controls (236,185) and (265,101) .. (242,98) ;
\draw [->][color={rgb, 255:red, 208; green, 2; blue, 27 }  ,draw opacity=1 ]   (357,237) .. controls (352,190) and (327,220) .. (320,205) ;
\draw [<-] [color={rgb, 255:red, 208; green, 2; blue, 27 }  ,draw opacity=1 ]   (332,194) .. controls (422,213) and (336,60) .. (314,155) ;
\draw [->][color={rgb, 255:red, 208; green, 2; blue, 27 }  ,draw opacity=1 ]   (305,25) .. controls (302,49) and (339,108) .. (379,109) ;

\draw (90,99.4) node [anchor=north west][inner sep=0.75pt]    {$l_{1}$};
\draw (141,19.4) node [anchor=north west][inner sep=0.75pt]    {$l_{2}$};
\draw (175,105) node [anchor=north west][inner sep=0.75pt]    {$a_{2}$};
\draw (200,58) node [anchor=north west][inner sep=0.75pt]    {$a_{3}$};
\draw (242,80.4) node [anchor=north west][inner sep=0.75pt]    {$a_{4}$};
\draw (200,125) node [anchor=north west][inner sep=0.75pt]    {$a_{1}$};
\draw (265,173) node [anchor=north west][inner sep=0.75pt]    {$b_{1}$};
\draw (283,140) node [anchor=north west][inner sep=0.75pt]    {$b_{2}$};
\draw (320,141.4) node [anchor=north west][inner sep=0.75pt]    {$b_{5}$};
\draw (336,170) node [anchor=north west][inner sep=0.75pt]    {$b_{4}$};
\draw (307,203.4) node [anchor=north west][inner sep=0.75pt]    {$b_{3}$};
\draw (305,5.4) node [anchor=north west][inner sep=0.75pt]    {$l_{3}$};
\draw (385,99.4) node [anchor=north west][inner sep=0.75pt]    {$l_{4}$};
\draw (359,240.4) node [anchor=north west][inner sep=0.75pt]    {$l_{5}$};
\draw (210,90.4) node [anchor=north west][inner sep=0.75pt]    {$1$};
\draw (302,174.4) node [anchor=north west][inner sep=0.75pt]    {$2$};

\end{tikzpicture}
	\caption{An example of an oriented wiring diagram. The labels sets are $L_0^{-}=\{l_1, l_2,l_4\}$, $L_0^{+}=\{l_3,l_5\}$, $L_1^-=\{a_2, a_3\}$, $L_1^+=\{a_1,a_4\}$, $L_2^-=\{b_1,b_2,b_5\}$, $L_2^+=\{b_3,b_4\}$. The manifold, drawn in red lines, is a disjoint union of seven oriented intervals and one oriented circle. We don't draw arrows on circles, since they are abstract, not embedded: there is only one homeomorphism type of an oriented circle.}\label{fig:WiringDiagramOri}
	\end{figure}

Wiring diagrams assemble into a \emph{coloured operad}. We briefly recall that for a fixed set of colours, $\mathfrak{C}$, a \emph{$\mathfrak{C}$-coloured operad} $\mathsf{P}=\{\mathsf{P}(c_0; c_1,\ldots,c_r)\}$ consists of a collection of sets $\mathsf{P}(c_0;c_1,\ldots,c_r)$: one for each sequence $c_0; c_1,\ldots,c_r$ of colours in $\mathfrak{C}$. This is equipped with an $\calS_r$--action permuting $c_1,\ldots,c_r$, together with an equivariant, associative and unital family of partial compositions
\[\begin{tikzcd}\circ_i:\mathsf{P}(c_0;c_1,\ldots,c_r)\times\mathsf{P}(d_0;d_1,\ldots,d_s)\arrow[r]& \mathsf{P}(c_0;c_1,\ldots,c_{i-1}, d_1,\ldots, d_s, c_{i+1},\ldots, c_r) \end{tikzcd},\] whenever $d_0=c_i$. For full details see \cite[Definition 1.1]{bm_resolutions}. 


We denote the set of all wiring diagrams of type $(L^{\pm}_0; L^{\pm}_1,\ldots,L^{\pm}_r)$ by $\mathsf{WD}(L^{\pm}_0;L^{\pm}_1,\ldots,L^{\pm}_r)$. The set $\mathsf{WD}(L^{\pm}_0;L^{\pm}_1,\ldots,L^{\pm}_r)$ has a natural action by $\calS_r$ permuting the input label sets, so that for any $\sigma\in\calS_r$ \[\begin{tikzcd} \mathsf{WD}(L^{\pm}_0;L^{\pm}_1,\ldots,L^{\pm}_r)\arrow[r, "\sigma^{-1}"] & \mathsf{WD}(L^{\pm}_0;L^{\pm}_{\sigma(1)},\ldots,L^{\pm}_{\sigma(r)}). \end{tikzcd}\]

\begin{definition}\label{def: WD}
The collection $\mathsf{WD} =\{\mathsf{WD}(L^{\pm}_0; L^{\pm}_1,\ldots,L^{\pm}_r)\}$ forms a discrete coloured operad called the \emph{operad of oriented wiring diagrams} with partial compositions as follows. For $\WD=(\calL,M,f)$ and $\WD'=(\calL',N,g)$, if $L_i^{-}=L_0'^{+}$ and  $L_i^{+}=L_0'^{-}$ then $\WD\circ_i \WD'$ is defined by the label set $\{L_1^\pm,...,L_{i-1}^\pm, L_1'^{\pm},...,L_s'^{\pm},L_{i+1}^\pm...L_r^\pm\}$, the manifold $M\sqcup N/ \sim $ obtained from $M$ and $N$ by gluing along the boundary identification $L_i^{-}=L_0'^{+}$ and  $L_i^{+}=L_0'^{-}$, and the set bijection $f \sqcup g /\sim $ induced by $f$ and $f'$ in the natural way.

\end{definition} 

Pictorially, wiring diagram composition shrinks the wiring diagram $\WD'$ and glues it into the $i$th input circle of $\WD$ in such a way that the labels match and the boundary points of the $1$-manifolds are identified: see Figure~\ref{fig:WiringDiagramComp}. Composition in the operad $\mathsf{WD}$ is similar to the operad of planar tangles in \cite{Jones:PA} and \cite[Definition  2.1]{HPT16}. 

\begin{figure}[h]
	\includegraphics[height=.3\linewidth]{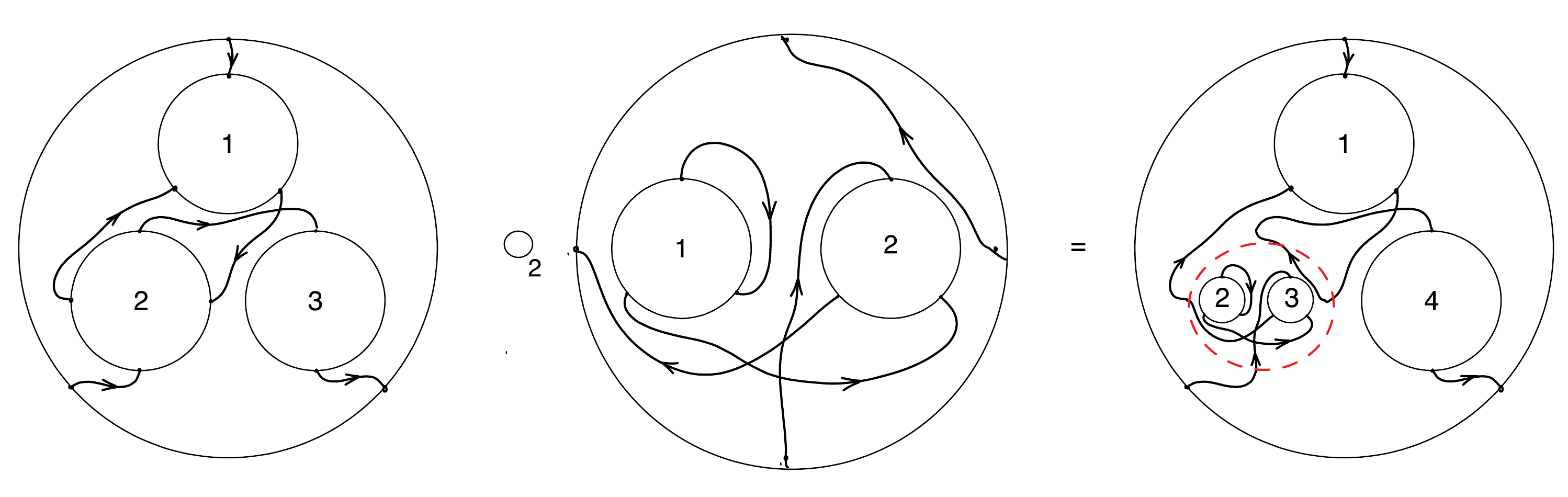}
	\caption{An example of the partial composition of oriented wiring diagrams.}\label{fig:WiringDiagramComp}
	\end{figure}

An \emph{oriented circuit algebra} is an algebra over the operad of wiring diagrams. 
Unwinding the definition of an algebra over a coloured operad (Definition 1.2 \cite{bm_resolutions}), we obtain the following definition:
	
\begin{definition}\label{def:CA}
An \emph{oriented circuit algebra} in sets  is a collection of sets $\V=\{\V[L^{-};L^{+}]\}$, where $L^{-}, L^{+}$ run over all pairs of label sets in $\calI$, together with a family of multiplication functions parametrised by oriented wiring diagrams. Namely, for each wiring diagram $\WD=(\mathcal{L},M,f)$, there is a corresponding function 
		$$F_\WD:\V[L^{-}_1;L^{+}_1] \times \hdots \times \V[L^{-}_r;L^{+}_{r}] \rightarrow \V[L^{+}_0;L^{-}_0].$$   This data must satisfy the following axioms: 
		
\begin{enumerate}

\item The assignment $\WD \mapsto F_\WD$ is compatible with wiring diagram composition in the following sense.  Let \[\WD=(\{L^{\pm}_0,\hdots,L^{\pm}_r\},M,f), \quad \WD'=(\{L'^{\pm}_0,\hdots,L'^{\pm}_s\},N,g)\] be two wiring diagrams composable as $\WD\circ_i \WD'$. Then the map corresponding to the composition $\WD\circ_i \WD'$ is 
			\[ F_{\WD\circ_{i}\WD'}= F_\WD\circ (\Id\times \dots \times \Id \times F_{\WD'}\times \Id \times \dots\times \Id), \]
			where $F_{\WD'}$ is inserted in the $i$th component.
			\item There is an action of the symmetric groups on wiring diagrams which permutes the input sets (that is, the input indices $i=1,...,r$). The maps $F_\WD$ are equivariant in the following sense. Let $\WD=(\{L^{\pm}_0, L^{\pm}_1,\hdots,L^{\pm}_{r}\},M,f)$ be a wiring diagram, $\sigma\in \mathcal S_r$, and let
			$\sigma \WD= (\{L^{\pm}_0, L^{\pm}_{\sigma^{-1}(1)},\hdots,L^{\pm}_{\sigma^{-1}(r)}\},M, f)$  be the wiring diagram $\WD$ with the input sets reordered; note that the output set $L^{\pm}_0$ is fixed.
			Then
			$F_{\sigma  \WD}= F_\WD\circ \sigma^{-1}$, where $\sigma^{-1}$ acts on  
			$\V[L^{-}_{\sigma(1)};L^{+}_{\sigma(1)}] \times \hdots \times \V[L^{-}_{\sigma(r)};L^{+}_{\sigma(r)}]$ by permuting the factors. 
	\end{enumerate}
\end{definition}

\begin{definition}A \emph{homomorphism} of circuit algebras $\Phi:\mathsf{V}\rightarrow\mathsf{W}$ is a family of maps 
	$\{\Phi_{L^{-};L^+}:\mathsf{V}[L^-;L^+]\rightarrow \mathsf{W}[L^-;L^+]\}_{L^{-}, L^{+} \subseteq \calI}$
	which commutes with the action of wiring diagrams. That is, for any wiring diagram $\WD=(\mathcal{L}, M, f)$ we have a commutative diagram:
	\[\xymatrixcolsep{5pc}\xymatrix{
	\V[L^{-}_1;L^{+}_1] \times \hdots \times \V[L^{-}_r;L^{+}_{r}] \ar[r]^-{(F_\V)_\WD} \ar[d]_{\Phi_{L_1^-;L_1^+}\times ... \times \Phi_{L_r^-;L_r^+} }& \V[L^{+}_0;L^{-}_0] \ar[d]^{\Phi_{L_0^+;L_0^-}} \\
	\W[L^{-}_1;L^{+}_1] \times \hdots \times \W[L^{-}_r;L^{+}_{r}] \ar[r]^-{(F_\W)_\WD} & \W[L^{+}_0;L^{-}_0]
	}
	\]
The category of all circuit algebras is denoted $\mathsf{CA}$. 
\end{definition}

\begin{remark}
One can define a circuit algebra in any closed, symmetric monoidal category $\mathcal{E}=(\mathcal{E},\otimes, 1)$. 
In Section~\ref{sec:completion} and in \cite{DHR1} we consider circuit algebras in $\mathbb{Q}$-vector spaces. In this case, the action maps $$F_D:\V[L^{-}_1;L^{+}_1] \otimes \hdots \otimes \V[L^{-}_r;L^{+}_{r}] \rightarrow \V[L^{+}_0;L^{-}_0]$$ are linear maps of $\mathbb{Q}$-vector spaces. 

\end{remark} 
		
\begin{remark}\label{rmk:presentation}
The main examples of circuit algebras we use in this paper will be defined using presentation notation $\mathsf{CA}\left< g_1,\ldots,g_n\mid r_1,\ldots, r_m\right>$ where the $g_i$ are generators and the $r_j$ are relations. The generators are elements living in a specified $\mathsf{V}[L^-; L^+]$ and the relations generate circuit algebra ideals by which one quotients the free circuit algebra generated by $g_1,\ldots, g_n$.
\end{remark}
 
\begin{notation}\label{notation: wiring diagram composition}
We write $\WD(p_1,\ldots, p_r)$ to refer to the image of a sequence of elements $p_1,\ldots, p_r$ under the composition map $$F_{\WD}:\V[L^{-}_1;L^{+}_1] \times \hdots \times \V[L^{-}_r;L^{+}_{r}] \rightarrow \V[L^{+}_0;L^{-}_0]$$ for a fixed wiring diagram $\WD$ (e.g. Example~\ref{example: w-tangle}). 
\end{notation}

\begin{example}\label{ex: symmetric group}
Any free circuit algebra always contains all of the wiring diagrams with no input label sets. In other words, the arity $0$ operations of the operad of wiring diagrams $\mathsf{WD}(L^\pm_0;-)$ are elements in every free circuit algebra. 
 
We single out the subset $\mathsf{WD}(n)\subseteq \mathsf{WD}(L^\pm_0;-)$ consisting of those wiring diagrams with no input discs and $L_{0}^{+}=L_{0}^{-}=\{1,2,...,n\}$. Elements of the set $\mathsf{WD}(n)$ are elements of the symmetric group on $n$ letters $\calS_n$ together with a nonnegative integer (the number of circle components of $M$). This gives a bijection of sets \[\begin{tikzcd} \mathsf{WD}(n)\arrow[r]& \arrow[l]\calS_n \times \mathbb{Z}_{\geq0}.\end{tikzcd}\] 
\end{example} 

\begin{remark} \label{rmk: CA are wheeled props} 
Circuit algebras are equivalent to a type of rigid tensor category called a wheeled prop \cite[Theorem 5.5]{DHR1}. One can interpret the sets $\mathsf{V}[L_i^-;L_i^+]$ as morphisms from objects $L_i^-$ to objects $L_i^+$. In this interpretation, the  generators $g_i$ are generating morphisms of this tensor category and the $1$-manifold depicts a chosen composition of these morphisms. We have chosen to write this paper using the circuit algebra language as our main examples, defined in the next sections, are easier to grapple with in this combinatorial/topological interpretation rather than in their categorical form. 
\end{remark}

\subsection{The circuit algebra of $w$-foams}\label{sec:wf}
Circuit algebras provide a combinatorial model for the main topological object of this paper, $w$-\emph{foams} (Definition~\ref{def: w-foams}). Topologically, $w$-foams are a class of tangled tubes with singular vertices in $\mathbb{R}^4$ equipped with a {\em ribbon filling}. Their simpler cousins, welded tangles, or \emph{$w$-tangles}, are described by a Reidemeister theory obtained by generalising the classical Reidemeister theory of tangles (\cite[Section 5]{BN05}) by replacing the planar algebra structure of classical tangles with a circuit algebra structure, and imposing an additional relation called ``Overcrossings Commute'' (Figure~\ref{fig:Reidemeister}). 
The topological description of $w$-tangles and $w$-foams plays a minimal role in this paper, but we recommend  \cite[Section 3.2]{DHR1}, \cite[Sections 3.4 and and 4.1]{BND:WKO2}, and \cite{Satoh} for the reader interested in the topological background.  


\begin{definition}\label{def:w-tangles}
The circuit algebra of \emph{$w$-tangles} is given by the presentation \[\mathsf{wT}= \textsf{CA}\left<\overcrossing, \undercrossing \mid R1^s, R2, R3, OC\right>\] where the relations $R1^s$, $R2$, $R3$ and $OC$ are pictured in Figure~\ref{fig:Reidemeister}. 
\end{definition} 

The circuit algebra $\mathsf{wT}$ is a discrete circuit algebra, where elements of the set $\mathsf{wT}[L^{-}_i;L^{+}_i]$ are $w$-tangles with $|L^{-}_i|=|L^{+}_i|$ (open) strands. 

\begin{remark} While this is not crucial for this paper, we mention that
topologically, under Satoh's tubing map \cite{Satoh}, the positive crossing $\overcrossing$ represents an interaction between two oriented tubes in $\mathbb{R}^4$, which can be described as a ``movie'' in $\mathbb R^3$ in which two horizontal circles switch places by one flying through the other. The circle flying through is represented by the under-strand. See \cite[Section 3.4]{BND:WKO2} for further details on the topology.
\end{remark}

\begin{figure}[h]
\centering
\begin{subfigure}{.5\textwidth}
  \centering
\includegraphics[height=4cm]{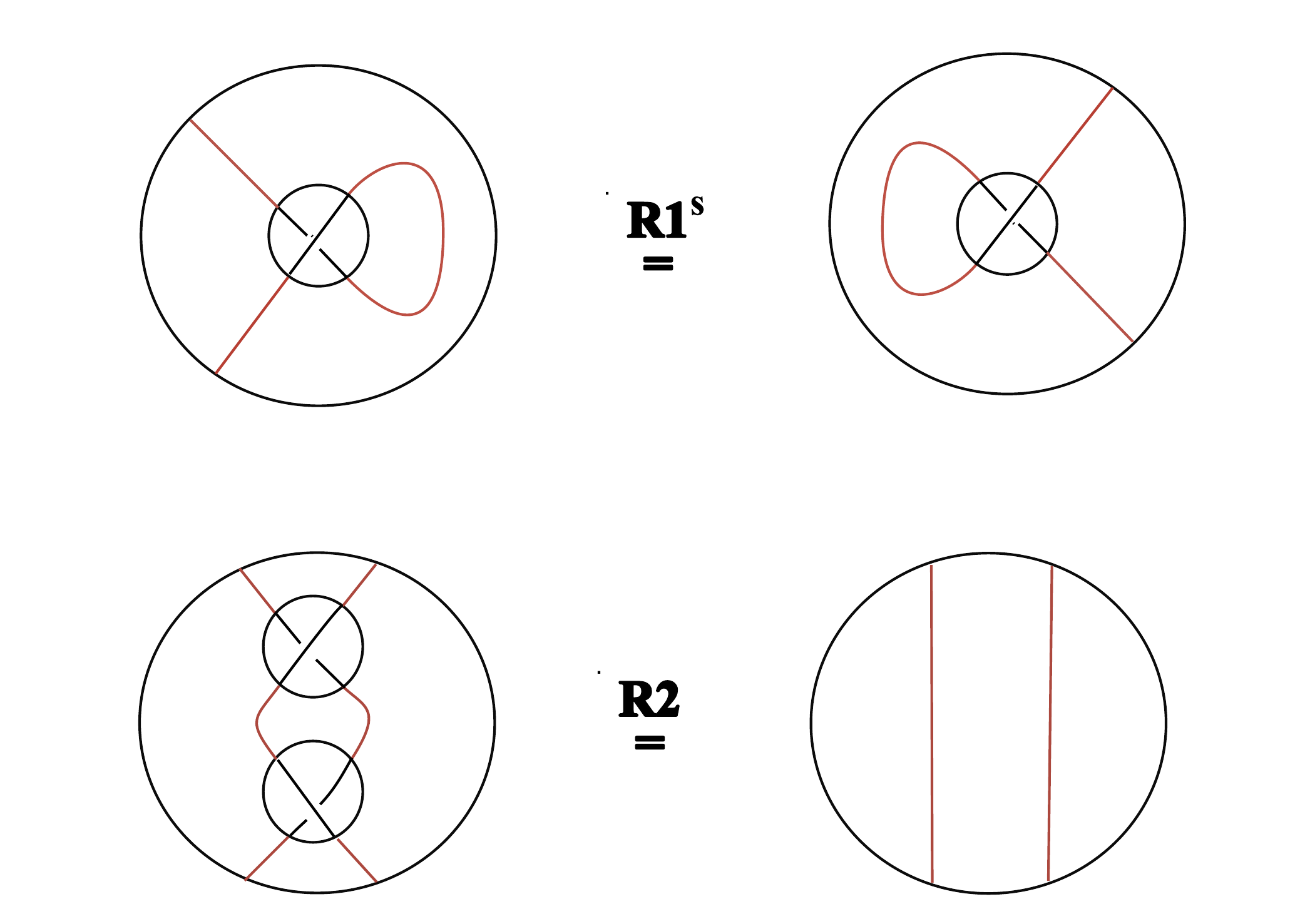}
\end{subfigure}%
\begin{subfigure}{.5\textwidth}
\centering
\includegraphics[height=4cm]{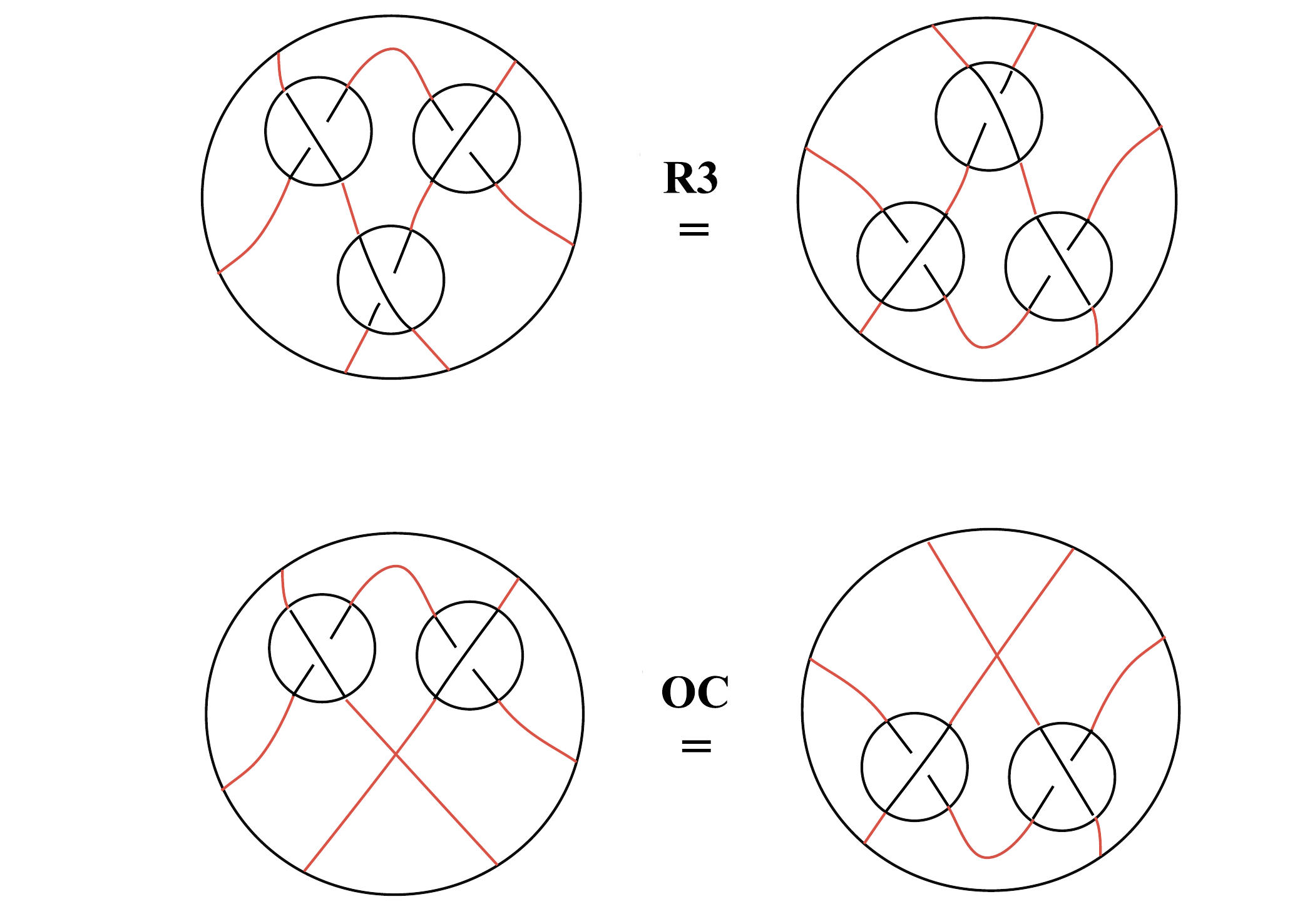}
\end{subfigure}
\caption{The classical Reidemeister moves and ``Overcrossings Commute'', presented as circuit algebra relations between the crossings (generators). The relations are imposed in all possible (consistent) strand orientations.}\label{fig:Reidemeister}
\end{figure}

\begin{example}\label{example: w-tangle}
Figure~\ref{tangle} shows a $w$-tangle with three strands as an element of the circuit algebra $\mathsf{wT}$.  This tangle is given by a wiring diagram composition of the generators $\undercrossing$ and $\overcrossing$, where the abstract $1$-manifold of the wiring diagram $\WD$ is depicted in red. Explicitly, this composition is 
\[\begin{tikzcd} \WD: \mathsf{wT}\times \mathsf{wT} \arrow[r] & \mathsf{wT} \\
(\undercrossing, \overcrossing) \arrow[r, mapsto] & \WD(\undercrossing, \overcrossing).
\end{tikzcd}\]

\begin{figure}[h!]
\includegraphics[scale=.15]{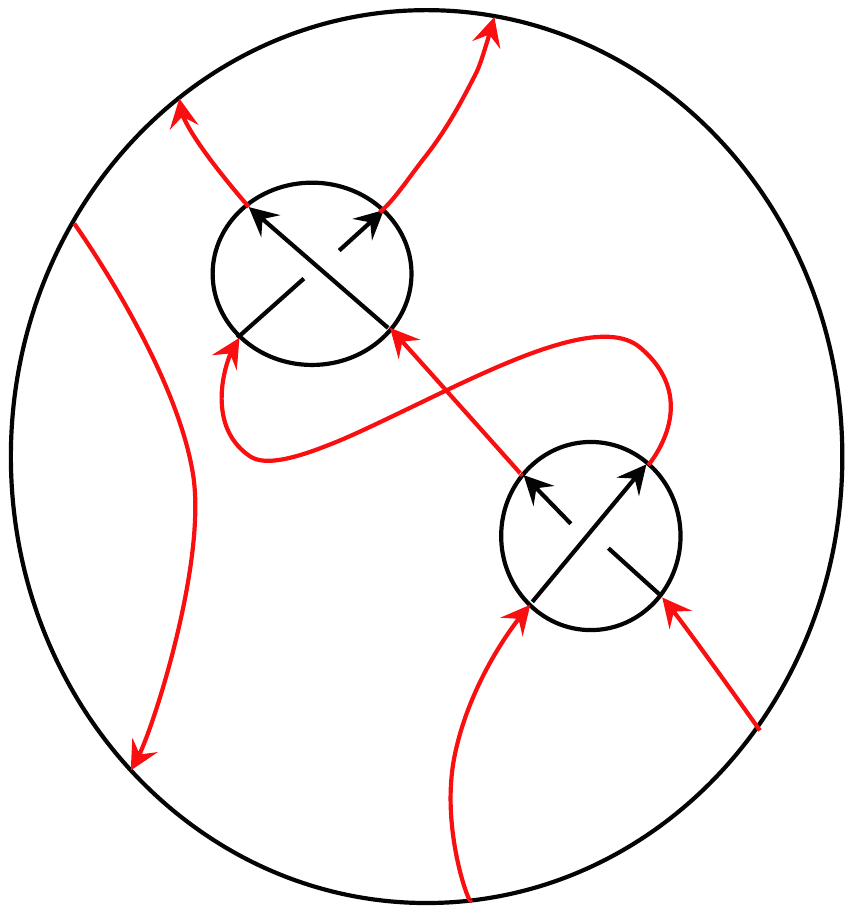}
\caption{An example of a $w$-tangle presented as an element $\WD(\undercrossing, \overcrossing)$ of the circuit algebra $\mathsf{wT}$.}\label{tangle}
\end{figure}
\end{example} 

We note that the additional crossing $\virtualcrossing$ in Example~\ref{example: w-tangle} arises only as a feature of the wiring diagram $\WD$. Topologically, this represents two tubes passing around one another in $\mathbb{R}^{4}$ without interaction (i.e. with disjoint fillings). Such crossings are often called \emph{virtual crossings} in the literature (\cite{BND:WKO2},\cite{Kup},\cite{DK:Virtual}, and more) but we emphasise that these crossings are not generators of the circuit algebra. We make this more precise in the following example. 

\begin{example}\label{ex: w-tangle skeleton}
The set of all $w$-tangles with $n$ strands (and possibly some circle components), $\mathsf{wT}[L_0^{-};L_0^{+}]$ where $|L^{-}_{0}|=|L^{+}_{0}|=n$, is fibered over the symmetric group. Recall from Example~\ref{ex: symmetric group} that the wiring diagrams $\mathsf{WD}(n)$ which have no input discs are in bijection with $\calS_n\times \mathbb{Z}_{\geq 0}$. There is a natural projection map \[\begin{tikzcd}\pi: \mathsf{wT}[L_0^{-};L_0^{+}]\arrow[r]& \mathsf{WD}(n)\end{tikzcd}\] which sends the generators $\overcrossing$ and $\undercrossing$ to the transposition $\virtualcrossing$ in $\mathsf{WD}(2)$. As we can identify $\mathsf{WD}(n)$ with the set of elements of $\calS_n\times \mathbb{Z}_{\geq 0}$, we say that every $w$-tangle has an underlying permutation, and write $\pi^{-1}(\sigma \times \mathbb{Z}_{\geq 0})$ for all $w$-tangles whose underlying permutation is $\sigma\in\calS_n$. 
\end{example} 

The circuit algebra of $w$-\emph{foams} is an extension of $w$-tangles in which we add \emph{foamed vertices}  and \emph{capped strands} to $w$-tangles.  
Topologically, a capped strand is the closure of an oriented tube by gluing a $2$-disc to one end. In the Reidemeister theory, we denote capped strands by $\upcap$. 

A foamed vertex, diagrammatically, is a trivalent vertex with a total ordering of the incident edges, denoted $\vertex$. The first edge in the total ordering is the top (blue) edge, and edges are ordered counterclockwise. Topologically, a vertex is a singular surface in $\mathbb R^4$, which is easiest to visualise as a movie in $\mathbb R^3$ in which a circle flies inside another, they merge, and become a single circle. The ``merged'' circle corresponds to the top edge. Those interested in the topological details may read \cite[Section 4.1.1]{BND:WKO2} for details. We allow all possible orientations of foamed vertices. The symbol $\vertex$ stands for the vertex all of whose edges are oriented upwards, and all seven other vertices can be obtained from this via orientation switch operations (as in Definition~\ref{def:w-foam operations} below). For more detail see \cite[Figure 16]{BND:WKO2}.
\begin{definition}\label{def: w-foams}
The circuit algebra of \emph{$w$-foams} is given by\footnote{To generate strictly as a circuit algebra, technically all orientations of caps and vertices are needed. To generate as a circuit algebra with auxiliary operations, which includes orientation switches, the upward oriented cap and vertex suffice.} the presentation 
\[ \wf =\mathsf{CA}\left< \overcrossing, \undercrossing, \vertex, \upcap \mid R1^s, R2, R3, R4, OC, CP\right>.\]
Moreover, $\wf$ is equipped with the following auxiliary operations:  orientation switches $S_e, A_e$, unzips $u_e$, and strand deletion $d_e$ described in Definition~\ref{def:w-foam operations} below.
\end{definition}

\begin{figure}[h]
\centering
\begin{subfigure}{.5\textwidth}
  \centering
\includegraphics[height=4cm]{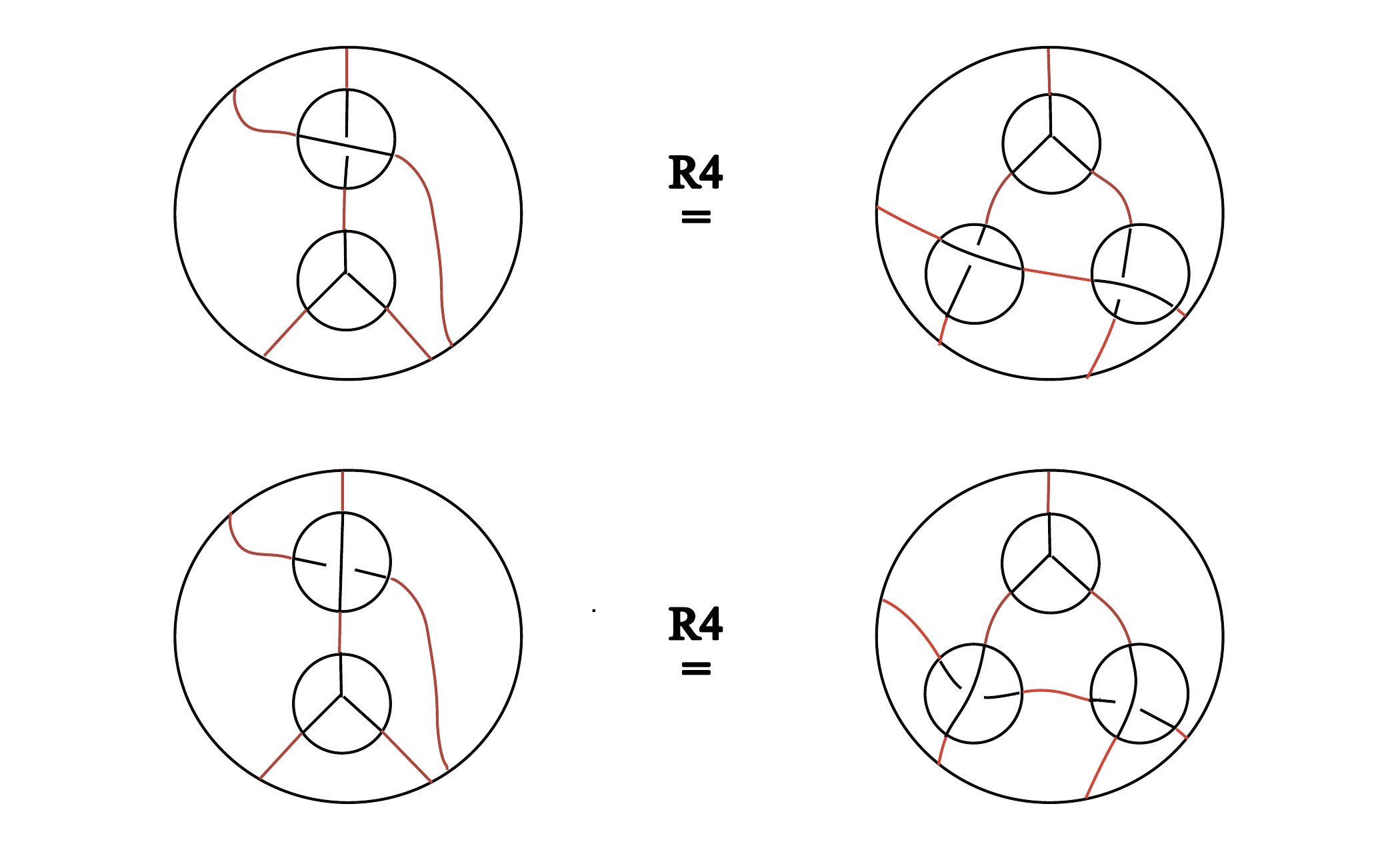}
\end{subfigure}%
\begin{subfigure}{.5\textwidth}
\centering
\includegraphics[height=4cm]{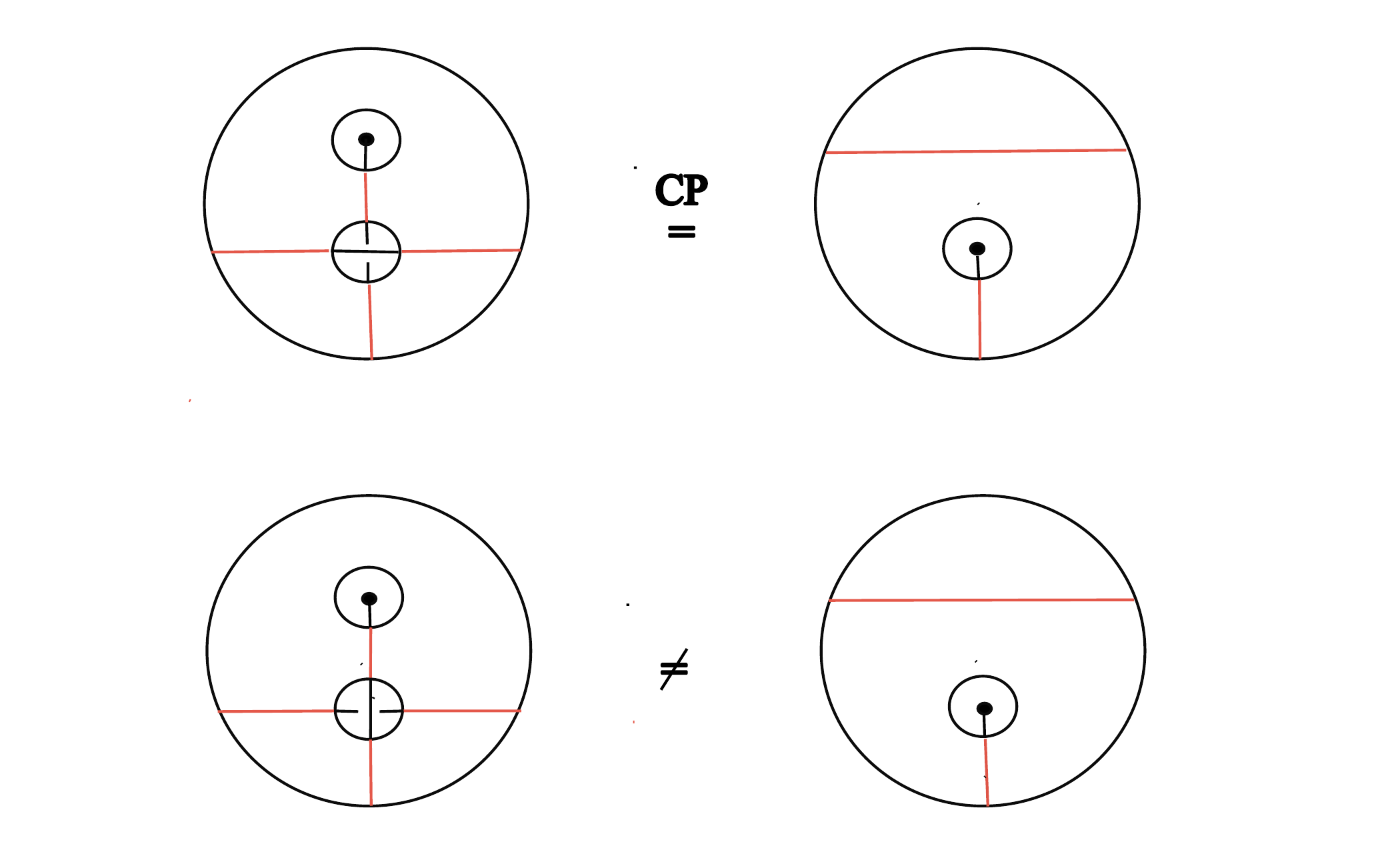}
\end{subfigure}
\caption{The additional relations on $\wf$; note the asymmetry of the CP relation.}\label{fig:R4 and CP}
\end{figure} 

As with $w$-tangles, $\wf$ is a discrete circuit algebra where elements of the set $\mathsf{\wf}[L^{-}_i;L^{+}_i]$ are $w$-foams. Notice that in $\wf$ the cardinality of $L^{-}_i$ may not be the same as $L^{+}_i$, as vertices allow for strands to merge and split. An example of a $w$-foam is presented on the left in Figure~\ref{skeleton_map_foams}.

The \textit{auxiliary operations} on $\wf$ are external to the structure of a circuit algebra, which is to say that these are operations that are not parametrised by the operad of wiring diagrams. The full topological explanation of these operations can be found in \cite[Section 4.1.3]{BND:WKO2}.

\begin{definition}\label{def:w-foam operations} The circuit algebra $\wf$ is equipped with the following auxiliary operations:
\begin{enumerate}
	\item \textit{Orientation switch $S_e$}: Diagrammatically, orientation switch reverses the direction of the strand $e$. Topologically, this operation switches both the 1D {\em direction} and the 2D {\em orientation} of the tube of the strand $e$. 
	\item \textit{Adjoint $A_e$}: Diagrammatically, the adjoint operation reverses the direction of the strand $e$ and conjugates each crossing $e$ passes {\em over} by virtual crossings. Topologically, it reverses only the 1D direction of a tube $e$, but not the 2D orientation of the surface. 
	\item \textit{Unzip, and disc unzip $u_e$}: Diagrammatically, unzip doubles the strand $e$ between two foam vertices using the blackboard framing, then attaches the ends of the doubled strands to the corresponding ends from the foam vertices, as in Figure~\ref{fig: unzip for foams}. A similar operation for capped strands is \textit{disc unzip}, also illustrated in Figure~\ref{fig: unzip for foams}. Topologically, this operation doubles a tube in the framing direction. 
	\item \textit{Deletion $d_e$}: It deletes the strand $e$,  as long as $e$ is not attached to foam vertices on either end (these are called ``long strands'').
\end{enumerate}
\end{definition} 
	

The circuit algebra of $w$-foams is an example of a circuit algebra with a \emph{skeleton}. In the context of Remark~\ref{rmk: CA are wheeled props}, a circuit algebra with a skeleton is a wheeled prop for which the set of objects also forms a wheeled prop.  

\begin{definition}\label{def: skeleta circuit alg}
	
The circuit algebra of $w$-foam \emph{skeleta}  \[\mathcal{S}=\mathsf{CA}\left<\upcap, \vertex\right>\] is the free circuit algebra generated by the oriented \emph{caps} $\upcap$ and \emph{vertices} $\vertex$. Moreover, this circuit algebra also possesses auxiliary operations: $S_e, A_e, u_e$ and $d_e$. 
\end{definition} 

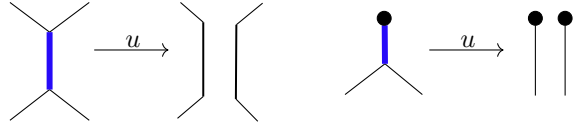
\begin{figure}

\begin{tikzpicture}[x=0.75pt,y=0.75pt,yscale=-.5,xscale=.5]

\draw    (60,32) -- (100,62) ;
\draw    (140,32) -- (100,62) ;
\draw    (140,151) -- (100,119.95) ;
\draw    (60,151) -- (100,119.95) ;
\draw [color={rgb, 255:red, 38; green, 10; blue, 248 }  ,draw opacity=1 ][line width=2.25]    (100,62) -- (100,119.95) ;
\draw [->]    (145,80) -- (220,80) ;
\draw    (232,32) -- (255,52) ;
\draw    (229,149) -- (255,126.95) ;
\draw [color={rgb, 255:red, 0; green, 0; blue, 0 }  ,draw opacity=1 ][line width=0.75]    (255,52) -- (255,126.95) ;
\draw    (310,152) -- (287.9,129.44) ;
\draw    (314.89,32.71) -- (288.38,54.5) ;
\draw [color={rgb, 255:red, 0; green, 0; blue, 0 }  ,draw opacity=1 ][line width=0.75]    (287.9,129.44) -- (288.38,54.5) ;
\draw    (476,125) -- (438.28,93.95) ;
\draw    (398,125) -- (438.28,93.95) ;
\draw [color={rgb, 255:red, 38; green, 10; blue, 248 }  ,draw opacity=1 ][line width=2.25]    (438,55) -- (438.28,93.95) ;
\draw  [fill={rgb, 255:red, 0; green, 0; blue, 0 }  ,fill opacity=1 ] (430.5,48) .. controls (430.5,44.13) and (433.63,41) .. (437.5,41) .. controls (441.37,41) and (444.5,44.13) .. (444.5,48) .. controls (444.5,51.87) and (441.37,55) .. (437.5,55) .. controls (433.63,55) and (430.5,51.87) .. (430.5,48) -- cycle ;
\draw[->]    (483,80) -- (558,80) ;
\draw    (590,55) -- (590,126) ;
\draw  [fill={rgb, 255:red, 0; green, 0; blue, 0 }  ,fill opacity=1 ] (583,48) .. controls (583,44.13) and (586.13,41) .. (590,41) .. controls (593.87,41) and (597,44.13) .. (597,48) .. controls (597,51.87) and (593.87,55) .. (590,55) .. controls (586.13,55) and (583,51.87) .. (583,48) -- cycle ;
\draw    (620,55) -- (620,126) ;
\draw  [fill={rgb, 255:red, 0; green, 0; blue, 0 }  ,fill opacity=1 ] (613,48) .. controls (613,44.13) and (616.13,41) .. (620,41) .. controls (623.87,41) and (627,44.13) .. (627,48) .. controls (627,51.87) and (623.87,55) .. (620,55) .. controls (616.13,55) and (613,51.87) .. (613,48) -- cycle ;

\draw (174,63) node [anchor=north west][inner sep=0.75pt]    {$u$};
\draw (512,63) node [anchor=north west][inner sep=0.75pt]    {$u$};

\end{tikzpicture}
\caption{The unzip operation doubles a tube edge ending in the distinguished edges of two foamed vertices. Disc unzip doubles a capped strand which ends as a distinguished strand of a vertex.}\label{fig: unzip for foams}
\end{figure}

An example of an element of $\mathcal{S}$ is depicted on the right in Figure~\ref{skeleton_map_foams}. We saw in Example~\ref{ex: symmetric group} that since $\calS$ is a free circuit algebra, it contains every element of $\mathsf{WD}(L^{\pm}_0;-)$.  In particular, $\calS$ contains all of the symmetric groups. As with $w$-tangles, we can project the circuit algebra of $w$-foams to the underlying circuit algebra of skeleta: 

\begin{definition}\label{def: skeletal projection}
For any pair of label sets $L^{-},L^{+}$ we define a \emph{skeleton} projection map \[\begin{tikzcd}\pi: \wf[L^{-};L^{+}]\arrow[r]& \calS[L^{-};L^{+}],\end{tikzcd}\] which ``flattens'' the generators $\overcrossing$ and $\undercrossing$ to $\virtualcrossing$. We write $\wf(s) =\pi^{-1}(s)$ for the subset of foams $\wf[L^{-};L^{+}]$ with skeleton $s\in\calS$. 
\end{definition} 

\begin{example} 
In Figure~\ref{skeleton_map_foams}, the circle on the left depicts a $w$-foam as an element of the circuit algebra $\wf$, given as a wiring diagram composition $\WD(\undercrossing, \vertex)$. The black circles represent the input discs of the wiring diagram $\WD$, in which we have drawn the generators $\undercrossing$ and $\vertex$.  These generators are composed via the (abstract) $1$-manifold in $\WD$, which is depicted in red. The picture on the right of Figure~\ref{skeleton_map_foams} is the projection of this foam onto its skeleton, which is an element of the circuit algebra $\calS$. The generator $\undercrossing$ is sent to $\virtualcrossing$ in the skeleton, which forms part of the wiring diagram. 

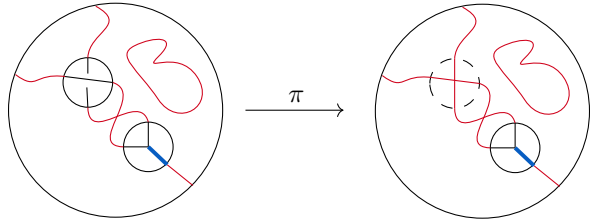
\begin{figure}
\begin{tikzpicture}[x=0.75pt,y=0.75pt,yscale=-.5,xscale=.5]

\draw   (21,141) .. controls (21,81.35) and (69.35,33) .. (129,33) .. controls (188.65,33) and (237,81.35) .. (237,141) .. controls (237,200.65) and (188.65,249) .. (129,249) .. controls (69.35,249) and (21,200.65) .. (21,141) -- cycle ;
\draw   (137,178) .. controls (137,164.19) and (148.19,153) .. (162,153) .. controls (175.81,153) and (187,164.19) .. (187,178) .. controls (187,191.81) and (175.81,203) .. (162,203) .. controls (148.19,203) and (137,191.81) .. (137,178) -- cycle ;
\draw   (76,113) .. controls (76,99.19) and (87.19,88) .. (101,88) .. controls (114.81,88) and (126,99.19) .. (126,113) .. controls (126,126.81) and (114.81,138) .. (101,138) .. controls (87.19,138) and (76,126.81) .. (76,113) -- cycle ;
\draw [color={rgb, 255:red, 208; green, 2; blue, 27 }  ,draw opacity=1 ]   (126,113) .. controls (160,124) and (98,177) .. (137,178) ;
\draw [color={rgb, 255:red, 208; green, 2; blue, 27 }  ,draw opacity=1 ]   (101,138) .. controls (102,178) and (153,120) .. (162,153) ;
\draw    (100,118) -- (101,138) ;
\draw    (78,107) -- (126,113) ;
\draw    (101,88) -- (101,104) ;
\draw [color={rgb, 255:red, 208; green, 2; blue, 27 }  ,draw opacity=1 ]   (28,105) .. controls (51,122) and (46,103) .. (78,107) ;
\draw [color={rgb, 255:red, 208; green, 2; blue, 27 }  ,draw opacity=1 ]   (101,88) .. controls (102,62) and (143,55) .. (109,36) ;
\draw [color={rgb, 255:red, 208; green, 2; blue, 27 }  ,draw opacity=1 ]   (181,196) -- (207,217) ;
\draw  [color={rgb, 255:red, 208; green, 2; blue, 27 }  ,draw opacity=1 ] (148,76) .. controls (168,66) and (193,66) .. (173,86) .. controls (153,106) and (191,77) .. (211,107) .. controls (231,137) and (192,160) .. (172,130) .. controls (152,100) and (128,86) .. (148,76) -- cycle ;
\draw    (137,178) -- (162,178) ;
\draw    (162,153) -- (162,178) ;
\draw [color={rgb, 255:red, 7; green, 93; blue, 194 }  ,draw opacity=1 ][line width=1.5]    (162,178) -- (181,196) ;
\draw   (391,142) .. controls (391,82.35) and (439.35,34) .. (499,34) .. controls (558.65,34) and (607,82.35) .. (607,142) .. controls (607,201.65) and (558.65,250) .. (499,250) .. controls (439.35,250) and (391,201.65) .. (391,142) -- cycle ;
\draw   (507,179) .. controls (507,165.19) and (518.19,154) .. (532,154) .. controls (545.81,154) and (557,165.19) .. (557,179) .. controls (557,192.81) and (545.81,204) .. (532,204) .. controls (518.19,204) and (507,192.81) .. (507,179) -- cycle ;
\draw  [dash pattern={on 4.5pt off 4.5pt}] (446,114) .. controls (446,100.19) and (457.19,89) .. (471,89) .. controls (484.81,89) and (496,100.19) .. (496,114) .. controls (496,127.81) and (484.81,139) .. (471,139) .. controls (457.19,139) and (446,127.81) .. (446,114) -- cycle ;
\draw [color={rgb, 255:red, 208; green, 2; blue, 27 }  ,draw opacity=1 ]   (496,114) .. controls (530,125) and (468,178) .. (507,179) ;
\draw [color={rgb, 255:red, 208; green, 2; blue, 27 }  ,draw opacity=1 ]   (471,139) .. controls (472,179) and (523,121) .. (532,154) ;
\draw [color={rgb, 255:red, 208; green, 2; blue, 27 }  ,draw opacity=1 ]   (448,108) -- (496,114) ;
\draw [color={rgb, 255:red, 208; green, 2; blue, 27 }  ,draw opacity=1 ]   (471,89) -- (471,139) ;
\draw [color={rgb, 255:red, 208; green, 2; blue, 27 }  ,draw opacity=1 ]   (398,106) .. controls (421,123) and (416,104) .. (448,108) ;
\draw [color={rgb, 255:red, 208; green, 2; blue, 27 }  ,draw opacity=1 ]   (471,89) .. controls (472,63) and (513,56) .. (479,37) ;
\draw [color={rgb, 255:red, 208; green, 2; blue, 27 }  ,draw opacity=1 ]   (551,197) -- (577,218) ;
\draw  [color={rgb, 255:red, 208; green, 2; blue, 27 }  ,draw opacity=1 ] (518,77) .. controls (538,67) and (563,67) .. (543,87) .. controls (523,107) and (561,78) .. (581,108) .. controls (601,138) and (562,161) .. (542,131) .. controls (522,101) and (498,87) .. (518,77) -- cycle ;
\draw    (507,179) -- (532,179) ;
\draw    (532,154) -- (532,179) ;
\draw [color={rgb, 255:red, 7; green, 93; blue, 194 }  ,draw opacity=1 ][line width=1.5]    (532,179) -- (551,197) ;
\draw [->]   (260,141) -- (359,141) ;

\draw (301,118.4) node [anchor=north west][inner sep=0.75pt]    {$\pi $};

\end{tikzpicture}

\caption{A projection of a $w$-foam to its skeleton.}\label{skeleton_map_foams}
\end{figure}

\end{example} 

\begin{prop}
The skeleton projection maps assemble to give a homomorphism of circuit algebras, meaning that the following diagram commutes for all wiring diagrams $\WD$:
\[\begin{tikzcd}  \wf(s_1)\times \ldots \times \wf(s_r)\arrow[r, "\WD"]\arrow[d, "\pi \times\ldots\times \pi", swap]& \wf(s)\arrow[d, "\pi"] \\ (s_1,\ldots,s_r)\in\calS^{r}\arrow[r, mapsto, "\WD"]& s=\WD(s_1,\ldots, s_r)
\end{tikzcd}\] 
Skeleton projections also commute with all auxiliary operations.
\end{prop}
\begin{proof}
To verify this, observe that since the map $\pi$ is defined on generators, it is a circuit algebra map if it respects the relations of $\wf$. As one example, consider the relation R$2$ (Figure~\ref{fig:Reidemeister}). The projection $\pi$ sends both sides of R$2$ to the identity element in the symmetric group $\calS_2$ and thus $\pi$ preserves the relation R$2$. The remainder of the relations are equally straightforward, and so is the commutativity with auxiliary operations. We leave these to the reader to check. 
\end{proof}

As the projection map is a homomorphism of circuit algebras, it follows that we can use the skeleton to provide an indexing of the circuit algebra of $w$-foams and write \[\wf:= \coprod_{s\in\calS}\wf(s).\]   Note, however, that for a specific $s\in\calS$ the set $\wf(s)$ is \emph{not} a circuit subalgebra of $\wf$.

\subsection{Completion of $\mathbf{w}$-foams}\label{sec:completion} 

In this section, we construct a prounipotent \emph{completion} of the circuit algebra $\wf$. This is largely formal: circuit algebras are algebras over the operad of wiring diagrams (Definition~\ref{def:CA}) and thus completion of a circuit algebra is the completion of an algebra over an operad similar to that in \cite[1.4.2]{FresseLie}. The main difference between this section and \cite[1.4.2]{FresseLie} is that we index the circuit algebra $\wf$ by skeleta as opposed to the arity of the operations. 

\medskip 
\begin{definition}
An \emph{ideal} $\calI$ of a linear circuit algebra $\mathsf{V}=\{\mathsf{V}[L^-;L^+]\}$ is a collection $$\calI=\{\calI[L^-;L^+]\subseteq \mathsf{V}[L^-;L^+]\} \quad \text{where}$$ 
\[\WD(p_1,\ldots, p_r,k)\in\calI[L_0^+;L_0^-] \quad \text{whenever} \quad p_i\in \mathsf{V}[L_i^-;L_i^+], 1\leq i\leq r, \quad \text{and}\quad k\in \calI[L_j^-;L_j^+]. \]
Equivalently, $\mathcal I$ is an ideal if the action of the operad $\mathsf{WD}$ descends to a circuit algebra structure on the quotient $\mathsf{V}/\calI$.

The \emph{$n$th power} ideal $\mathcal{I}^{n}$ consists of operations $\WD(p_1,\ldots,p_r)$ in which at least $n$ of the $p_i$, $1\leq i\leq r$, are in $\mathcal{I}$.  

A quotient circuit algebra $\mathsf{V}/\calI^{n}$ is said to be \emph{nilpotent} if the circuit algebra multiplications \[\begin{tikzcd}F_{\WD}:(\mathsf{V}/\calI^n)[L_1^-;L_1^+]\otimes \ldots \otimes (\mathsf{V}/\calI^n)[L_r^-;L_r^+] \arrow[r] & (\mathsf{V}/\calI^n)[L_0^+;L_0^-]\end{tikzcd} \] vanish for all wiring diagrams $\WD$ with $r>R$ input discs, for some $R$. 
\end{definition}

\begin{definition}\label{def:completion}
A circuit algebra $\mathsf{V}$ is \emph{complete} if $\mathsf{V}=\lim_{n} \mathsf{V}/\calI^n$, where $\calI^n$, $n\geq 1$ is a descending sequence of ideals of $\mathsf{V}$ and each $\mathsf{V}/\calI^n$ is nilpotent.
\end{definition} 

To complete the circuit algebra $\wf$, we first linearly extend $\wf$ to a circuit algebra in $\mathbb{Q}$-vector spaces. For each skeleton $s\in\calS$, let $\mathbb{Q}[\wf](s)$ denote the $\mathbb{Q}$-vector space of formal linear combinations of $w$-foams $T_i$ with skeleton $s\in\calS$, 
\[
\sum_{T_i\in \wf(s)}\alpha_iT_i \in \mathbb{Q}[\wf](s).
\]   
The collection of vector spaces $$\mathbb{Q} [\wf] = \bigsqcup_{s\in \mathcal{S}} \mathbb{Q} [\wf](s)$$ forms a circuit algebra where the operad of wiring diagrams acts by the linear extension of the action on $\wf$. In particular, $\mathbb{Q}[\wf]$ is a linear circuit algebra with skeleton $\calS$.

\begin{definition}\label{def: augmentation ideal} 
At each skeleton $s\in\calS$ we define an \emph{augmentation map}
\begin{align*}
\epsilon_{s}:\mathbb{Q}[\wf](s)\ \longrightarrow \ \mathbb{Q}\\ 
\Sigma\alpha_iT_i \mapsto \Sigma\alpha_i.
\end{align*} We denote the kernel of $\epsilon_{s}$ by $\mathcal{I}(s)$. 
The \emph{augmentation ideal} of $\mathbb{Q}[\wf]$ is the disjoint union $$\mathcal{I}= \bigsqcup_{s\in\mathcal{S}} \mathcal{I}(s).$$ 
\end{definition}

\begin{lemma}\label{lem:IandOps} $\calI$ is an ideal in $\mathbb{Q}[\wf]$. 
\end{lemma}

\begin{proof}
Let $\WD$ be a wiring diagram composing the elements $p_{1},...,p_{r}, k$, where $p_i\in \mathbb{Q}[\wf](s_i)$, $1\leq i \leq r$, and $k\in\calI(s_k)$. We then verify that the composite $\WD(p_1,\ldots,p_r,k)$ is in $\calI$.

Because a circuit algebra is an algebra over an operad, circuit algebra composition is associative and equivariant. It follows that any circuit algebra composition $\WD(p_1,\ldots,p_r,k)$ can be equivalently written $\WD_1(\WD_2(p_1,\ldots p_{r-1},k), p_r).$ Therefore, without loss of generality we may assume that $\WD$ has only two inputs.

Given elements $p\in\mathbb{Q}[\wf](s)$ and $k\in\calI (s_k)$ with 
\[
p = \sum_{j} \alpha_{j} T_{j} \quad \text{and} \quad k=\sum_{i} \lambda_{i} K_{i}, \quad \text{ with } T_j\in\wf(s), \text{ and } K_i\in\wf(s_k),
\] 

the composite via the wiring diagram $\WD$ is by definition 
\[
\WD(p,k)= \sum_{j}\alpha_j\sum_{i}\lambda_i \;\; \WD(T_j, K_i).
\] 
Since $k$ is in $\calI(s_k)$ we know that $\sum_{i} \lambda_{i}=0$.  It follows that $\sum_{j}\alpha_j\sum_{i}\lambda_i =0$ and therefore the composition $\WD(p,k)$ is in $\calI$. The lemma now follows.
\end{proof}

Since $\mathcal I$ is an ideal, $\mathbb{Q}[\wf]/\mathcal I$ is a circuit algebra. The next lemma implies that the auxiliary operations are also well-defined on the quotient:

\begin{lemma} 
For any $k\in \calI$, the w-foams resulting from auxiliary operations $u_e(k)$, $S_e(k)$, $A_e(k)$ and $d_e(k)$ are in $\calI$ whenever these operations are defined. 
\end{lemma}

\begin{proof}
Let $p\in\mathbb{Q}[\wf](s)$ be an element with $p = \sum_{T_j\in\wf(s)} \alpha_{j} T_{j}$.  For a strand $e$ in the skeleton $s\in\calS$, the unzip operation produces $u_e(p)\in\mathbb{Q}[\wf](u_e(s))$, which is by definition the element \[u_e(p) = \sum_{u_e(T_j)\in\wf(u_e(s))} \alpha_{j} u_e(T_{j}).\] In this case, the coefficients remain unchanged, so if $\sum_{T_j}\alpha_j =0$ and $p\in \calI$, this still holds for $u_e(p)$. The same argument works for the operations $S_e$, $A_e$ and $d_e$. 
\end{proof}

For each $s\in\calS$, the $\mathbb{Q}$-vector space $\mathbb{Q}[\wf](s)$ admits a descending filtration given by powers of the augmentation ideal: 
\[\mathbb{Q}[\wf](s)  \supset \calI(s) \supset \calI^2(s)\supset\ldots \supset \calI^n(s)\supset\ldots.\] Since circuit algebra composition and the auxiliary operations are compatible with this filtration, $\mathbb{Q}[\wf]$ is a circuit algebra in \emph{filtered} $\mathbb{Q}$-vector spaces. 

\begin{definition}\label{def:completion}
The (prounipotent) {\em completion} of the $\mathbb{Q}$-vector space $\mathbb{Q}[\wf](s)$ is the inverse limit of the system $$ \mathbb{Q}[\wf](s)/\mathcal{I}(s)\stackrel{}{\leftarrow} \mathbb{Q}[\wf](s)/\mathcal{I}^{2}(s) \stackrel{}{\leftarrow} \mathbb{Q}[\wf](s)/\mathcal{I}^{3}(s) \stackrel{}{\leftarrow}...$$ We denote the resulting completed $\mathbb{Q}$-vector space by $\hatwf(s) = \lim_n \mathbb{Q}[\wf](s)/\calI^n(s)$. 
\end{definition}

The passage from filtered $\mathbb{Q}$-vector spaces to completed $\mathbb{Q}$-vector spaces extends to a lax symmetric monoidal functor (\cite[Proposition 7.3.11; Section 7.3.12]{FresseVol1}) \[\begin{tikzcd} \widehat{(-)}:\operatorname{fVect}(\mathbb{Q}) \arrow[r] & \widehat{\operatorname{fVect}}(\mathbb{Q}).\end{tikzcd}\] Algebras over coloured operads transfer over symmetric monoidal functors, which in this case means that, for every wiring diagram $\WD$, the following diagram commutes: 

\[\begin{tikzcd} 
\mathbb{Q}[\wf](s_1)\otimes\ldots\otimes \mathbb{Q}[\wf](s_r) \arrow[r, "\widehat{(-)}"]\arrow[d, swap, "F_\WD"] &\widehat{(\mathbb{Q}[\wf](s_1)\otimes\ldots\otimes \mathbb{Q}[\wf](s_r))}\arrow[r,"\cong"] & \hatwf(s_1)\otimes\ldots\otimes \hatwf(s_r) \arrow[d, "F_\WD"]
\\
\mathbb{Q}[\wf](s) \arrow[rr, "\widehat{(-)}"] && \hatwf(s).
\end{tikzcd}\] The following proposition follows immediately. 

\begin{prop}
The completion of $w$-foams $$\hatwf :=\coprod_{s\in\calS}\hatwf(s)$$ is a circuit algebra with auxiliary operations $A_e$, $S_e$, $u_e$ and $d_e$. \qed
\end{prop}

\begin{definition}\label{def:gr}
For every skeleton $s\in\calS$, the filtration of $\mathbb{Q}[\wf](s)$ by powers of the augmentation ideal gives rise to the (complete) associated graded vector space 
$$\A(s)= \prod_{n\geq0} \mathcal{I}^{n}(s)/\mathcal{I}^{n+1}(s).$$ 
\end{definition} 

	The functor from filtered $\mathbb{Q}$-vector spaces to completed $\mathbb{Q}$-vector spaces restricts to a lax symmetric monoidal functor from filtered $\mathbb{Q}$-vector spaces to graded $\mathbb{Q}$-vector spaces~\cite[Lemma 7.3.10; Section 7.3.13]{FresseVol1}.  It follows from the same argument as above that the graded $\mathbb{Q}$-vector spaces $\A(s)$ assemble into a circuit algebra $$\A:=\coprod_{s\in\mathcal{S}} \A(s).$$  Moreover, the filtration is compatible with the auxiliary operations of $\wf$, and therefore $\A$ is a circuit algebra with the associated graded auxiliary operations, also denoted $A_e$, $S_e$, $u_e$ and $d_e$.  The following proposition shows that completion of the circuit algebra $\wf$ is filtration-preserving and that $\A$ is the \emph{associated graded space} of $\widehat{\wf}$ also. 

\begin{prop}\label{prop:AssocGradedOfCompletion}
The circuit algebra $\hatwf$ is filtered by $\widehat{\mathcal{I}^{n}}:= \displaystyle{\lim_{k>n} \mathcal{I}^{n}/\mathcal{I}^{k}}$. Moreover, with this filtration it is a complete circuit algebra and there is a canonical isomorphism of circuit algebras $\text{gr}( \hatwf )\cong \A$.
\end{prop}

\begin{proof}

Our first goal is to show that $\hatwf$ is filtered. We know that the $\mathbb{Q}$-module $\mathbb{Q}[\wf](s)$ has a filtration by powers of the augmentation ideal and that the quotient maps $\mathbb{Q}[\wf](s)\rightarrow \mathbb{Q}[\wf](s)/\calI^{\ell}(s)$ lifts canonically to a morphism $$\mathbb{Q}[\wf](s)\rightarrow \hatwf(s)=\lim_{n}\mathbb{Q}[\wf](s)/\calI^n(s).$$ This induces a canonical filtration on $\hatwf(s)$ with $F_\ell(\hatwf(s))=\text{ker}\left(\hatwf(s)\rightarrow \mathbb{Q}[\wf](s)/\calI^{\ell+1}(s)\right)$ which are the kernels of the projection of maps on the right-hand tower  
\[\begin{tikzcd} F_{\ell}(\hatwf(s)) \arrow[d, "\subseteq"]\arrow[r, hook]&\hatwf(s) \arrow[d] \arrow[r]\arrow[d, "="] & \mathbb{Q}[\wf](s)/\calI^{\ell + 1} (s)\arrow[d, two heads] \\ F_{\ell-1}(\hatwf(s))\arrow[r, hook]&\hatwf(s) \arrow[r] & \mathbb{Q}[\wf](s)/\calI^{\ell }(s) \end{tikzcd} .\]    
What is more, we have $F_\ell(\hatwf(s))=\widehat{\calI}^\ell(s)=  \displaystyle{\lim_{k>\ell} \mathcal{I}(s)^{\ell}/\mathcal{I}(s)^{k}}$ by \cite[7.3.6]{FresseVol1} and \[F_{\ell}(\hatwf(s))/F_{\ell+1}(\hatwf(s)) \cong \calI^{\ell}(s)/\calI^{\ell+1}(s).\]

In a similar fashion, we define the associated graded as \[\A^{\ell}(s) = \text{ker}\left(\mathbb{Q}[\wf](s)/\calI^{\ell + 1}(s)\twoheadrightarrow \mathbb{Q}[\wf](s)/\calI^{\ell}(s)\right).\] By a short calculation, each component of the associated graded $\A$ is identified as \[\A^{\ell}(s)= \calI^{\ell}(s)/\calI^{\ell+1}(s) \cong \widehat{\calI}^{\ell}(s)/\widehat{\calI}^{\ell +1}(s)= \text{gr}(\hatwf(s))^{\ell}.\]  

These filtrations respect the symmetric monoidal structure of filtered $\mathbb{Q}$-modules \cite[7.3.10]{FresseVol1} and thus we can lift this identification to a canonical isomorphism of circuit algebras $\text{gr}( \hatwf )\cong \A$ . 
\end{proof}


\section{Arrow diagrams and the Kashiwara--Vergne Lie algebras}
\subsection{The circuit algebra of arrow diagrams}\label{sec: arrow diagrams} 

Variants of chord diagram spaces -- chord diagrams, Jacobi diagrams, arrow diagrams -- are prevalent throughout knot theory, and in particular, the theory of finite type invariants. They are combinatorial models for the associated graded spaces of filtered linearised spaces of knots, braids, or other knotted objects. Chord diagram spaces also arise in quantum algebra as a computational tool capturing properties of Lie algebras. 

Informally, a Jacobi diagram is a  graph with trivalent and univalent vertices, whose univalent vertices are arranged on a knot or tangle skeleton (e.g. a circle, or $n$ disjoint horizontal lines). These diagrams are considered modulo a set of relations: STU, IHX and anti-symmetry, and others depending on the context. The trivalent vertex captures the algebraic properties of a Lie bracket, for example, the IHX relation corresponds to the Jacobi identity. For more details on Jacobi diagrams in the classical setting see, for example,  \cite{BN_Survey_Knot_Invariants} or \cite[Chapter 5]{chmutov_duzhin_mostovoy_2012} and references therein.  

The space of \emph{$w$-Jacobi diagrams} is an oriented variant of Jacobi diagrams, on $w$-foam skeleta, which gives a combinatorial description for the associated graded space of $w$-foams. They are introduced and studied in \cite[Definition 3.8 and Section 4.2]{BND:WKO2}. Here, we briefly review their definition and most important properties.

\begin{definition} 
A $w$-\emph{Jacobi diagram} on a $w$-foam skeleton $s\in \calS$ is a -- possibly infinite -- linear combination of uni-trivalent oriented graphs with the following properties:
\begin{enumerate}
\item trivalent vertices are incident to two incoming edges and one outgoing edge, and are cyclically ordered
\item univalent vertices are attached to $s$ (this data is combinatorial, i.e. only the ordering of univalent vertices along each skeleton edge matters).
\end{enumerate}
The collection of $w$-Jacobi diagrams on a given $s\in \calS$ forms a graded complete vector space, where the degree is given by half the number of vertices in a diagram (including trivalent and univalent vertices).
\end{definition} 

An example of a $w$-Jacobi diagram is depicted in Figure~\ref{fig: examples of diagrams} where the oriented graph is depicted in red and the $w$-foamed skeleton is in black. 

\begin{figure}[h]

\begin{tikzpicture}[x=0.75pt,y=0.75pt,yscale=-.5,xscale=.5]

\draw [<-] [line width=1.5]    (98,60) -- (98,200) ;
\draw [<-] [line width=1.5]    (140,60) -- (140,200) ;
\draw [<-][line width=1.5]    (179,60) -- (179,200) ;
\draw [->][color={rgb, 255:red, 208; green, 2; blue, 27 }  ,draw opacity=1 ]   (100.69,73.83) -- (119.21,97.43) ;
\draw [->] [color={rgb, 255:red, 208; green, 2; blue, 27 }  ,draw opacity=1 ]   (177.91,83.15) .. controls (126.73,121.28) and (87.15,141.14) .. (138.67,164.03) ;
\draw [->] [color={rgb, 255:red, 208; green, 2; blue, 27 }  ,draw opacity=1 ]   (137.79,78.02) .. controls (125.08,91.3) and (108.17,95.67) .. (129.69,119.21) ;
\draw    (516,170) -- (478.28,132.95) ;
\draw    (438,170) -- (478.28,132.95) ;
\draw [color={rgb, 255:red, 38; green, 10; blue, 248 }  ,draw opacity=1 ][line width=2.25]    (478,94) -- (478.28,132.95) ;
\draw  [fill={rgb, 255:red, 0; green, 0; blue, 0 }  ,fill opacity=1 ] (470.5,87) .. controls (470.5,83.13) and (473.63,80) .. (477.5,80) .. controls (481.37,80) and (484.5,83.13) .. (484.5,87) .. controls (484.5,90.87) and (481.37,94) .. (477.5,94) .. controls (473.63,94) and (470.5,90.87) .. (470.5,87) -- cycle ;
\draw  [color={rgb, 255:red, 208; green, 2; blue, 27 }  ,draw opacity=1 ][fill={rgb, 255:red, 255; green, 255; blue, 255 }  ,fill opacity=1 ][line width=1.5]  (440,145) -- (515,145) -- (515,160) -- (440,160) -- cycle ;
\draw [->][color={rgb, 255:red, 208; green, 2; blue, 27 }  ,draw opacity=1 ]   (140.5,135) -- (178.72,135) ;
\end{tikzpicture}

\caption{On the left is an example of a $w$-Jacobi diagram. Sometimes, we denote an arbitrary w-Jacobi diagram located on a certain part of a skeleton by a red box, as shown on the right. } \label{fig: examples of diagrams} 
\end{figure}
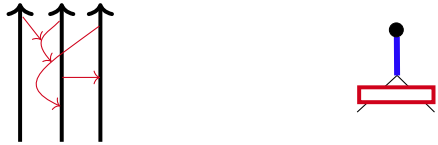 

\begin{definition}
Given a $w$-foam skeleton $s\in\calS$, let $\J(s)$ denote the graded complete $\mathbb{Q}$-vector space of $w$-Jacobi diagrams on the skeleton $s$, modulo the STU, AS, IHX, TC, VI, CP and RI relations shown in Figures~\ref{fig:relations on J}~and~\ref{fig:VICPRI} and briefly explained below. 
\end{definition}

The convention for diagrammatic relations is that only the relevant part of the diagrams is depicted, and the rest of the diagrams is arbitrary, but the same throughout the relation. Importantly, the TC relation states that arrow Tails Commute along skeleton strands. The CP relation means that any arrow diagram with an arrow head immediately adjacent to a cap is zero. The VI relation states that an arrow diagram with a single arrow ending (head or tail) on the distinguished edge of a skeleton vertex is equivalent to the sum of the two arrow diagrams with the arrow ending on each of the two merging strands. This version of the VI relation applies when the non-distinguished strands of the skeleton vertex are both incoming or both outgoing. If their orientations are opposite, a sign is introduced. Note that the VI relation implies that an arrow diagram with $k$ arrows ending (heads or tails) on the distinguished edge of a skeleton vertex is equal to a sum of $2^k$ arrow diagrams where the arrow endings are split between the two non-distinguished edges in all possible ways.

\begin{figure}
	
\tikzset{every picture/.style={line width=0.75pt}} 
	
	\tikzset{every picture/.style={line width=0.75pt}} 

\begin{tikzpicture}[x=0.75pt,y=0.75pt,yscale=-.7,xscale=.7]

\draw [->] [line width=1.5]    (505,100) -- (505,20) ;

\draw [->] [line width=1.5]    (580, 100) -- (580,20) ;
\draw [->] [color={rgb, 255:red, 208; green, 2; blue, 27 }  ,draw opacity=1 ]   (505,48) -- (540,48) ;
\draw [->] [color={rgb, 255:red, 208; green, 2; blue, 27 }  ,draw opacity=1 ]   (505,76) -- (540,76) ;
\draw [->][color={rgb, 255:red, 208; green, 2; blue, 27 }  ,draw opacity=1 ]   (582,76) -- (615,43) ;
\draw [<-][color={rgb, 255:red, 208; green, 2; blue, 27 }  ,draw opacity=1 ]   (615,80) -- (582,43) ;

\draw [-] [color={rgb, 255:red, 208; green, 2; blue, 27 }  ,draw opacity=1 ]   (223.4,80) -- (202.31,58) ;
\draw [-] [color={rgb, 255:red, 208; green, 2; blue, 27 }  ,draw opacity=1 ]   (226,35) -- (202.31,58) ;
\draw[<-] [color={rgb, 255:red, 208; green, 2; blue, 27 }  ,draw opacity=1 ]   (181,59) -- (202.31,58) ;
\draw [<-][line width=1.5]    (180,20) -- (180,100) ;
\draw [<-][line width=1.5]    (100,20) -- (100,100) ;
\draw [<-][color={rgb, 255:red, 208; green, 2; blue, 27 }  ,draw opacity=1 ]   (101,80) -- (132,44) ;
\draw[->] [color={rgb, 255:red, 208; green, 2; blue, 27 }  ,draw opacity=1 ]   (132,80) -- (101,44) ;
\draw [<-][color={rgb, 255:red, 208; green, 2; blue, 27 }  ,draw opacity=1 ]   (40,44) -- (75,44) ;
\draw [<-][color={rgb, 255:red, 208; green, 2; blue, 27 }  ,draw opacity=1 ]   (40,80) -- (75,80) ;
\draw [<-][line width=1.5]    (38,20) -- (38,100) ;

\draw [->][color={rgb, 255:red, 208; green, 2; blue, 27 }  ,draw opacity=1 ][line width=0.75]    (460,80) -- (436.31,58) ;
\draw [<-][color={rgb, 255:red, 208; green, 2; blue, 27 }  ,draw opacity=1 ][line width=0.75]    (460,35) -- (436.31,58) ;
\draw[->] [color={rgb, 255:red, 208; green, 2; blue, 27 }  ,draw opacity=1 ]   (415,58) -- (436.31,58) ;
\draw [<-][line width=1.5]    (414,20) -- (414,100) ;
\draw [<-][line width=1.5]    (334,20) -- (334,100) ;
\draw [->] [color={rgb, 255:red, 208; green, 2; blue, 27 }  ,draw opacity=1 ]   (335,80) -- (368,44) ;
\draw [->] [color={rgb, 255:red, 208; green, 2; blue, 27 }  ,draw opacity=1 ]   (368,80) -- (335,44) ;
\draw[->] [color={rgb, 255:red, 208; green, 2; blue, 27 }  ,draw opacity=1 ]   (271,40) -- (307,40) ;
\draw [<-][color={rgb, 255:red, 208; green, 2; blue, 27 }  ,draw opacity=1 ]   (272,80) -- (308,80) ;
\draw [<-][line width=1.5]    (271,20) -- (271,100) ;

\draw  [color={rgb, 255:red, 208; green, 2; blue, 27 }  ,draw opacity=1 ][line width=0.75]    (71,270) -- (96.5,232) ;
\draw [color={rgb, 255:red, 208; green, 2; blue, 27 }  ,draw opacity=1 ][line width=0.75]    (125,270) -- (96.5,232) ;
\draw [->][color={rgb, 255:red, 208; green, 2; blue, 27 }  ,draw opacity=1 ][line width=0.75]    (96.5,232) -- (96.5,200) ;
\draw [->][color={rgb, 255:red, 208; green, 2; blue, 27 }  ,draw opacity=1 ][line width=0.75]    (198.5,233) -- (197.56,200) ;
\draw [color={rgb, 255:red, 208; green, 2; blue, 27 }  ,draw opacity=1 ][line width=0.75]    (198.5,233) .. controls (225,233) and (212,255) .. (177,270) ;
\draw [color={rgb, 255:red, 208; green, 2; blue, 27 }  ,draw opacity=1 ][line width=0.75]    (198.5,233) .. controls (178.5,233) and (189,255) .. (221,270) ;

\draw [->][color={rgb, 255:red, 208; green, 2; blue, 27 }  ,draw opacity=1 ][line width=0.75]    (345,261) -- (367.5,244) ;
\draw [->][color={rgb, 255:red, 208; green, 2; blue, 27 }  ,draw opacity=1 ][line width=0.75]    (391,261) -- (367.5,244) ;
\draw [->][color={rgb, 255:red, 208; green, 2; blue, 27 }  ,draw opacity=1 ][line width=0.75]    (366,243) -- (366,210) ;
\draw [->] [color={rgb, 255:red, 208; green, 2; blue, 27 }  ,draw opacity=1 ][line width=0.75]    (366,209) -- (391,192) ;
\draw [<-][color={rgb, 255:red, 208; green, 2; blue, 27 }  ,draw opacity=1 ][line width=0.75]    (366,209) -- (342,192) ;
\draw [->][color={rgb, 255:red, 208; green, 2; blue, 27 }  ,draw opacity=1 ][line width=0.75]    (464.5,223) -- (497.5,223) ;
\draw [->][color={rgb, 255:red, 208; green, 2; blue, 27 }  ,draw opacity=1 ][line width=0.75]    (564,259) -- (621,206) ;
\draw[->] [color={rgb, 255:red, 208; green, 2; blue, 27 }  ,draw opacity=1 ][line width=0.75]    (634,259) -- (580.92,206) ;
\draw [->][color={rgb, 255:red, 208; green, 2; blue, 27 }  ,draw opacity=1 ][line width=0.75]    (586,206) -- (616,206) ;

\draw [->] [color={rgb, 255:red, 208; green, 2; blue, 27 }  ,draw opacity=1 ][line width=0.75]    (562,190) -- (580,206) ;
\draw [->][color={rgb, 255:red, 208; green, 2; blue, 27 }  ,draw opacity=1 ][line width=0.75]    (622,206) -- (641.45,190) ;
\draw [->] [color={rgb, 255:red, 208; green, 2; blue, 27 }  ,draw opacity=1 ][line width=0.75]    (431,191) -- (463.07,222.39) ;

\draw [->][color={rgb, 255:red, 208; green, 2; blue, 27 }  ,draw opacity=1 ][line width=0.75]    (431,259) -- (463.12,225.24) ;

\draw [->][color={rgb, 255:red, 208; green, 2; blue, 27 }  ,draw opacity=1 ][line width=0.75]    (530,261) -- (500.76,224.76) ;
\draw [->][color={rgb, 255:red, 208; green, 2; blue, 27 }  ,draw opacity=1 ][line width=0.75]    (499.5,223.21) -- (530.6,191.43) ;

\draw (539,55) node [anchor=north west][inner sep=0.75pt]    {$=$};
\draw (543,104) node [anchor=north west][inner sep=0.75pt]    {$\textnormal{TC}$};
\draw (83,104) node [anchor=north west][inner sep=0.75pt]    {$\textnormal{STU}\ 1$};
\draw (73,55) node [anchor=north west][inner sep=0.75pt]    {$-$};
\draw (148,55) node [anchor=north west][inner sep=0.75pt]    {$=$};
\draw (317,104) node [anchor=north west][inner sep=0.75pt]    {$\textnormal{STU}\ 2$};
\draw (306,55) node [anchor=north west][inner sep=0.75pt]    {$-$};
\draw (382,55) node [anchor=north west][inner sep=0.75pt]    {$=$};
\draw (145,225) node [anchor=north west][inner sep=0.75pt]    {$+$};
\draw (227,225) node [anchor=north west][inner sep=0.75pt]    {$=\ 0$};
\draw (140,266) node [anchor=north west][inner sep=0.75pt]    {$\textnormal{AS}$};
\draw (400,214) node [anchor=north west][inner sep=0.75pt]    {$=$};
\draw (547,214) node [anchor=north west][inner sep=0.75pt]    {$-$};
\draw (468,266) node [anchor=north west][inner sep=0.75pt]    {$\textnormal{IHX}$};

\end{tikzpicture}

	\caption{The STU, TC, AS and IHX relations. }\label{fig:relations on J}
\end{figure}
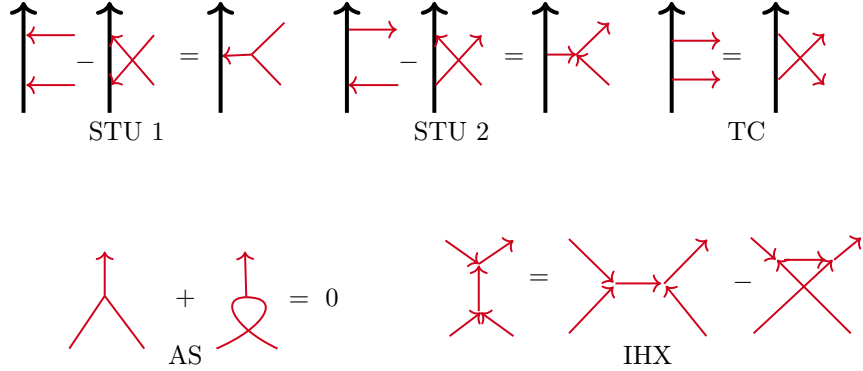

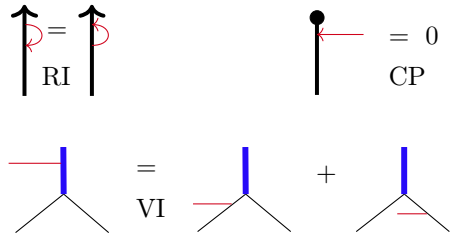
\begin{figure}

\begin{tikzpicture}[x=0.75pt,y=0.75pt,yscale=-.5,xscale=.5]

\draw   [line width=1.5] (415,50) -- (415,121) ;
\draw  [fill={rgb, 255:red, 0; green, 0; blue, 0 }  ,fill opacity=1 ] (408.13,43) .. controls (408.13,39.13) and (411.27,36) .. (415.13,36) .. controls (419,36) and (422.13,39.13) .. (422.13,43) .. controls (422.13,46.87) and (419,50) .. (415.13,50) .. controls (411.27,50) and (408.13,46.87) .. (408.13,43) -- cycle ;
\draw [<-][color={rgb, 255:red, 208; green, 2; blue, 27 }  ,draw opacity=1 ]   (415,60) -- (462,60) ;
\draw    [<-][line width=1.5](120,30) -- (120,122) ;
\draw [<-][color={rgb, 255:red, 208; green, 2; blue, 27 }  ,draw opacity=1 ]   (121,71) .. controls (143.93,71) and (142,50) .. (121,50) ;
\draw   [<-][line width=1.5] (188,30) -- (188,122) ;
\draw [->][color={rgb, 255:red, 208; green, 2; blue, 27 }  ,draw opacity=1 ]   (188,71) .. controls (210.93,71) and (209,50) .. (188,50) ;
\draw    (204.94,260) -- (159.52,221.09) ;
\draw    (111,260) -- (159.52,221.09) ;
\draw [color={rgb, 255:red, 38; green, 10; blue, 248 }  ,draw opacity=1 ][line width=2.25]    (159.18,173.58) -- (159.52,221.09) ;
\draw    (388.57,258.97) -- (343.14,221.09) ;
\draw    (294.62,258.97) -- (343.14,221.09) ;
\draw [color={rgb, 255:red, 38; green, 10; blue, 248 }  ,draw opacity=1 ][line width=2.25]    (342.8,173.58) -- (343.14,221.09) ;
\draw    (548.57,258.97) -- (503.14,221.09) ;
\draw    (454.62,258.97) -- (503.14,221.09) ;
\draw [color={rgb, 255:red, 38; green, 10; blue, 248 }  ,draw opacity=1 ][line width=2.25]    (503,173.58) -- (503,221.09) ;
\draw [color={rgb, 255:red, 208; green, 2; blue, 27 }  ,draw opacity=1 ]   (104,190) -- (157,190) ;
\draw [color={rgb, 255:red, 208; green, 2; blue, 27 }  ,draw opacity=1 ]   (290,231) -- (329,231) ;
\draw [color={rgb, 255:red, 208; green, 2; blue, 27 }  ,draw opacity=1 ]   (496,241) -- (525.85,241) ;

\draw (485,50) node [anchor=north west][inner sep=0.75pt]    {$=\ 0$};
\draw (485,90) node [anchor=north west][inner sep=0.75pt]    {$\textnormal{CP}$};
\draw (140,50) node [anchor=north west][inner sep=0.75pt]    {$=$};
\draw (135,90) node [anchor=north west][inner sep=0.75pt]    {$\textnormal{RI}$};
\draw (230,190) node [anchor=north west][inner sep=0.75pt]    {$=$};
\draw (230,223.09) node [anchor=north west][inner sep=0.75pt]    {$\textnormal{VI}$};
\draw (411.78,190) node [anchor=north west][inner sep=0.75pt]  [font=\normalsize]  {$+$};

\end{tikzpicture}
\caption{Relations VI, CP and RI in $\arrows$.}\label{fig:VICPRI}
\end{figure}

All of these relations are degree-preserving and hence the quotient $\J(s)$ remains graded. The STU and TC relations in fact imply the AS and IHX relations; the latter are stated for convenience. We remark that the arrow diagram relations described arise from the relations of $\wf$ via the associated graded construction. For a brief explanation of this point, see \cite[Definition 3.8 and Proposition 3.9]{BND:WKO2}.

The collection $\J=\{\J(s)\}_{s\in\calS}$ assembles naturally into a linear circuit algebra where the operad of wiring diagrams acts on $\J$ by wiring together the skeleta and preserving the arrow graphs.

\medskip 
The circuit algebra $\J$ admits a -- not obvious -- finite presentation. Note that via iterated applications of the STU relations, all trivalent vertices in Jacobi diagrams can be eliminated. In other words, all Jacobi diagrams are equivalent to an ``arrows only'' form: a linear combination of diagrams in which only oriented arrows are attached to the foam skeleton. Such diagrams are called \emph{arrow diagrams}, as below. It turns out that the arrow diagram equivalent of a $w$-Jacobi diagram is unique up to the arrow diagram relations 4T, TC, VI, CP and RI, and this leads to an isomorphic finite presentation of $\J$ as a circuit algebra, denoted $\arrows$, as below. \cite[Theorem 4.6, based on Theorem 3.13]{BND:WKO2}

\begin{definition}\label{def: arrow diagram circuit alg}\cite[Definition 4.8]{BND:WKO2}
The circuit algebra of \emph{arrow diagrams} is the complete circuit algebra in $\mathbb{Q}$-vector spaces with presentation\footnote{One technically needs to include caps and vertices in all possible orientations as circuit algebra generators. However, these can all be obtained from the upward-oriented cap and vertex using orientation switch operations.}
\[\arrows := \mathsf{CA}\left< \rightarrowdiagram, \upcap, \vertex \mid \text{4T, TC, VI, CP, RI}\right>.\] In addition, $\arrows$ is equipped with the associated graded auxiliary operations:  orientation switches $S_e, A_e$, unzips $u_e$, and strand deletion $d_e$ described in Definition~\ref{def:operations on arrows}.

The relations are shown in Figures~\ref{fig:VICPRI}~and~\ref{fig:TC4T}. There is a well-defined skeleton map $\pi:\arrows\rightarrow\calS$ which forgets arrows, so $\arrows$ is indexed by skeleta: $\arrows=\coprod_{s\in\calS} \arrows(s)$, where $\arrows(s):=\pi^{-1}(s)$. For any given $s\in S$, $\arrows(s)$ is a graded complete vector space, where the degree is given by the number of arrows.
\end{definition}

\begin{figure}[h]
\begin{tikzpicture}[x=0.75pt,y=0.75pt,yscale=-.7,xscale=.7]

\draw   [<-][line width=1.5]   (61,50) -- (61,150) ;
\draw [<-][line width=1.5]     (101,50) -- (101,150) ;
\draw  [<-][line width=1.5]    (140,50) -- (140,150) ;
\draw  [<-][line width=1.5]    (182,50) -- (182,150) ;
\draw  [<-][line width=1.5]    (222,50) -- (222,150) ;
\draw   [<-][line width=1.5] (261,50) -- (261,150) ;
\draw   [<-][line width=1.5] (334,50) -- (334,150) ;
\draw  [<-][line width=1.5]  (374,50) -- (374,150) ;
\draw  [<-][line width=1.5]  (413,50) -- (413,150) ;
\draw  [<-][line width=1.5]  (455,50) -- (455,150) ;
\draw  [<-][line width=1.5]  (495,50) -- (495,150) ;
\draw  [<-][line width=1.5]  (534,50) -- (534,150) ;
\draw [<-][color={rgb, 255:red, 208; green, 2; blue, 27 }  ,draw opacity=1 ]   (61,80) -- (100,80) ;
\draw [->][color={rgb, 255:red, 208; green, 2; blue, 27 }  ,draw opacity=1 ]   (63,110) -- (138,110) ;
\draw [->][color={rgb, 255:red, 208; green, 2; blue, 27 }  ,draw opacity=1 ]   (183,110) -- (258,110) ;
\draw [->][color={rgb, 255:red, 208; green, 2; blue, 27 }  ,draw opacity=1 ]   (222,80) -- (261,80) ;
\draw [->][color={rgb, 255:red, 208; green, 2; blue, 27 }  ,draw opacity=1 ]   (334,80) -- (409,80) ;
\draw [<-][color={rgb, 255:red, 208; green, 2; blue, 27 }  ,draw opacity=1 ]   (335,110) -- (374,110) ;
\draw [->] [color={rgb, 255:red, 208; green, 2; blue, 27 }  ,draw opacity=1 ]   (457,80) -- (532,80) ;
\draw [->][color={rgb, 255:red, 208; green, 2; blue, 27 }  ,draw opacity=1 ]   (495,110) -- (534,110) ;

\draw (288,76.4) node [anchor=north west][inner sep=0.75pt]    {$=$};
\draw (288,105.4) node [anchor=north west][inner sep=0.75pt]    {$4T$};
\draw (154,81) node [anchor=north west][inner sep=0.75pt]    {$+$};
\draw (430,81) node [anchor=north west][inner sep=0.75pt]    {$+$};

\end{tikzpicture}
\caption{Relation $4$T in $\arrows$.}\label{fig:TC4T}
\end{figure}

The following is Theorem~4.5 of \cite{BND:WKO2}. In the theory of finite type invariants, this is called a ``bracket rise'' theorem, and the proof uses standard Jacobi diagram techniques, as in \cite[Chapter 5]{chmutov_duzhin_mostovoy_2012}, for example. We give a proof sketch below to illustrate the main points.

\begin{thm}\label{equivalence J and A}
There is a canonical isomorphism of circuit algebras $\arrows\cong \J$.
\end{thm}

\begin{proof} ({\em Sketch}) Since arrow diagrams are in particular $w$-Jacobi diagrams (without trivalent vertices), there is a natural inclusion
$\iota_s: \arrows(s)\rightarrow \J(s)$.  The image of the $4T$ relation under $\iota_s$ vanishes by the STU relation. The rest of the relations of $\arrows(s)$ are also imposed in $\J(s)$, so $\iota_s$ is well-defined.  The inclusion $\iota_s$ has an inverse  $\phi_s:\J(s)\rightarrow \arrows(s)$ which resolves all of the trivalent vertices of a $w$-Jacobi diagram by repeated applications of the STU relation, as mentioned above; to show that this is well-defined is a case analysis.  

The operad $\mathsf{WD}$ acts on both $\J$ and $\arrows$ by wiring together skeleta and preserving arrows or arrow graphs. Thus, the collection of all isomorphisms $\iota_s: \arrows(s)\rightarrow \J(s)$ assembles into a circuit algebra isomorphism $\iota:\arrows\rightarrow \J$. 
\end{proof}

\begin{remark}
The presentation $\arrows$ has the obvious advantages of a finite presentation; while the definition of $\J$ has a more convenient set of relations, especially in the context of Lie algebras, which plays a prominent role later in this paper. Given the canonical isomorphism $\arrows \cong \J$, from now on we do not distinguish $w$-Jacobi and arrow diagrams, refer to both as ``arrow diagrams'', and use $\arrows$ to denote either space. In general, we work with Jacobi diagrams, but we use the finite presentation of $\arrows$ in a fundamental way in Sections~\ref{sec: expansions} and \ref{sec: KRV}. 
\end{remark} 

The circuit algebra $\arrows$ is equipped with auxiliary operations: $S_e, A_e, u_e$ and $d_e$, which are the associated graded operations of the $w$-foam operations by the same name described in Definition~\ref{def:w-foam operations}. 

\begin{definition}\label{def:operations on arrows} 
Let $e$ denote a strand (edge) in the skeleton $s\in\calS$. The auxiliary operations on $\arrows$ are defined as follows:
\begin{enumerate}
	\item The \textbf{orientation switch} $S_e: \arrows(s) \rightarrow \arrows(S_e(s))$ reverses the direction of $e$ and multiplies each arrow diagram by $$(-1)^{\#(\text{arrow heads and tails on }e)}.$$
	\item The \textbf{adjoint} map $A_e: \arrows(s) \rightarrow \arrows(A_e(s))$ reverses the direction of $e$ and multiplies each arrow diagram by
	$$(-1)^{\#(\text{arrow heads on } e)}.$$
	\item The \textbf{unzip} and \textbf{disc unzip} maps, both denoted $u_e: \arrows(s) \rightarrow \arrows(u_e(s))$, unzip the strand $e$ and map each arrow ending on $e$ to a sum of two arrows, as shown in Figure~\ref{fig: unzip for arrows}.
	\item The \textbf{deletion} map $d_e: \arrows(s) \rightarrow \arrows(d_e(s))$ deletes the long strand $e$ (which is not incident to any skeleton vertices), and an arrow diagram with any arrow head or tail on $e$ is sent to 0.
\end{enumerate}
\end{definition}

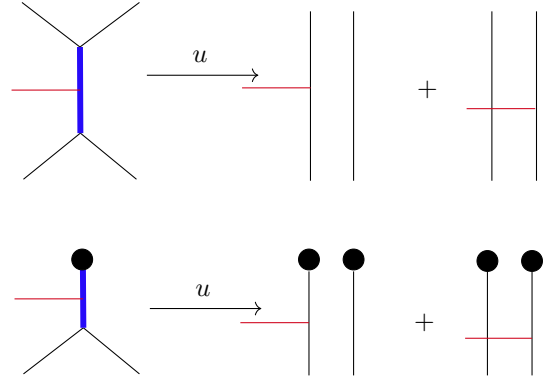
\begin{figure}[h]
\begin{tikzpicture}[x=0.75pt,y=0.75pt,yscale=-.75,xscale=.75]

\draw    (7,19) -- (46,49) ;
\draw    (86,19) -- (46,49) ;
\draw    (84,138) -- (46.28,106.95) ;
\draw    (8,138) -- (46.28,106.95) ;
\draw [color={rgb, 255:red, 38; green, 10; blue, 248 }  ,draw opacity=1 ][line width=2.25]    (46,49) -- (46.28,106.95) ;
\draw  [->]  (91,68) -- (166,68) ;
\draw [color={rgb, 255:red, 0; green, 0; blue, 0 }  ,draw opacity=1 ]    (201,25) -- (201,139) ;
\draw    (86,269) -- (48.28,237.95) ;
\draw    (8,269) -- (48.28,237.95) ;
\draw [color={rgb, 255:red, 38; green, 10; blue, 248 }  ,draw opacity=1 ][line width=2.25]    (48,199) -- (48.28,237.95) ;
\draw  [fill={rgb, 255:red, 0; green, 0; blue, 0 }  ,fill opacity=1 ] (40.5,192) .. controls (40.5,188.13) and (43.63,185) .. (47.5,185) .. controls (51.37,185) and (54.5,188.13) .. (54.5,192) .. controls (54.5,195.87) and (51.37,199) .. (47.5,199) .. controls (43.63,199) and (40.5,195.87) .. (40.5,192) -- cycle ;
\draw  [->]  (93,225) -- (168,225) ;
\draw    (200,200) -- (200,270) ;
\draw  [fill={rgb, 255:red, 0; green, 0; blue, 0 }  ,fill opacity=1 ] (193,192) .. controls (193,188.13) and (196.13,185) .. (200,185) .. controls (203.87,185) and (207,188.13) .. (207,192) .. controls (207,195.87) and (203.87,199) .. (200,199) .. controls (196.13,199) and (193,195.87) .. (193,192) -- cycle ;
\draw    (230,200) -- (230,270) ;
\draw  [fill={rgb, 255:red, 0; green, 0; blue, 0 }  ,fill opacity=1 ] (223,192) .. controls (223,188.13) and (226.13,185) .. (230,185) .. controls (233.87,185) and (237,188.13) .. (237,192) .. controls (237,195.87) and (233.87,199) .. (230,199) .. controls (226.13,199) and (223,195.87) .. (223,192) -- cycle ;
\draw [color={rgb, 255:red, 208; green, 2; blue, 27 }  ,draw opacity=1 ]   (0,78) -- (46.14,77.97) ;
\draw [color={rgb, 255:red, 208; green, 2; blue, 27 }  ,draw opacity=1 ]   (155,76.5) -- (201.14,76.47) ;
\draw [color={rgb, 255:red, 208; green, 2; blue, 27 }  ,draw opacity=1 ]   (2,218.5) -- (48.14,218.47) ;
\draw [color={rgb, 255:red, 208; green, 2; blue, 27 }  ,draw opacity=1 ]   (153.86,234.53) -- (200,234.5) ;
\draw [color={rgb, 255:red, 0; green, 0; blue, 0 }  ,draw opacity=1 ]  (230,139) -- (230,25) ;
\draw [color={rgb, 255:red, 0; green, 0; blue, 0 }  ,draw opacity=1 ]  (323,25) -- (323,139) ;
\draw [color={rgb, 255:red, 208; green, 2; blue, 27 }  ,draw opacity=1 ]   (306,90.5) -- (352.14,90.47) ;
\draw [color={rgb, 255:red, 0; green, 0; blue, 0 }  ,draw opacity=1 ]    (353,139) -- (353,25) ;
\draw    (320,200) -- (320,270) ;
\draw  [fill={rgb, 255:red, 0; green, 0; blue, 0 }  ,fill opacity=1 ] (313,193) .. controls (313,189.13) and (316.13,186) .. (320,186) .. controls (323.87,186) and (327,189.13) .. (327,193) .. controls (327,196.87) and (323.87,200) .. (320,200) .. controls (316.13,200) and (313,196.87) .. (313,193) -- cycle ;
\draw    (350,200) -- (350,270) ;
\draw  [fill={rgb, 255:red, 0; green, 0; blue, 0 }  ,fill opacity=1 ] (343,193) .. controls (343,189.13) and (346.13,186) .. (350,186) .. controls (353.87,186) and (357,189.13) .. (357,193) .. controls (357,196.87) and (353.87,200) .. (350,200) .. controls (346.13,200) and (343,196.87) .. (343,193) -- cycle ;
\draw [color={rgb, 255:red, 208; green, 2; blue, 27 }  ,draw opacity=1 ]   (305,245) -- (350,245) ;

\draw (120,50) node [anchor=north west][inner sep=0.75pt]    {$u$};
\draw (122,207) node [anchor=north west][inner sep=0.75pt]    {$u$};
\draw (271,70) node [anchor=north west][inner sep=0.75pt]    {$+$};
\draw (269,227) node [anchor=north west][inner sep=0.75pt]    {$+$};

\end{tikzpicture}
\caption{The unzip operations $u_e$ for arrow diagrams. A single arrow ending on $e$ is mapped to a sum of two arrows. If a diagram has $k$ arrows ending on $e$, it is mapped to a sum of $2^k$ terms.}\label{fig: unzip for arrows}
\end{figure}


\begin{remark} 
The operations are well-defined, that is, they commute with all the relations on $\arrows$. For example, the commutative diagram in Figure~\ref{VI and unzip} shows how unzip and the ``vertex invariance'' operation (VI) are compatible. The reader may verify as an exercise that all of the operations are compatible with the relations on $\arrows$. 
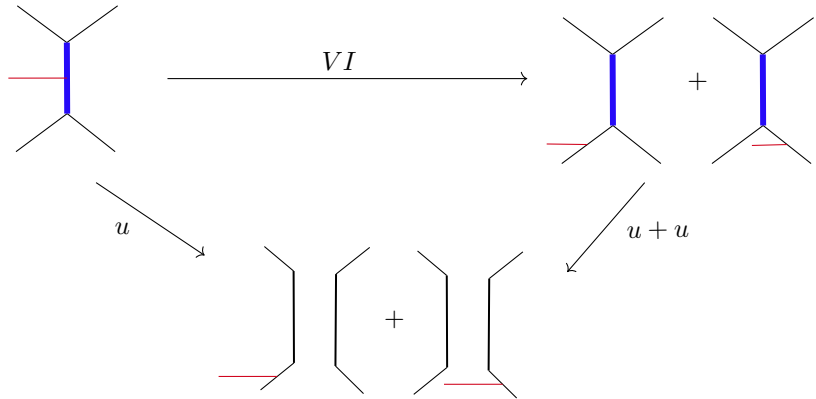
\begin{figure}[h]

\begin{tikzpicture}[x=0.75pt,y=0.75pt,yscale=-.75,xscale=.75]

\draw    (56.98,36) -- (90.32,60.75) ;
\draw    (124.51,36) -- (90.32,60.75) ;
\draw    (122.8,134.18) -- (90.56,108.56) ;
\draw    (56.13,134.18) -- (90.56,108.56) ;
\draw [color={rgb, 255:red, 38; green, 10; blue, 248 }  ,draw opacity=1 ][line width=2.25]    (90.32,60.75) -- (90.56,108.56) ;
\draw    (223.15,198) -- (242.81,214.5) ;
\draw    (220.59,294.52) -- (243.05,276.33) ;
\draw [color={rgb, 255:red, 0; green, 0; blue, 0 }  ,draw opacity=1 ][line width=0.75]    (242.81,214.5) -- (243.05,276.33) ;
\draw    (289.82,297) -- (270.94,278.39) ;
\draw    (294,198.59) -- (271.35,216.56) ;
\draw [color={rgb, 255:red, 0; green, 0; blue, 0 }  ,draw opacity=1 ][line width=0.75]    (270.94,278.39) -- (271.35,216.56) ;
\draw [color={rgb, 255:red, 208; green, 2; blue, 27 }  ,draw opacity=1 ]   (51,84.68) -- (90.44,84.65) ;
\draw [color={rgb, 255:red, 208; green, 2; blue, 27 }  ,draw opacity=1 ]   (192.38,285.45) -- (231.82,285.43) ;
\draw    (423.98,44) -- (457.32,68.75) ;
\draw    (491.51,44) -- (457.32,68.75) ;
\draw    (489.8,142.18) -- (457.56,116.56) ;
\draw    (423.13,142.18) -- (457.56,116.56) ;
\draw [color={rgb, 255:red, 38; green, 10; blue, 248 }  ,draw opacity=1 ][line width=2.25]    (457.32,68.75) -- (457.56,116.56) ;
\draw [color={rgb, 255:red, 208; green, 2; blue, 27 }  ,draw opacity=1 ]   (413,129) -- (440.34,129.37) ;
\draw    (524.98,44) -- (558.32,68.75) ;
\draw    (592.51,44) -- (558.32,68.75) ;
\draw    (590.8,142.18) -- (558.56,116.56) ;
\draw    (524.13,142.18) -- (558.56,116.56) ;
\draw [color={rgb, 255:red, 38; green, 10; blue, 248 }  ,draw opacity=1 ][line width=2.25]    (558.32,68.75) -- (558.56,116.56) ;
\draw [color={rgb, 255:red, 208; green, 2; blue, 27 }  ,draw opacity=1 ]   (551,130) -- (574.68,129.37) ;
\draw [->]   (158,85) -- (400,85) ;
\draw    (326.15,201) -- (345.81,217.5) ;
\draw    (323.59,297.52) -- (346.05,279.33) ;
\draw [color={rgb, 255:red, 0; green, 0; blue, 0 }  ,draw opacity=1 ][line width=0.75]    (345.81,217.5) -- (346.05,279.33) ;
\draw    (392.82,300) -- (373.94,281.39) ;
\draw    (397,201.59) -- (374.35,219.56) ;
\draw [color={rgb, 255:red, 0; green, 0; blue, 0 }  ,draw opacity=1 ][line width=0.75]    (373.94,281.39) -- (374.35,219.56) ;
\draw [color={rgb, 255:red, 208; green, 2; blue, 27 }  ,draw opacity=1 ]   (343.94,290.72) -- (383.38,290.7) ;
\draw  [->]  (478.91,155) -- (426.81,215.16) ;
\draw  [->]  (110,155) -- (183,204) ;

\draw (506,79.4) node [anchor=north west][inner sep=0.75pt]    {$+$};
\draw (302,239.4) node [anchor=north west][inner sep=0.75pt]    {$+$};
\draw (260,65.4) node [anchor=north west][inner sep=0.75pt]    {$VI$};
\draw (465,180) node [anchor=north west][inner sep=0.75pt]  [rotate=-359.05]  {$u+u$};
\draw (121,180) node [anchor=north west][inner sep=0.75pt]    {$u$};

\end{tikzpicture}
\caption{The compatibility of the unzip operation with the VI relation.}\label{VI and unzip}
\end{figure}
\end{remark} 

\begin{remark} 
We note that the rotation invariance $RI$ relation of $\arrows$ has an equivalent ``wheel'' formulation for Jacobi diagrams, as shown in Figure~\ref{fig:wheelRI}. 
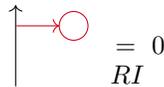
\begin{figure}[h] 

\begin{tikzpicture}[x=0.75pt,y=0.75pt,yscale=-.5,xscale=.5]

\draw [<-]   (127,121) -- (127,203) ;
\draw [->][color={rgb, 255:red, 208; green, 2; blue, 27 }  ,draw opacity=1 ]   (128,142) -- (171,142) ;
\draw  [color={rgb, 255:red, 208; green, 2; blue, 27 }  ,draw opacity=1 ] (171,142) .. controls (171,133.99) and (177.49,127.5) .. (185.5,127.5) .. controls (193.51,127.5) and (200,133.99) .. (200,142) .. controls (200,150.01) and (193.51,156.5) .. (185.5,156.5) .. controls (177.49,156.5) and (171,150.01) .. (171,142) -- cycle ;

\draw (225,150) node [anchor=north west][inner sep=0.75pt]    {$=\ 0$};
\draw (220,180) node [anchor=north west][inner sep=0.75pt]    {$RI$};
\end{tikzpicture}
\caption{The RI relation, reformulated.}\label{fig:wheelRI}
\end{figure}
\end{remark}

Recall that the complete associated graded space of $\wf$ (Definition~\ref{def:gr}) was also denoted $\arrows$, like the circuit algebra of arrow diagrams. This is no accident, as the two spaces are canonically isomorphic. This follows from the existence of a homomorphic expansion for $\wf$ combined with \cite{BND:WKO2} Propositions 3.9 and 2.7: 

\begin{theorem} \cite{BND:WKO2}
The circuit algebra $\arrows$ of arrow diagrams is canonically isomorphic to the complete associated graded circuit algebra of $\wf$ and the diagrammatic operations $S_e, A_e, u_e$ and $d_e$ are the associated graded maps of their topological counterparts.
\end{theorem}


\subsection{Algebraic structures in $\arrows$}\label{sec:coproduct and product} 
At each $w$-foam skeleton $s\in\calS$, the $\mathbb{Q}$-vector space $\arrows(s)$ can be equipped with a counital, coassociative coproduct:
\begin{definition}\label{def: coprod}
For each $s\in\calS$ the coproduct
 \[\begin{tikzcd}\Delta: \arrows(s)\arrow[r]& \arrows(s)\otimes\arrows(s)\end{tikzcd}\] 
 sends an arrow diagram $D\in\arrows(s)$ to the sum of all ways of distributing the connected components of the arrow graph of $D$ -- that is, the diagram minus the skeleton -- amongst two copies of the skeleton. 
 \end{definition} 

An example of the coproduct is given in Figure~\ref{fig:coproduct}. Note that the relative position of the two connected components of the Jacobi diagram is preserved. 
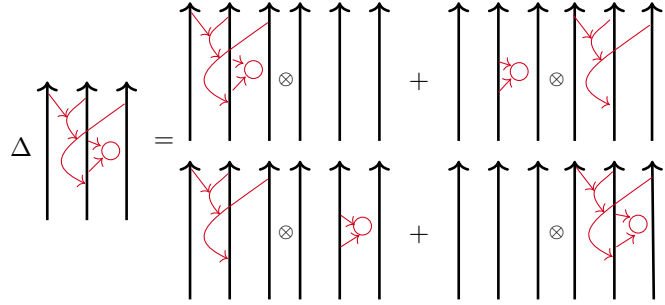
\begin{figure}[h!]

\begin{tikzpicture}[x=0.75pt,y=0.75pt,yscale=-.5,xscale=.5]

\draw [<-][line width=1]   (45,80) -- (45,220) ;
\draw[<-][line width=1]   (85,80) -- (85,220) ;
\draw[<-][line width=1]   (125,80) -- (125,220) ;
\draw[->] [color={rgb, 255:red, 208; green, 2; blue, 27 }  ,draw opacity=1 ]    (46.35,92.41) -- (64.61,116.21) ;
\draw [->, color={rgb, 255:red, 208; green, 2; blue, 27 }  ,draw opacity=1 ]    (123,102.58) .. controls (71.87,140.14) and (32.07,159.57) .. (83.33,183.01) ;
\draw [->][color={rgb, 255:red, 208; green, 2; blue, 27 }  ,draw opacity=1 ]   (83.4,97.01) .. controls (70.54,110.15) and (53.59,114.33) .. (74.84,138.1) ;
\draw  [color={rgb, 255:red, 208; green, 2; blue, 27 }  ,draw opacity=1 ] (100,150) .. controls (100,145.03) and (104.03,141) .. (109,141) .. controls (113.97,141) and (118,145.03) .. (118,150) .. controls (118,154.97) and (113.97,159) .. (109,159) .. controls (104.03,159) and (100,154.97) .. (100,150) -- cycle ;
\draw[<-] [color={rgb, 255:red, 208; green, 2; blue, 27 }  ,draw opacity=1 ]    (101,157) -- (87,171) ;
\draw [->] [color={rgb, 255:red, 208; green, 2; blue, 27 }  ,draw opacity=1 ]    (86,140) -- (99,145) ;
\draw[<-][line width=1]    (189,-0) -- (189,140) ;
\draw[<-][line width=1]   (229,-0) -- (229,140) ;
\draw [<-][line width=1]   (269,-0) -- (269,140) ;
\draw [->][color={rgb, 255:red, 208; green, 2; blue, 27 }  ,draw opacity=1 ]   (190.35,10.41) -- (208.61,34.21) ;
\draw [->, color={rgb, 255:red, 208; green, 2; blue, 27 }  ,draw opacity=1 ]  (267.46,20.58) .. controls (215.87,58.14) and (176.07,77.57) .. (227.33,101.01) ;
\draw [->] [color={rgb, 255:red, 208; green, 2; blue, 27 }  ,draw opacity=1 ]   (227.4,15.01) .. controls (214.54,28.15) and (197.59,32.33) .. (218.84,56.1) ;
\draw  [->][color={rgb, 255:red, 208; green, 2; blue, 27 }  ,draw opacity=1 ]  (244,68) .. controls (244,63.03) and (248.03,59) .. (253,59) .. controls (257.97,59) and (262,63.03) .. (262,68) .. controls (262,72.97) and (257.97,77) .. (253,77) .. controls (248.03,77) and (244,72.97) .. (244,68) -- cycle ;
\draw[<-] [color={rgb, 255:red, 208; green, 2; blue, 27 }  ,draw opacity=1 ]   (245,75) -- (232,89) ;
\draw [->] [color={rgb, 255:red, 208; green, 2; blue, 27 }  ,draw opacity=1 ]   (232,58) -- (243,62) ;
\draw [<-][line width=1]   (460,0) -- (460,140) ;
\draw [<-][line width=1]   (500,0) -- (500,140) ;
\draw [<-][line width=1]   (540,0) -- (540,140) ;
\draw  [color={rgb, 255:red, 208; green, 2; blue, 27 }  ,draw opacity=1 ]  (512,70) .. controls (512,65.03) and (516.03,61) .. (521,61) .. controls (525.97,61) and (530,65.03) .. (530,70) .. controls (530,74.97) and (525.97,79) .. (521,79) .. controls (516.03,79) and (512,74.97) .. (512,70) -- cycle ;
\draw [<-][color={rgb, 255:red, 208; green, 2; blue, 27 }  ,draw opacity=1 ]   (513,77) -- (500,91) ;
\draw[->] [color={rgb, 255:red, 208; green, 2; blue, 27 }  ,draw opacity=1 ]  (500,60) -- (513,64) ;
\draw [<-][line width=1]   (189,160) -- (189,300) ;
\draw[<-][line width=1]   (229,160) -- (229,300) ;
\draw [<-][line width=1]   (269,160) -- (269,300) ;
\draw[->] [color={rgb, 255:red, 208; green, 2; blue, 27 }  ,draw opacity=1 ]  (190.35,168.41) -- (208.61,192.21) ;
\draw [->, color={rgb, 255:red, 208; green, 2; blue, 27 }  ,draw opacity=1 ]   (267.46,178.58) .. controls (215.87,216.14) and (176.07,235.57) .. (227.33,259.01) ;
\draw[->] [color={rgb, 255:red, 208; green, 2; blue, 27 }  ,draw opacity=1 ]  (227.4,173.01) .. controls (214.54,186.15) and (197.59,190.33) .. (218.84,214.1) ;
\draw [<-][line width=1]   (460,160) -- (460,300) ;
\draw [<-][line width=1]    (500,160) -- (500,300) ;
\draw [<-][line width=1]   (540,160) -- (540,300) ;
\draw [<-][line width=1]    (300,-0) -- (300,140) ;
\draw [<-][line width=1]   (340,-0) -- (340,140) ;
\draw [<-][line width=1]   (380,-0) -- (380,140) ;
\draw [<-][line width=1]  (300,160) -- (300,300) ;
\draw [<-][line width=1]   (340,160) -- (340,300) ;
\draw [<-][line width=1]    (380,160) -- (380, 300) ;
\draw  [color={rgb, 255:red, 208; green, 2; blue, 27 }  ,draw opacity=1 ]  (354,226) .. controls (354,221.03) and (358.03,217) .. (363,217) .. controls (367.97,217) and (372,221.03) .. (372,226) .. controls (372,230.97) and (367.97,235) .. (363,235) .. controls (358.03,235) and (354,230.97) .. (354,226) -- cycle ;
\draw [<-][color={rgb, 255:red, 208; green, 2; blue, 27 }  ,draw opacity=1 ]    (360,235) -- (341,247) ;
\draw[->] [color={rgb, 255:red, 208; green, 2; blue, 27 }  ,draw opacity=1 ]    (341,215) -- (355,225) ;

\draw[<-] [line width=1]    (577,0) -- (577,140) ;
\draw[<-] [line width=1]    (617,0) -- (617,140) ;
\draw[<-] [line width=1]    (655,0) -- (655,140) ;
\draw[->] [color={rgb, 255:red, 208; green, 2; blue, 27 }  ,draw opacity=1 ]   (575.6,12.99) -- (593.86,36.79) ;
\draw[->] [color={rgb, 255:red, 208; green, 2; blue, 27 }  ,draw opacity=1 ]  (652.71,23.15) .. controls (601.12,60.71) and (561.32,80.14) .. (612.58,103.59) ;
\draw[->] [color={rgb, 255:red, 208; green, 2; blue, 27 }  ,draw opacity=1 ]   (612.65,17.59) .. controls (599.79,30.73) and (582.84,34.91) .. (604.09,58.68) ;

\draw[<-] [line width=1]    (577,160) -- (577,300) ;
\draw[<-] [line width=1]    (617,160) -- (617,300) ;
\draw [<-][line width=1]    (655.22,160) -- (656.75,300) ;
\draw[->] [color={rgb, 255:red, 208; green, 2; blue, 27 }  ,draw opacity=1 ]   (578.35,166.41) -- (596.61,190.21) ;
\draw[->] [color={rgb, 255:red, 208; green, 2; blue, 27 }  ,draw opacity=1 ]   (655.46,176.58) .. controls (603.87,214.14) and (564.07,233.57) .. (615.33,257.01) ;
\draw[->] [color={rgb, 255:red, 208; green, 2; blue, 27 }  ,draw opacity=1 ]   (615.4,171.01) .. controls (602.54,184.15) and (585.59,188.33) .. (606.84,212.1) ;
\draw  [color={rgb, 255:red, 208; green, 2; blue, 27 }  ,draw opacity=1 ]  (632,224) .. controls (632,219.03) and (636.03,215) .. (641,215) .. controls (645.97,215) and (650,219.03) .. (650,224) .. controls (650,228.97) and (645.97,233) .. (641,233) .. controls (636.03,233) and (632,228.97) .. (632,224) -- cycle ;
\draw [<-][color={rgb, 255:red, 208; green, 2; blue, 27 }  ,draw opacity=1 ]    (638,235) -- (619,247) ;
\draw [->][color={rgb, 255:red, 208; green, 2; blue, 27 }  ,draw opacity=1 ]   (618,214) -- (633,217) ;

\draw (150,134) node [anchor=north west][inner sep=0.75pt]    {$=$};
\draw (5,134) node [anchor=north west][inner sep=0.75pt]    {$\Delta$};
\draw (407,68) node [anchor=north west][inner sep=0.75pt]    {$+$};
\draw (407,223) node [anchor=north west][inner sep=0.75pt]    {$+$};
\draw (275,68) node [anchor=north west][inner sep=0.75pt]  [font=\scriptsize]   {$\otimes $};
\draw (275,223) node [anchor=north west][inner sep=0.75pt]  [font=\scriptsize]   {$\otimes $};
\draw (548,223) node [anchor=north west][inner sep=0.75pt]   [font=\scriptsize]  {$\otimes $};
\draw (548,68) node [anchor=north west][inner sep=0.75pt]   [font=\scriptsize]  {$\otimes $};
\end{tikzpicture}

\caption{The coproduct of an arrow diagram in $\arrows(\uparrow_3)$. The arrow diagram on the left has two connected components (after the skeleton is removed). Hence, the coproduct is a sum of four terms.}\label{fig:coproduct} 
\end{figure} 
 
It is a straightforward exercise, which we leave to the reader, to verify that the coproduct in Definition~\ref{def: coprod} is counital and coassociative. 

\begin{definition}\label{def:primitives}
An arrow diagram $D\in\arrows(s)$ is said to be \emph{primitive} if \[\Delta(D) = D\otimes 1 + 1 \otimes D.\] 
\end{definition}

A quick inspection shows that primitive arrow diagrams are (linear combinations of) diagrams whose arrow graph (the arrow diagram with the skeleton removed)  is connected. For example, the arrow diagram in Figure~\ref{fig:coproduct} is \emph{not} primitive, but the arrow diagram in  Figure~\ref{fig:wheel example} is primitive.

In \cite[Section 3.2]{BND:WKO2} there is a full characterisation of connected arrow graphs in terms of {\em trees} and {\em wheels}:
\begin{enumerate}
\item {\em trees} are cycle-free arrow graphs oriented towards a single head, and
\item {\em wheels} consist of an oriented cycle with inward-oriented arrow ``spokes''.
\end{enumerate}
The two components of the arrow graph in Figure~\ref{fig:coproduct} are a tree and a 2-wheel. Arrow graphs where trees are attached to a cycle can be resolved using IHX relations to a linear combination of wheels. Arrow graphs with more than one cycle are zero as they cannot satisfy the two-in-one-out condition at all vertices.

\medskip 

An important case of arrow diagram spaces are arrow diagrams on a skeleton which  consists of $n$ long strands, drawn as $n$ upward-oriented vertical lines,  $\arrows(\uparrow_n)$.  For example, the arrow diagrams in Figure~\ref{fig:coproduct} are in $\arrows(\uparrow_3)$.  In addition to the coalgebra structure, $\arrows(\uparrow_n)$ is equipped with an associative product 
\[\begin{tikzcd} \arrows(\uparrow_n)\otimes \arrows(\uparrow_n) \arrow[r,"\WD_{s}"] & \arrows(\uparrow_n) \end{tikzcd}\]  
given by ``stacking''. Stacking is a circuit algebra operation realised by the wiring diagram in Figure~\ref{WD for stacking}. This product is associative because wiring diagram composition is associative. The multiplicative unit is the empty arrow diagram (with no arrows) in $\arrows(\uparrow_n)$.

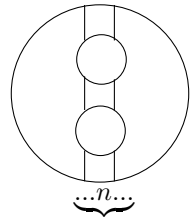
\begin{figure}[h] 
\begin{tikzpicture}[x=0.75pt,y=0.75pt,yscale=-.5,xscale=.5]

\draw   (302,179) .. controls (302,165.19) and (313.19,154) .. (327,154) .. controls (340.81,154) and (352,165.19) .. (352,179) .. controls (352,192.81) and (340.81,204) .. (327,204) .. controls (313.19,204) and (302,192.81) .. (302,179) -- cycle ;
\draw   (303,107) .. controls (303,93.19) and (314.19,82) .. (328,82) .. controls (341.81,82) and (353,93.19) .. (353,107) .. controls (353,120.81) and (341.81,132) .. (328,132) .. controls (314.19,132) and (303,120.81) .. (303,107) -- cycle ;
\draw   (237,141.31) .. controls (237,190.74) and (277.07,230.81) .. (326.5,230.81) .. controls (375.93,230.81) and (416,190.74) .. (416,141.31) .. controls (416,91.88) and (375.93,51.81) .. (326.5,51.81) .. controls (277.07,51.81) and (237,91.88) .. (237,141.31) -- cycle ;
\draw    (311,124) -- (311,158) ;
\draw    (341,127) -- (341,158) ;
\draw    (311,199) -- (311, 230) ;
\draw    (341,199) -- (341, 230) ;
\draw    (311,52) -- (311,87.81) ;
\draw    (341,52) -- (341,87.81) ;

\draw (298,235) node [anchor=north west][inner sep=0.75pt]    {$\underbrace{...n...}$};
\end{tikzpicture}
\caption{The wiring diagram $\WD_s$ for stacking elements of $\arrows(\uparrow_n)$.}\label{WD for stacking}
\end{figure}

The following fact is established in \cite[Section 3.2]{BND:WKO2}:
\begin{prop}
The stacking product and arrow diagram coproduct make $\arrows(\uparrow_n)$ into a Hopf algebra, and, by the Milnor--Moore theorem \cite{MR174052}, $\arrows(\uparrow_n)$ is isomorphic to the completed universal enveloping algebra over its primitive elements. 
\end{prop}

Note that arrow diagrams on an arbitrary skeleton do not assemble into a Hopf algebra. However, the following lemma is a straightforward application of the $\textnormal{VI}$ relation:

\begin{lemma}\label{lemma: VI}\cite[Lemma 4.7]{BND:WKO2}
There is a canonical isomorphism of vector spaces \[\arrows(\vertex)\cong \arrows(\uparrow_2).\]  Therefore, the Hopf algebra structure on $\arrows(\uparrow_2)$ can be pulled back along this isomorphism to give a Hopf algebra structure on $\arrows(\vertex).$
\end{lemma}

\begin{example}
The product of two arrow diagrams $A$ and $B$ in $\arrows(\vertex)$ can be realised by a sequence of circuit algebra compositions and unzip, as shown in Figure \ref{fig:vertexcomp}. 
First compose $A$ with a vertex of the opposite orientation, and compose with $B$ (circuit algebra operations). Then unzip the middle edge. This product, together with the coproduct of Definition~\ref{def: coprod} makes $\arrows(\vertex)$ into a Hopf algebra.

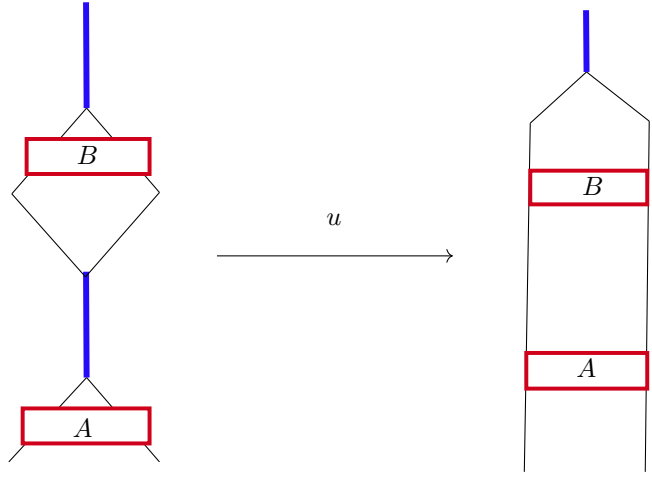
\begin{figure}[h]
\[
\begin{tikzpicture}[x=0.75pt,y=0.75pt,yscale=-1,xscale=1]

\draw    (90,266) -- (53.25,223.41) ;
\draw    (14,266) -- (53.25,223.41) ;
\draw [color={rgb, 255:red, 38; green, 10; blue, 248 }  ,draw opacity=1 ][line width=2.25]    (52.97,170) -- (53.25,223.41) ;
\draw    (337,94) -- (305.39,69.16) ;
\draw    (277,95) -- (305.39,69.16) ;
\draw [color={rgb, 255:red, 38; green, 10; blue, 248 }  ,draw opacity=1 ][line width=2.25]    (305.15,38) -- (305.39,69.16) ;
\draw    (15.52,130.69) -- (52.7,172.61) ;
\draw    (90,130) -- (52.7,172.61) ;
\draw    (90,130) -- (53.25,87.41) ;
\draw    (15.52,130.69) -- (53.25,87.41) ;
\draw [color={rgb, 255:red, 38; green, 10; blue, 248 }  ,draw opacity=1 ][line width=2.25]    (52.97,34) -- (53.25,87.41) ;
\draw  [color={rgb, 255:red, 208; green, 2; blue, 27 }  ,draw opacity=1 ][fill={rgb, 255:red, 255; green, 255; blue, 255 }  ,fill opacity=1 ][line width=1.5]  (21,239) -- (85,239) -- (85,256.67) -- (21,256.67) -- cycle ;
\draw  [color={rgb, 255:red, 208; green, 2; blue, 27 }  ,draw opacity=1 ][fill={rgb, 255:red, 255; green, 255; blue, 255 }  ,fill opacity=1 ][line width=1.5]  (23,103) -- (85,103) -- (85,120.67) -- (23,120.67) -- cycle ;
\draw[->]    (119,162) -- (238,162) ;
\draw    (277,95) -- (274,271) ;
\draw    (337,94) -- (335,271) ;
\draw  [color={rgb, 255:red, 208; green, 2; blue, 27 }  ,draw opacity=1 ][fill={rgb, 255:red, 255; green, 255; blue, 255 }  ,fill opacity=1 ][line width=1.5]  (277,119) -- (336,119) -- (336,136) -- (277,136) -- cycle ;
\draw  [color={rgb, 255:red, 208; green, 2; blue, 27 }  ,draw opacity=1 ][fill={rgb, 255:red, 255; green, 255; blue, 255 }  ,fill opacity=1 ][line width=1.5]  (275,211) -- (336,211) -- (336,229) -- (275,229) -- cycle ;

\draw (173,139.4) node [anchor=north west][inner sep=0.75pt]    {$u$};
\draw (45,243) node [anchor=north west][inner sep=0.75pt]    {$A$};
\draw (299,213) node [anchor=north west][inner sep=0.75pt]    {$A$};
\draw (47,105) node [anchor=north west][inner sep=0.75pt]    {$B$};
\draw (302,121) node [anchor=north west][inner sep=0.75pt]    {$B$};

\end{tikzpicture}
\] 
\caption{Stacking multiplication realised on a vertex.}\label{fig:vertexcomp}
\end{figure}

\end{example}

\subsection{The Kashiwara--Vergne spaces}\label{sec: AT spaces}

In this section, we review the relationship between the circuit algebra of arrow diagrams and the spaces where the Kashiwara--Vergne equations ``live'', based on \cite[Section 3.2]{BND:WKO2} and \cite{AT12}. 

Throughout, we will write $\lie_n$ for the degree completed \emph{free Lie algebra} generated by $x_1,\ldots, x_n$. In the case of $n=1,2,3$ we will often write the generators as $x,y,z$ to reduce notational clutter.  We let $\widehat{\mathfrak{ass}}_n:=\mathbb{Q}\left<\left<x_1,\ldots,x_n\right>\right>$ denote the degree completed free associative algebra generated by the same symbols.  
As a completed Hopf algebra, $\widehat{\mathfrak{ass}}_n$ is the universal enveloping algebra of $\lie_n$: $\widehat{U}(\lie_n)=\widehat{\mathfrak{ass}}_n.$ Degree completions throughout ensure that the group-like elements in $\widehat{U}(\lie_n)=\widehat{\mathfrak{ass}}_n$ are identified with the group $\exp(\lie_n)$.

 A \emph{tangential derivation} on $\lie_n$ is a derivation $u$ of $\lie_n$ which acts on the generators $x_i$ by $u(x_i)=[x_i,a_i]$, for some $a_i\in\lie_n$. For this reason, we write tangential derivations as a tuple of Lie words $u:=(a_1,\ldots,a_n)$ (\cite[Definition 3.2]{AT12}).   The collection of all tangential derivations of $\lie_n$ forms a Lie algebra which we denote by $\tder_n$, where the bracket $[u,v]$ is given by $[u,v](x_k):=u(v(x_k))-v(u(x_k))$; see \cite[Proposition 3.4]{AT12}. 

\begin{example}\label{example: tan der t}
	The tuple $t^{1,2}=(y,x)$ is a tangential derivation of $\lie_2$, given by $t^{1,2}(x)=[x,y]$ and $t^{1,2}(y)=[y,x]$.  
\end{example}  

The tuple description of tangential derivations defines a map $(\lie_n)^{\oplus n} \to \tder_n$, whose kernel is generated by the tuples $(0,...,0,x_i,0,...,0)$, where $x_i$ appears in the $i$th component. Therefore, there is an isomorphism $(\lie_n)^{\oplus n}\cong\tder_n \oplus \mathfrak{a}_n$, where $\mathfrak{a}_{n}$ is the $n$-dimensional abelian Lie algebra generated by $\{x_1,...,x_n\}$.

Both $\lie_n$  and $\widehat{\ass}_n$ are naturally graded. Explicitly,  $\lie_n=\prod_{k\geq 1}\lie_n^k$ where $\lie_n^k$ is spanned by Lie words with $k$ letters. For example,  $[x, z]\in\lie_3^2$ and $[y,[x,z]]\in\lie_3^3$. Similarly, $\widehat{\mathfrak{ass}}_n=\prod_{k\geq 1}\widehat{\mathfrak{ass}}_n^k$. This grading descends to the vector space of \emph{cyclic words}.

\begin{definition}\label{def: cyc}
The complete graded vector space of \emph{cyclic words} is the linear quotient  $$\cyc_n := \widehat{\mathfrak{ass}}_n/\left<ab-ba\mid \forall a,b\in\widehat{\mathfrak{ass}}_n\right> = \widehat{\mathfrak{ass}}_n/[\widehat{\mathfrak{ass}}_n,\widehat{\mathfrak{ass}}_n].$$  We denote the natural projection map by $\trace:\widehat{\mathfrak{ass}}_n\rightarrow\cyc_n$. 
\end{definition} 

The action of $\tder_n$ on $\lie_n$ extends to $\widehat{\ass}_n$  by the Leibniz rule, and in turn descends to $\cyc_n$. We illustrate this in the following example, borrowed from \cite[Example 3.18]{AT12}:

\begin{example}\label{ex:tdercyc}
Let $u=(y,0)\in \tder_2$ and $a=\trace(xy)\in \cyc^2_2$. Then $$u\cdot a=\trace(u(x)\,y+x\,u(y))=\trace([x,y]y)=\trace(xy^2-yxy)=0.$$ 
\end{example}

In \cite[Section 3.2]{BND:WKO2}, primitive elements of $\arrows(\uparrow_n)$ are mapped isomorphically to tangential derivations and cyclic words in the following way. Let $D$ be either a single tree or single wheel arrow graph on a skeleton of $n$ vertical strands. Label the skeleton strands with the generators $x_1, x_2,...,x_n$.
\begin{enumerate} 
\item If $D$ is a tree, we construct an element of $(\lie_n)^{\oplus n}$. The tree determines a Lie word by reading the generator corresponding to the skeleton strand of each leaf (tail) of the tree, and combining these with brackets corresponding to each trivalent vertex; see Example~\ref{example: arrow diagram that gives an element of tder} below. The placement of the \emph{root} (head) of the tree determines which coordinate of $(\lie_n)^{\oplus n}$ the Lie word is placed in, as in the example. 
\item If $D$ is a wheel, $D$ determines a cyclic word in the letters $x_i$, which is determined by the placement of the ``spokes'' of the wheel on the skeleton, and the orientation of the wheel. 
\end{enumerate}

\begin{example} \label{example: arrow diagram that gives an element of tder}
Consider the element $t^{1,2}=(y,x)\in\tder_2$ from Example~\ref{example: tan der t}.  The arrow diagram corresponding to $t^{1,2}$ is depicted on the left in Figure~\ref{fig: tder arrow diagrams}. Similarly, the tree arrow diagram on the right of Figure~\ref{fig: tder arrow diagrams} corresponds to $u=(0,[[x,y],z],0)\in\tder_3$. The wheel in Figure~\ref{fig:wheel example} corresponds to the cyclic word $\trace(x^2yz) \in\cyc_3^4$. 
\end{example}

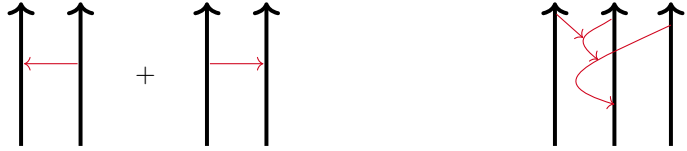
\begin{figure}[h]
\begin{tikzpicture}[x=0.75pt,y=0.75pt,yscale=-.5,xscale=.75]

\draw [<-, line width=1.5]    (75,80) -- (75,225) ;
\draw [<-, line width=1.5]    (115,80) -- (115,225) ;
\draw [<-, line width=1.5]    (200,80) -- (200,225) ;
\draw [<-, line width= 1.5]    (240,80) -- (240,225) ;
\draw [<-,color={rgb, 255:red, 208; green, 2; blue, 27 }  ,draw opacity=1 ]   (77,142) -- (113,142) ;
\draw [->,color={rgb, 255:red, 208; green, 2; blue, 27 }  ,draw opacity=1 ]   (202,142) -- (238,142) ;
\draw [<-, line width=1.5]    (434,80) -- (434,225) ;
\draw [<-, line width=1.5]    (474,80) -- (474,225) ;
\draw [<-, line width=1.5]    (512,80) -- (512,225) ;
\draw [->][ color={rgb, 255:red, 208; green, 2; blue, 27 }  ,draw opacity=1 ]   (435,92) -- (453,116) ;
\draw [->, color={rgb, 255:red, 208; green, 2; blue, 27 }  ,draw opacity=1 ]     (512,103) .. controls (460,140) and (420,159) .. (474,183) ;
\draw [->] [ color={rgb, 255:red, 208; green, 2; blue, 27 }  ,draw opacity=1 ]    (472,97) .. controls (459,110) and (442,114) .. (463,138) ;

\draw (150,142) node [anchor=north west][inner sep=0.75pt]    {$+$};

\end{tikzpicture}
\caption{The arrow diagram on the left corresponds to $t^{1,2}=(y,x)\in \tder_2$; the diagram on the right corresponds to $u=(0,[[x,y],z],0)\in\tder_3$.}\label{fig: tder arrow diagrams}
\end{figure}

\begin{figure}[h]

\begin{tikzpicture}[x=0.75pt,y=0.75pt,yscale=-.5,xscale=.5]

\draw [<-, line width=1.5]    (220,80) -- (220,230) ;
\draw [<-, line width=1.5]    (275,80) -- (275,230) ;
\draw [<-, line width=1.5]    (330,80) -- (330,230) ;
\draw  [color={rgb, 255:red, 208; green, 2; blue, 27 }  ,draw opacity=1 ] (247.45,134.69) .. controls (255.55,134.52) and (262.26,140.96) .. (262.42,149.07) .. controls (262.59,157.18) and (256.15,163.88) .. (248.04,164.04) .. controls (239.94,164.21) and (233.23,157.77) .. (233.07,149.66) .. controls (232.9,141.56) and (239.34,134.85) .. (247.45,134.69) -- cycle ;
\draw [<-, color={rgb, 255:red, 208; green, 2; blue, 27 }  ,draw opacity=1 ] (240,160) -- (223,175) ;
\draw [ ->, color={rgb, 255:red, 208; green, 2; blue, 27 }  ,draw opacity=1 ] (220,120) -- (240,138) ;
\draw [<-, color={rgb, 255:red, 208; green, 2; blue, 27 }  ,draw opacity=1 ]    (253,162) -- (273,175) ;
\draw [->, color={rgb, 255:red, 208; green, 2; blue, 27 }  ,draw opacity=1 ]   (328,110) -- (260,140) ;

\end{tikzpicture}
\caption{The cyclic word $\trace(x^2yz)$ as an arrow diagram. We use the convention that wheels are oriented couterclockwise, unless otherwise marked.}\label{fig:wheel example}
\end{figure}
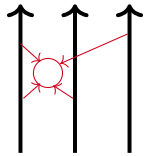

The subspace of $\arrows(\uparrow_n)$ spanned by primitive elements, $\mathsf{P}(\uparrow_n)$, forms a Lie algebra with the bracket given by the algebra commutator of $\arrows(\uparrow_n)$. There is a Lie algebra isomorphism\footnote{The isomorphism arises from a split short exact sequence of Lie algebras $$0\to \cyc_n \to \mathsf{P}(\uparrow)_n \to \tder_n\to 0.$$ The map $\mathsf{P}(\uparrow)_n \to \tder_n$ is as explained in (1) above; the splitting map depends on a choice in placing the ``heads'' of trees in relation to their ``tails'' on the skeleton strand. See \cite[Proposition 3.19]{BND:WKO2} for details.} $\mathsf{P}(\uparrow_n)\cong (\tder_n\oplus \mathfrak{a}_n)\ltimes \cyc_n$ \cite[Proposition 3.19]{BND:WKO2}. In the semidirect product, $\mathfrak{a}_n$ is central and $\tder_n$ acts on $\cyc_n$ as in Example~\ref{ex:tdercyc}, which in light of the discussion above can be realised in $\arrows(\uparrow_n)$, as illustrated in Figure~\ref{fig:tder_action_cyc}.
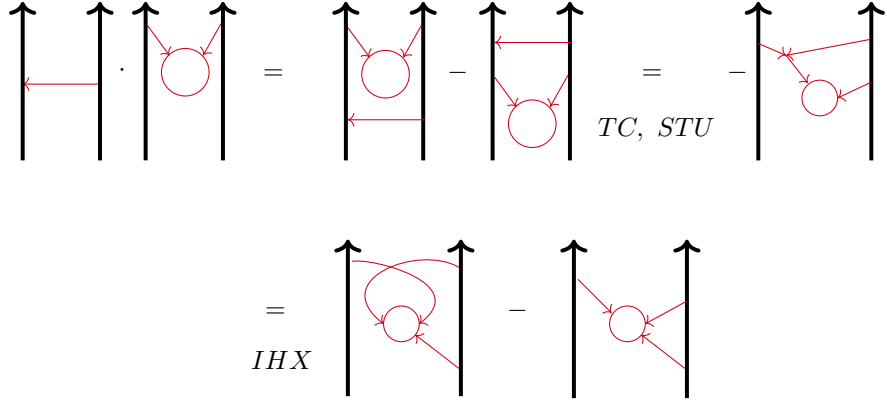
\begin{figure}[h]

\begin{tikzpicture}[x=0.75pt,y=0.75pt,yscale=-1,xscale=1]

\draw [->] [color={rgb, 255:red, 208; green, 2; blue, 27 }  ,draw opacity=1 ]   (99,52) -- (110,67) ;
\draw[<-] [line width=1.5]    (98,40) -- (98,120) ;
\draw[<-] [line width=1.5]    (137,40) -- (137,120) ;
\draw [<-] [color={rgb, 255:red, 208; green, 2; blue, 27 }  ,draw opacity=1 ]   (127,67) -- (136,51) ;
\draw  [color={rgb, 255:red, 208; green, 2; blue, 27 }  ,draw opacity=1 ] (106,75.5) .. controls (106,68.87) and (111.37,63.5) .. (118,63.5) .. controls (124.63,63.5) and (130,68.87) .. (130,75.5) .. controls (130,82.13) and (124.63,87.5) .. (118,87.5) .. controls (111.37,87.5) and (106,82.13) .. (106,75.5) -- cycle ;
\draw[->] [color={rgb, 255:red, 208; green, 2; blue, 27 }  ,draw opacity=1 ]   (200,53) -- (211,68) ;
\draw[<-] [line width=1.5]    (199,40) -- (199,120) ;
\draw [<-][line width=1.5]    (238,40) -- (238,120) ;
\draw [<-] [color={rgb, 255:red, 208; green, 2; blue, 27 }  ,draw opacity=1 ]   (228,68) -- (237,52) ;
\draw  [color={rgb, 255:red, 208; green, 2; blue, 27 }  ,draw opacity=1 ] (207,76.5) .. controls (207,69.87) and (212.37,64.5) .. (219,64.5) .. controls (225.63,64.5) and (231,69.87) .. (231,76.5) .. controls (231,83.13) and (225.63,88.5) .. (219,88.5) .. controls (212.37,88.5) and (207,83.13) .. (207,76.5) -- cycle ;
\draw[<-] [color={rgb, 255:red, 208; green, 2; blue, 27 }  ,draw opacity=1 ]   (36,81.5) -- (75,81.5) ;
\draw[<-] [line width=1.5]    (36,40) -- (36,120) ;
\draw[<-] [line width=1.5]    (75,40) -- (75,120) ;
\draw[<-] [color={rgb, 255:red, 208; green, 2; blue, 27 }  ,draw opacity=1 ]   (200,99.5) -- (239,99.5) ;
\draw [->] [color={rgb, 255:red, 208; green, 2; blue, 27 }  ,draw opacity=1 ]   (274,78) -- (285,93) ;
\draw[<-] [line width=1.5]    (273,40) -- (273,120) ;
\draw[<-] [line width=1.5]    (312,40) -- (312,120) ;
\draw[<-] [color={rgb, 255:red, 208; green, 2; blue, 27 }  ,draw opacity=1 ]   (302,93) -- (311,77) ;
\draw  [color={rgb, 255:red, 208; green, 2; blue, 27 }  ,draw opacity=1 ] (281,101.5) .. controls (281,94.87) and (286.37,89.5) .. (293,89.5) .. controls (299.63,89.5) and (305,94.87) .. (305,101.5) .. controls (305,108.13) and (299.63,113.5) .. (293,113.5) .. controls (286.37,113.5) and (281,108.13) .. (281,101.5) -- cycle ;
\draw [<-][color={rgb, 255:red, 208; green, 2; blue, 27 }  ,draw opacity=1 ]   (274,60.5) -- (313,60.5) ;
\draw[->] [color={rgb, 255:red, 208; green, 2; blue, 27 }  ,draw opacity=1 ]   (421,67) -- (432,81) ;
\draw[<-] [line width=1.5]    (407,40) -- (407,120) ;
\draw[<-] [line width=1.5]    (464,40) -- (464,120) ;
\draw [<-] [color={rgb, 255:red, 208; green, 2; blue, 27 }  ,draw opacity=1 ]   (447,88.5) -- (464,80.5) ;
\draw  [color={rgb, 255:red, 208; green, 2; blue, 27 }  ,draw opacity=1 ] (429,88.5) .. controls (429,83.53) and (433.03,79.5) .. (438,79.5) .. controls (442.97,79.5) and (447,83.53) .. (447,88.5) .. controls (447,93.47) and (442.97,97.5) .. (438,97.5) .. controls (433.03,97.5) and (429,93.47) .. (429,88.5) -- cycle ;
\draw[->] [color={rgb, 255:red, 208; green, 2; blue, 27 }  ,draw opacity=1 ]   (407,61) -- (421,67); 
\draw [<-] [color={rgb, 255:red, 208; green, 2; blue, 27 }  ,draw opacity=1 ]   (421,67) -- (463,59) ;
\draw[<-] [line width=1.5]    (200,160) -- (200,239) ;
\draw[<-] [line width=1.5]    (257,160) -- (257,239) ;
\draw [->] [color={rgb, 255:red, 208; green, 2; blue, 27 }  ,draw opacity=1 ]   (316,180) -- (333,197) ;
\draw[<-] [line width=1.5]    (314,160) -- (314,240) ;
\draw [<-] [line width=1.5]    (371,160) -- (371,240) ;
\draw [<-] [color={rgb, 255:red, 208; green, 2; blue, 27 }  ,draw opacity=1 ]   (350,202.5) -- (371,191) ;
\draw  [color={rgb, 255:red, 208; green, 2; blue, 27 }  ,draw opacity=1 ] (332,202.5) .. controls (332,197.53) and (336.03,193.5) .. (341,193.5) .. controls (345.97,193.5) and (350,197.53) .. (350,202.5) .. controls (350,207.47) and (345.97,211.5) .. (341,211.5) .. controls (336.03,211.5) and (332,207.47) .. (332,202.5) -- cycle ;
\draw [<-] [color={rgb, 255:red, 208; green, 2; blue, 27 }  ,draw opacity=1 ]   (348,208) -- (370,225) ;
\draw  [color={rgb, 255:red, 208; green, 2; blue, 27 }  ,draw opacity=1 ] (218,202.5) .. controls (218,197.53) and (222.03,193.5) .. (227,193.5) .. controls (231.97,193.5) and (236,197.53) .. (236,202.5) .. controls (236,207.47) and (231.97,211.5) .. (227,211.5) .. controls (222.03,211.5) and (218,207.47) .. (218,202.5) -- cycle ;
\draw [<-] [color={rgb, 255:red, 208; green, 2; blue, 27 }  ,draw opacity=1 ]   (234,208) -- (256,225) ;
\draw[->] [color={rgb, 255:red, 208; green, 2; blue, 27 }  ,draw opacity=1 ]   (256,174) .. controls (236,162) and (189,180) .. (218,202.5) ;
\draw [<-] [color={rgb, 255:red, 208; green, 2; blue, 27 }  ,draw opacity=1 ]   (236,202.5) .. controls (263,182) and (212,169) .. (202,171) ;

\draw (83,70) node [anchor=north west][inner sep=0.75pt]    {$\cdot $};
\draw (156,73) node [anchor=north west][inner sep=0.75pt]    {$=$};
\draw (347,73) node [anchor=north west][inner sep=0.75pt]    {$=$};
\draw (325,97.4) node [anchor=north west][inner sep=0.75pt]    {$TC,\ STU$};
\draw (390,70) node [anchor=north west][inner sep=0.75pt]    {$-$};
\draw (279,190) node [anchor=north west][inner sep=0.75pt]    {$-$};
\draw (157,193) node [anchor=north west][inner sep=0.75pt]    {$=$};
\draw (150,215) node [anchor=north west][inner sep=0.75pt]    {$IHX$};
\draw (249,70) node [anchor=north west][inner sep=0.75pt]    {$-$};

\end{tikzpicture}

\caption{An illustration of the action of $\tder_n$ on $\cyc_n$ from Example~\ref{ex:tdercyc} realised as arrow diagrams.}\label{fig:tder_action_cyc}
\end{figure}

 Since, by the Milnor--Moore Theorem, $\arrows(\uparrow_n)$ is the universal enveloping algebra of $\mathsf{P}(\uparrow_n)$, this implies the following:

\begin{prop}\label{prop:arrows are hopf algebra}
There is an isomorphism of Hopf algebras $$\Upsilon: \arrows(\uparrow_n)\stackrel{\cong}{\longrightarrow} \widehat{U}((\tder_n\oplus \mathfrak{a}_n) \ltimes \cyc_n).$$ 
\end{prop} 

It is important to keep in mind that at the level of primitive elements $\Upsilon$ identifies $\tder_n\oplus \mathfrak{a}_n$ with connected tree diagrams, and $\cyc_n$ with wheels.

\begin{remark}\label{ex: action of arrow diag on lie}
In \cite[Proof of Proposition 19]{BND:WKO2} the action of $\tder_n$ on $\lie_n$ is described in diagrammatic terms; we give a summary here.
Let $D\in\mathsf{P}(\uparrow_n)$ be an arrow diagram consisting of a single tree, corresponding to $u=(0,\ldots, a_i,\ldots, 0)\in\tder_n$. The tails of $D$ correspond to the letters in the Lie word $a_i$ and the root of the tree (the single arrow head) is on the $i$th strand of $
\arrows(\uparrow_n)$. 

We consider $D$ as an element of $\arrows(\uparrow_{n+1})$ by adding an empty skeletal strand to the right. The free Lie algebra $\lie_n$ embeds in $\arrows(\uparrow_n)$, by representing $x_j$ as an arrow from strand $j$ to the extra strand. Then $D$ acts on $x_j$ via the commutator, i.e. the difference of the stacking products $x_jD-Dx_j$. Using the TC and STU relations, one can verify that this commutator is in the image of $\lie_n$. 

We write $\TAut_n:=\exp(\tder_n)$ to denote the group associated to the Lie algebra $\tder_n$. As a 
group, $\TAut_n$ consists of ``basis-conjugating'' automorphisms of $\exp(\lie_n)$:  for any $e^u \in \TAut_n$ there exists an element $e^{a_i}\in\exp(\lie_n)$ where $e^u\cdot e^{x_i}=e^{-a_i}e^{x_i}e^{a_i}$. 

After exponentiation, we obtain that the following diagram commutes:
\begin{equation}
\begin{tikzcd}
\TAut_n \times \exp(\lie_n)  \arrow[d, hookrightarrow, swap, "\Upsilon^{-1}\times \iota"]\arrow[rr, "\ad"]  && \exp(\lie_n) \arrow[d, hookrightarrow, "\iota"]
\\
\arrows(\uparrow_{n+1})\times \arrows(\uparrow_{n+1}) \arrow[rr, "\text{conj}"] && \arrows(\uparrow_{n+1})  
\end{tikzcd}
\end{equation} 

Here, given $e^u\in \TAut_n$, by $\Upsilon^{-1}(u)$ we mean $\Upsilon^{-1}(((u+0),0))$ embedded on the first $n$ strands of $\arrows(\uparrow_n)$. The map $\iota$ is the exponential of the embedding of $\lie_n$ above. Ad denotes the adjoint action and ``conj'' denotes conjugation of the second component by the first.
\end{remark}

Arrow diagrams with an arrow head adjacent to a cap vanish by the $\textnormal{CP}$ relation. Thus, when the skeleton consists only of capped strands, arrow heads can be eliminated entirely by successive applications of $\textnormal{STU}$ and $\textnormal{CP}$ relations. Arrow diagrams with no arrow heads can always be expressed as a linear combination of only {\em wheels}, which in turn can be encoded as cyclic words. This leads to the following lemma, which is a straightforward generalisation of \cite[Lemma 4.6]{BND:WKO2}.

\begin{lemma}\label{hopf_vertex}
There is an isomorphism of graded vector spaces $\kappa: \arrows(\upcap_n) \stackrel{\cong}{\longrightarrow} \cyc_n/\cyc_n^1$. Here $\cyc_n^1$ denotes the degree 1 component of cyclic words.
\end{lemma} 

The reason for factoring out the degree $1$ component in this lemma is the $RI$ relation, see Figure~\ref{fig:wheelRI}. 

\begin{remark}\label{rmk: wheels power series} Note that cyclic words in one letter are power series, so we will also write 
$\kappa: \arrows(\upcap_1)\stackrel{\cong}{\longrightarrow} \mathbb Q[[\xi]]/\langle\xi \rangle$, where the quotient is understood linearly. In addition, since $\cyc_n\cong \trace(\mathbb Q\langle\langle x_1,x_2,...,x_n \rangle\rangle)$, we have, for example $\arrows(\upcap_2)\cong \trace(\mathbb{Q}\langle\langle \xi, \zeta \rangle\rangle) /\langle \xi, \zeta\rangle$. 
\end{remark}

\begin{remark}\label{ex: action on cyc 2}
The $\mathbb{Q}$-vector space $\arrows(\upcap_2)$ is a left $\arrows(\uparrow_2)$-module with the action given by stacking: \[\arrows(\uparrow_2)\times\arrows(\upcap_2)\xrightarrow{\text{stack}} \arrows(\upcap_2).\]
This stacking action is compatible with the natural action of $\tder_2$ on $\cyc_2$ described in Example~\ref{ex:tdercyc}, in the sense that the following diagram commutes:
\begin{equation}
\begin{tikzcd}
\tder_n \times \cyc_n  \arrow[d, hookrightarrow, swap, "\Upsilon^{-1}\times \kappa^{-1}"]\arrow[rr, "\cdot"]  && \cyc_n \arrow[d, hookrightarrow, "\kappa^{-1}"]
\\
\arrows(\uparrow_n) \times \arrows(\upcap_n) \arrow[rr, "\text{stack}"] && \arrows(\upcap_n)  
\end{tikzcd}
\end{equation} 

It is non-trivial to see why a stacking action agrees with a commutation action. In short, this is because $\tder_n$ is embedded in $\arrows(\uparrow_n)$ as {\em trees}, and the term of the commutator where the tree is adjacent to the cap vanishes by the $\textnormal{CP}$ relation. For a more thorough explanation, see \cite[Remark 3.24]{BND:WKO2}.

The action of $\tder_n$ on $\cyc_n$ can be formally exponentiated to a conjugation action of $\TAut_n$ on $\exp(\cyc_n)$. This action commutes with the stacking action of $\arrows(\uparrow_n)$ on $\arrows(\upcap_n)$ for the same reason as above:
\begin{equation}
\begin{tikzcd}
\TAut_n \times \exp(\cyc_n)  \arrow[d, hookrightarrow, swap, "\Upsilon^{-1}\times \kappa^{-1}"]\arrow[rr, "\cdot"]  && \exp(\cyc_n) \arrow[d, hookrightarrow, "\kappa^{-1}"]
\\
\arrows(\uparrow_n) \times \arrows(\upcap_n) \arrow[rr, "\text{stack}"] && \arrows(\upcap_n)  
\end{tikzcd}
\end{equation} 
\end{remark}

\subsubsection{Cosimplicial Structure} \label{sec: cosimplicial structure} 
Tangential derivations admit a semi-cosimplicial structure, as described in \cite[Section 3.2]{AT12}. That is, for $i=0,1,...,n,$ there are Lie algebra homomorphisms, called {\em coface maps} $\delta_i:\tder_{n-1}\rightarrow \tder_n$ defined by sending the tangential derivation $u=(a_1,\ldots, a_{n-1})$ to $\delta_i(u)=(a_1,\ldots, 0,\ldots, a_{n-1})$ in $\tder_n$, where $0$ is inserted in the $i$th entry of $\delta_i(u)$, and the subsequent components are shifted.

To simplify notation, we will suppress the coface maps and use a superscript notation with $\delta_i(u):= u^{1,\ldots, i-1, i+1,\ldots, n}.$ For example, if $u=(a_1(x,y),a_2(x,y))\in\tder_2$, then $u^{2,3}=(0,a_1(y,z),a_2(y,z)) \in \tder_3$. Diagrammatically, that is, after applying $\Upsilon^{-1}$, this corresponds to inserting an empty strand in position $i$.

A second type of coface map, denoted by double entries in the superscripts, corresponds to an ``unzip-style'' strand-doubling. 
For instance, given $u\in \tder_2$ as above, $$u^{1,23}=(a_1(x,y+z),a_2(x,y+z),a_2(x,y+z))\in \tder_3.$$
Diagrammatically, $\Upsilon^{-1}(u^{1,...,i(i+1),i+2,...,n})$ is obtained from $\Upsilon^{-1}(u)$ by doubling the $i$th strand and replacing any arrows ending on the $i$th strand with a sum of two arrows, as in the definition of the unzip operation, Figure~\ref{fig: unzip for arrows}.

\subsection{Divergence and the Jacobian}\label{sec: div and Jacobian}
Note that each element $a\in\widehat{\mathfrak{ass}}_n$ has a unique decomposition $$a=a_0+\partial_1(a)x_1+\ldots +\partial_n(a)x_n=a_0+\Sigma_{i=1}^{n}\partial_i(a)x_i$$ for some $a_0\in\mathbb{Q}$ and $\partial_i(a)\in\widehat{\mathfrak{ass}}_n$ for $1\leq i\leq n$. In practice, $\partial_i$ picks out the words of a sum which end in $x_i$ and deletes their last letter $x_i$, as well as all other words. This enables the definition of the non-commutative \emph{divergence}, a $1$-cocycle of $\tder_n$ (\cite[Proposition 3.20]{AT12}):

 \begin{definition}\label{def AT div}
The non-commutative \emph{divergence} map $j:\tder_n\rightarrow \cyc_n$ is a linear map
 defined by $$j(u):=\trace(\Sigma_{i=1}^{n} \partial_i(a_i)x_i)$$ where $u$ is the derivation given by the tuple of Lie words $(a_1,\ldots,a_n)$.  
 \end{definition} 
 
The divergence map is a $1$-cocycle of the Lie algebra $\tder_n$, that is: $$j([u,v])=u\cdot j(v)-v\cdot j(u),$$ where the $\cdot$ notation denotes the natural $\tder_n$ action on $\cyc_n$ is in Example~\ref{ex:tdercyc}.


\begin{example}\label{example: div} 
The tangential derivation $t^{1,2} = (y, x) \in\tder_2$ described in Example~\ref{example: tan der t} has vanishing divergence since $\partial_x(y)=\partial_y(x)=0$ so $j(t^{1,2})=\trace\left(\partial_x(y) x + \partial_y(x)y\right) = 0$. 
\end{example}

Integrating the divergence cocycle leads to a linear map called the non-commutative \emph{Jacobian}.

 \begin{definition} 
 The non-commutative \emph{Jacobian} map\footnote{Our notation $(j,J)$ matches that of \cite{ATE10}, and corresponds to $(div,j)$ in \cite{AT12} and \cite[Section 3.3]{BND:WKO2}.} $$J:\TAut_n\rightarrow\cyc_n$$ is given by setting 
 \[J(1)=0 \quad \text{and} \quad \frac{d}{dt}\Big|_{t=0}J(e^{tu}g)=j(u)+u\cdot J(g)\] for $g\in \TAut_n$ and $u\in\tder_n$. 
 \end{definition} 
 
 The map $J$ is a group $1$-cocycle, that is, for any $g,h\in\TAut_n$, $$J(gh)=J(g)+g\cdot J(h).$$    
 
Finally, the exponential Jacobian, denoted  $\mathcal J$, is the map $\mathcal{J}:\TAut_n\rightarrow \exp(\cyc_n)$ given by $\mathcal{J}(e^u)=e^{J(e^u)}$, for $u\in \tder_n$. 

The exponential Jacobian has a diagrammatic interpretation captured by the  \emph{adjoint} operation $A_e$, defined in Definition~\ref{def:operations on arrows}. For $a\in\arrows(\uparrow_n)$ let $a^*$ denote $A_1...A_n(a)$, that is, the adjoint operation applied to every strand of $a$. Then, for any $u\in\tder_n$, we have $\Upsilon^{-1}(\mathcal J(e^u))=\Upsilon^{-1}(e^u)(\Upsilon^{-1}(e^u))^*$. See \cite[Proposition 3.27]{BND:WKO2} for details.

\section{Homomorphic expansions}\label{sec: expansions}

Expansions are known in knot theory and cognate areas as ``universal finite type invariants''. An expansion is {\em homomorphic} if it respects any additional structure or operations possessed by a class of knotted objects. For example, an expansion for classical knots is the Kontsevich integral, which is {\em homomorphic} in the sense that it {\em respects} connected sum and cabling of knots (e.g. \cite{MR1237836}, \cite{DancsoKI}). 
Constructing homomorphic expansions is often difficult, for instance, different constructions of the Kontsevich integral involve complex analysis, Drinfeld associators, and perturbative Chern--Simons theory (also known as ``configuration space integrals''). We suggest the paper \cite{BN_Survey_Knot_Invariants} for an introduction to the theory of finite type invariants, however, a familiarity with the theory is not necessary for reading this paper. 

Our interest in homomorphic expansions stems from their capacity to translate problems from topology to quantum algebra and vice versa. A classical example of this, analogous to the results of this paper, is the description of Drinfeld associators and the Grothendieck-Teichm\"uller groups in terms of homomorphic expansions of parenthesised braids and their symmetries \cite{BNGT}.

In \cite{BND:WKO2}, Bar-Natan and the first author show that homomorphic expansions of $\wf$ are in one-to-one\footnote{Up to a minor technical condition on the value of the vertex given in Definition~\ref{def:vsmall}.} correspondence with solutions to the Kashiwara--Vergne equations.  As a consequence, the existence of homomorphic expansions for $w$-foams follows from the existence of solutions to the Kashiwara--Vergne conjecture. 

In this section we review useful classification criteria \cite{BND:WKO2} for homomorphic expansions of $w$-foams $Z:\wf \rightarrow \arrows$ (Proposition~\ref{prop:ExpansionEquations}). Furthermore, in Theorem~\ref{thm: expansions are isomorphisms} we show that homomorphic expansions $Z: \wf \to \arrows$ induce isomorphisms $\widehat{Z}:\hatwf\to \arrows$ of completed circuit algebras. This is a new contribution to the theory of $w$-foams, since \cite{BND:WKO2} did not consider prounipotent completions of $w$-foams. Combining Theorem~\ref{thm: expansions are isomorphisms} with the \cite{BND:WKO2} correspondence, it follows that that solutions of the Kashiwara--Vergne conjecture are in one-to-one correspondence with a class of isomorphisms $\widehat{Z}:\hatwf\to \arrows$ of completed circuit algebras.

\subsection{Homomorphic expansions and Kashiwara--Vergne solutions}
Homomorphic expansions of $w$-foams are circuit algebra homomorphisms $Z:\wf\rightarrow\arrows$ satisfying a universal property. The terminology for expansions comes from group theory, so we begin with the group theory context for motivation. 

\begin{example}\label{ex:expansion for groups}
Given a group $G$,  consider the group ring $ \mathbb{Q}[G]$, which is filtered by powers of its augmentation ideal $\calI=\text{ker}\left(\mathbb{Q}[G]\xrightarrow{\epsilon} \mathbb{Q}\right)$. Here, the map $\epsilon$ sends a linear combination of group elements to the sum of their coefficients. Define the complete \emph{associated graded algebra} of $\mathbb Q[G]$ to be the $\mathbb{Q}$-algebra $$A(G) :=\text{gr}(\mathbb{Q}[G]) = \prod_{m\geq 0} \calI^m/\calI^{m+1}.$$ 
An \emph{expansion} of the group $G$ is a map $Z : G \rightarrow A(G)$, such that the
linear extension of $Z$ to the group ring, $Z : \mathbb{Q}[G] \rightarrow A(G)$, is a filtration-preserving map of algebras and the induced map $$\text{gr}(Z): (\text{gr}(\mathbb{Q}[G] )= A(G)) \rightarrow (\text{gr}(A(G)) = A(G))$$
is the identity of $A(G)$.  The latter condition is equivalent to saying that the degree $m$ piece of $Z$ restricted
to $\calI^m$ is the projection onto $\calI^m/\calI^{m+1}$.  
\end{example}

As we move to homomorphic expansions of $w$-foams, we will require that they respect the skeleton index of the underlying circuit algebras:

\begin{definition}
Given two circuit algebras $\mathsf{V}$ and $\mathsf{W}$ with skeleta in $\calS$, a circuit algebra morphism $F:\mathsf{V}\rightarrow \mathsf{W}$ is said to be \emph{skeleton preserving} if it restricts to a homomorphism $F(s): \mathsf{V}(s)\rightarrow\mathsf{W}(s)$ for each $s\in\calS$. \end{definition}

\begin{definition}\label{def: homomorphic expansion}
A {\em homomorphic expansion} of $w$-foams is a circuit algebra homomorphism $Z:\wf\rightarrow\A$ such that its linear extension $Z: \mathbb{Q}[\wf] \to \A$ is a filtered, skeleton preserving homomorphism of linear circuit algebras for which $\gr Z = \id_{\A}$, and which intertwines all auxiliary operations with their associated graded counterparts. 
\end{definition}

\begin{definition}
A homomorphic expansion is {\em group-like} if the $Z$-values of generators are exponentials of primitive elements in $\arrows$ (group-like in the target space).  
\end{definition}

\subsubsection{A classification of homomorphic expansions of $\wf$}
Homomorphisms of finitely presented circuit algebras are determined by their values on the generators, and must satisfy the relations between those generators -- this is true broadly for any homomorphism of a finitely presented algebraic structure. 

Therefore, a homomorphic expansion of $\wf$ is determined by its values on the generators, and for a group-like expansion these values are exponentials of primitive elements in $\A$. By \cite[Theorem 3.30]{BND:WKO2}, any group-like homomorphic expansion of $w$-foams sends the crossing $\overcrossing$ to the exponential of a single arrow from the over-strand to the under-strand $e^{\rightarrowdiagram}$:
\[ Z(\overcrossing) = e^{\rightarrowdiagram} = R \in \A(\uparrow_2), \quad Z(\vertex)=e^v=V \in\A(\vertex)\cong\A(\uparrow_2), \] \[ \text{and} \quad Z(\upcap) =e^c=C  \in \A(\upcap)\cong \mathbb Q[[\xi]]/\langle \xi \rangle.\]

Since $Z$ is skeleton preserving, the value of $Z(\vertex)$ is an arrow diagram in $\arrows(\vertex)$. By Lemma ~\ref{prop:arrows are hopf algebra}, the VI relation on $\arrows$ induces an isomorphism of Hopf algebras $\arrows(\vertex)\cong \arrows(\uparrow_2)$. Under the isomorphism $\Upsilon: \arrows(\uparrow_2)\rightarrow \widehat{U}(\cyc_2\rtimes(\tder_2\oplus \mathfrak{a}_2))$ we can identify the value $Z(\vertex)\in\arrows(\uparrow_2)$ with $$\Upsilon(Z(\vertex))=e^be^\nu \in  \widehat{U}(\cyc_2 \rtimes (\tder_2\oplus \mathfrak{a}_2)),$$ where $b\in \cyc_2$ and $\nu \in \tder_2 \oplus \mathfrak a _2$. (We will soon assume that the $\mathfrak a _2$ component of $\nu$ is zero.)

Similarly, we recall from Lemma~\ref{hopf_vertex} that there is a linear isomorphism $\kappa: \A(\upcap)\to \mathbb Q[[\xi]]/\langle \xi \rangle$ from arrow diagrams on a single capped strand to the completed polynomial algebra understood as a graded vector space and factored out by linear terms. Therefore, the value of the cap can be described as the exponential of a power series $\kappa(c)\in \mathbb Q[[\xi]]/\langle \xi \rangle$.

Kashiwara--Vergne solutions are in fact in bijection with families of group-like homomorphic expansions of $\wf$; and in each family one representative has the following special property: 
\begin{definition}\label{def:vsmall}
We say that a homomorphic expansion is \emph{$v$-small} if the projection of $\log \Upsilon(Z(\vertex))$ onto $\mathfrak{a}_2$ is zero. 
\end{definition} 

We are now ready to recall the main theorem of \cite{BND:WKO2}. Note that Kashiwara--Vergne solutions will be formally defined in Definition~\ref{def:solkv}.

\begin{theorem}\label{expansions are solKV}\cite[Theorem 4.9]{BND:WKO2}
There is a one-to-one correspondence between the set of $v$-small, group-like homomorphic expansions $Z:\wf\xrightarrow{} \A$, and Kashiwara--Vergne solutions. 
\end{theorem}

\subsection{Characterisation of group-like homomorphic expansions}
The following classification for group-like homomorphic expansions of $w$-foams can be found in \cite[Section 4.3]{BND:WKO2}; note that it is not formally stated as a theorem there, but discussed in a short section. In this statement, we use the (diagrammatic) cosimplicial notation as explained in Section~\ref{sec: cosimplicial structure}.

\begin{prop}\label{prop:ExpansionEquations} A filtered, skeleton preserving homomorphism of circuit algebras $Z:\wf\xrightarrow{} \arrows$ is a v-small group-like homomorphic expansion of $\wf$ if and only if the values \[ Z(\overcrossing) = e^{\rightarrowdiagram} = R \quad Z(\vertex)=e^v = V \quad \text{and} \quad Z(\upcap) =e^c = C\] satisfy the following equations: 
\begin{equation}\label{R4} 
 V^{12}R^{(12)3}=R^{23}R^{13}V^{12} \tag{R4}
\end{equation}
\begin{equation}\label{U}
V\cdot A_1A_2(V)=1\tag{U}
\end{equation} 
\begin{equation}\label{C} V^{12}C^{12}=C^1C^2  \quad \textnormal{ in } \quad \A(\upcap_2).  \tag{C}\end{equation}
\end{prop}

\begin{proof}[Sketch of the proof:]
As $Z:\wf\rightarrow\A$ is a homomorphism of circuit algebras, it is uniquely determined
by its values on the generators. These values must satisfy the equations obtained from applying $Z$ to the relations of $\wf$, and the equations forced by the homomorphicity condition with respect to the auxiliary operations. The result is obtained by going through each of these conditions. Many of the equations obtained are tautologically true, due to the choice of $Z(\overcrossing) = e^{\rightarrowdiagram}$. The few which are not give the equations stated: the $R4$ relation with a strand moving under a vertex implies the equation $(R4)$; homomorphicity with respect to unzip implies the unitarity equation $(U)$; homomorphicity with respect to disc unzip implies the cap equation $(C)$.
\end{proof}

\subsection{Completing homomorphic expansions}
We end this section with a few important facts about extending homomorphic expansions to the prounipotent completion $\hatwf$. Recall from Proposition \ref{prop:AssocGradedOfCompletion} that $\hatwf$ is filtered, and its associated graded circuit algebra is canonically isomorphic to $\A$.

\begin{prop}\label{prop:Zhat} Any homomorphic expansion $Z: \wf \to \A$ induces a map on the completion $\widehat{Z}: \hatwf \to \A$, with $\gr \widehat{Z} = \id_A$. Furthermore, this correspondence is a bijection between the set of homomorphic expansions and filtered circuit algebra maps $W: \hatwf \to \A$ which respect auxiliary operations and for which $\gr W=\id_\A$.
\end{prop}

\begin{proof}The setup can be summarised in the following commutative diagram:
\[\begin{tikzcd}
\wf  \arrow[d, "\gamma"]\arrow[rrd, "Z"] 
\\
\hatwf \arrow[rr, dashed, "W"] && \A  
\end{tikzcd}\]
It is clear that if $W: \hatwf \to \A$ is a filtered circuit algebra map respecting auxiliary operations and $\gr W=\id_\A$, then it induces a homomorphic expansion $Z=W\circ \gamma$ via pre-composition with the canonical map $\gamma: \wf \rightarrow \hatwf$. It therefore remains to show that each homomorphic expansion $Z$ can be completed to a map $W$ with the necessary properties. 

Let $Z_{\leq m}: \mathbb{Q}[\wf] \to \A_{\leq m}$ denote the composition $p_{\leq m} \circ Z$, where $p_{\leq m}$ is the quotient map which truncates $\A$ at degree $m$. Note that by the universal property of expansions, $Z_{\leq m}$ restricted to $\calI^{m+1}$ is zero, and therefore it makes sense to talk about $Z_{\leq m}$ as a map on $\mathbb{Q}[\wf]/\calI^{m+1}$: 
$$Z_{\leq m}: \mathbb{Q}[\wf]/\calI^{m+1} \to \A_{\leq m} \hookrightarrow \A.$$
The maps $Z_{\leq m}$ for $ m\in \Z_{\geq 0}$ are compatible with the projections $\mathbb{Q}[\wf]/\calI^{m} \leftarrow \mathbb{Q}[\wf]/\calI^{m+1}$, and $\A_{\leq m} \leftarrow \A_{\leq m+1}$, and therefore, by the functoriality of inverse limits, induce a filtered circuit algebra map
$$W=\widehat{Z}: \hatwf \to \A.$$ 
By construction, $\gr \widehat{Z}=\id_\A$, $\widehat{Z}$ respects the auxiliary operations, and  $Z=W\circ \gamma$. Therefore, this is a bijective correspondence between homomorphic expansions $Z$ and maps $W: \hatwf \to \arrows$ with the stated properties.  
\end{proof}

\begin{theorem}\label{thm: expansions are isomorphisms}
For any homomorphic expansion $Z:\wf\xrightarrow{}\A$, the induced map $\widehat{Z}: \hatwf \to \A$ is an isomorphism of filtered, completed circuit algebras. 
\end{theorem}

\begin{proof}
Since $\widehat{Z}$ is a filtered circuit algebra homomorphism, we only need to prove that it is invertible. Note that by the universal property of expansions, $Z_{\leq m}: \mathbb{Q}[\wf]/\calI^{m+1} \to \A$ restricts to the identity map $\calI^{m}/\calI^{m+1} \to \A_m$. Let $\psi_m$ denote the composition of the projection onto degree $m$ with the identity map and the inclusion:
$$\psi_m: \A \twoheadrightarrow \A_m  \to \calI^m/\calI^{m+1} \hookrightarrow \mathbb{Q}[\wf]/\calI^{m+1}$$ 
By construction, both $\psi_m \circ Z_{\leq m}$ and $Z_{\leq m}\circ \psi_m$ restrict to the identity map on $\calI^m/\calI^{m+1}$ in $\mathbb{Q}[\wf]/\calI^{m+1}$ and in $\A$, respectively. This can be summarised in the following commutative diagram:
\[\begin{tikzcd}[cells={nodes={minimum height=2em}}]
& \A=\widehat{\A} \arrow[dl] \arrow[d] \arrow[dr] \arrow[drr, yshift=2pt] & & \\
\A_{\leq 0} \arrow[d, "\psi_0"] & \A_{\leq 1}  \arrow[l] \arrow[d, "\psi_1"] &  \A_{\leq 2} \arrow[l] \arrow[d, "\psi_2"]  &  \phantom{\A}\cdots \arrow[l] \\
\mathbb{Q}[wF]/I \arrow[u,xshift=-5pt,"Z_{\leq 0}"] & \mathbb{Q}[wF]/I^2  \arrow[l] \arrow[u,xshift=-5pt,"Z_{\leq 1}"] & \mathbb{Q}[wF]/I^3 \arrow[l] \arrow[u,xshift=-5pt,"Z_{\leq 2}"] & \phantom{\A}\cdots \arrow[l] 
\end{tikzcd}\]

By the universal property of inverse limits, there is a map $\Psi : \A \to \hatwf$, compatible with the maps $\psi_m$, which implies that $\Psi$ is an inverse to $\widehat{Z}$.
\end{proof}

\section{Topological characterisation of the groups $\kv$ and $\krv$}\label{sec: KRV}
Alekseev and Torossian describe solutions to the generalised Kashiwara--Vergne (KV) equations in Section 5.3 of \cite{AT12}. A solution of the KV equations is a tangential automorphism  of the degree completed free Lie algebra on two generators, $F\in\TAut_2$, such that $$F(e^xe^y)=e^{x+y},$$ which also satisfies an additional divergence condition.\footnote{We note that what we denote by $F\in\TAut_2$ would be $F^{-1}$ in \cite{AT12} and \cite{BND:WKO2}. The reason for this is that we match our notation with the group presentations of $\kv$ and $\krv$ in \cite{ATE10}.} Recall that the logarithm of $e^xe^y$ is an infinite Lie series \[\mathfrak{bch}(x,y)= \log(e^x e^y)= x+y+\frac{1}{2}[x,y]+ \frac{1}{12}[x,[x,y]]+ \ldots \] called the Baker--Campbell--Hausdorff (BCH) series in $x$ and $y$. In this sense, the KV problem is a refinement of the BCH Theorem.

\begin{definition}\label{def:solkv}The set of solutions to the KV equations is given by 
\begin{multline*}
\text{SolKV}=\text{SolKV}(\mathbb{Q}):= \Big\{(F, r)\in \TAut_2(\mathbb{Q})\times u^2\mathbb{Q}[[u]] \Big| \\ 
F(e^xe^y)=e^{x+y} \ \text{and} \  J(F)=\trace\Big(r(x+y)-r(x)-r(y)\Big)\Big\}.
\end{multline*} 
\end{definition}

In fact, in any pair $(F,r)\in \text{SolKV}$, the element $F$ uniquely determines the power series $r$ (\cite[before Proposition 6]{ATE10}). This assignment $F \mapsto r$ is called the \emph{Duflo map}, and $r$ is called the \emph{Duflo function} of the automorphism $F$. Accordingly, the automorphism $F\in \TAut_2$ is often called a KV solution and, following the literature, we will write $F\in \text{SolKV}$, as long as there exists $r\in u^2\mathbb Q[[u]]$ with $(F,r)\in \text{SolKV}$. 

The symmetry groups $\kv$ and $\krv$ act freely and transitively on the set of KV solutions. Both $\kv$ and $\krv$ are subgroups of tangential automorphisms which satisfy a divergence condition. While elements of $\kv$ and $\krv$ are defined below as pairs, in both cases the first component of the pair determines the second \emph{uniquely}, and the group structures are given by composition in $\TAut_2$ (\cite[Section 6.1]{ATE10}). 
\begin{definition}\label{def: kv}\cite[after Proposition 7]{ATE10}
The Kashiwara--Vergne group $\kv$ is
\begin{multline*}
\kv=\kv(\mathbb{Q}):= \Big\{(a, \sigma)\in \TAut_2(\mathbb{Q})\times u^2\mathbb{Q}[[u]] \Big| \\ 
a(e^{x}e^{y})= e^{x}e^{y} \ \text{and} \  J(a)=\trace\Big(\sigma(\mathfrak{bch}(x,y))-\sigma(x)-\sigma(y)\Big)\Big\}
\end{multline*}
The left action of $\kv$ on SolKV is given by right composition with the inverse in $\TAut_2$. That is, $a\cdot F=F\circ a^{-1}.$
\end{definition}

The graded Kashiwara--Vergne group $\krv$ is the group corresponding to the Lie subalgebra of the tangential derivations of $\lie_2$ which vanish on the sum of the generators, and whose divergence is in the kernel of the 1-cocycle $\delta$ \cite[Remark 50]{ATE10}, \cite[Sec.\ 4.1]{AT12}:
\begin{definition}\label{def: krv}\cite[after Proposition 7]{ATE10}
 The group $\krv$ is defined as 
\begin{multline*}
\krv=\krv(\mathbb{Q}):= \Big\{(\alpha, s)\in \TAut_2(\mathbb{Q})\times u^2\mathbb{Q}[[u]] \Big| \\ 
\alpha(e^{x+y})=e^{x+y} \ \text{and} \  J(\alpha)=\trace\Big(s(x+y)-s(x)-s(y)\Big)\Big\}
\end{multline*}
The right action of $\krv$ on SolKV is given by left composition with the inverse in $\TAut_2$. That is, $F\cdot \alpha=\alpha^{-1}\circ F$.
\end{definition} 

For both $\kv$ and $\krv$ there are Duflo maps given by $a\mapsto \sigma$ and $\alpha \mapsto s$. These Duflo maps are group homomorphisms $\TAut_2 \to u^2\mathbb Q[[u]]$, where the group operation in $\TAut_2$ is composition, and $u^2\mathbb Q[[u]]$ is the additive group \cite[Proposition 34]{ATE10}. For example, if $(\alpha_1, s_1)\in \krv$ and $(\alpha_2,s_2)\in \krv$ then $(\alpha_2\circ \alpha_1, s_1+s_2)\in \krv$.

\medskip 

In this section we identify the two symmetry groups $\kv$ and $\krv$ with the groups of automorphisms of the circuit algebras of $w$-foams and arrow diagrams, respectively. These act freely and transitively on the set of homomorphic expansions -- which are identified with $\text{SolKV}$ in \cite{BND:WKO2} -- by pre- and post-composition. Since working with graded spaces is more tractable, we consider automorphisms of arrow diagrams first. 

 \subsection{Automorphisms of arrow diagrams and the group $\krv$}
A circuit algebra automorphism of $\arrows$ is said to be \emph{skeleton preserving} if it restricts to a linear map $\arrows(s) \to \arrows(s)$ for each skeleton $s \in \calS$.  

\begin{definition}\label{def:aut}
We denote by $\Aut(\arrows)$ the group of skeleton preserving, filtered circuit algebra automorphisms $G:\arrows\to \arrows$ which commute with the auxiliary operations (orientation switches and unzips as in Definition~\ref{def:operations on arrows}), and which satisfy \[G(\rightarrowdiagram)=\rightarrowdiagram, \quad G(\vertex)=N=e^\eta, \quad \text{and} \quad G(\upcap)=\Gamma=e^\gamma,\] for primitive elements $\eta \in \arrows(\uparrow_2)$ and $\gamma \in \arrows(\upcap)$.  
\end{definition}

The following proposition justifies why $\Aut(\arrows)$ is the correct choice of automorphism group to consider for $\A$. 

\begin{prop} \label{prop:AutAndExpansions} Given a circuit algebra automorphism $G:\arrows \to \arrows$,  the following are equivalent:
\begin{enumerate}
\item $G\in \Aut(\arrows)$.
\item For any group-like homomorphic expansion $Z: \wf \to \arrows$, the composition $G\circ Z: \wf \rightarrow \arrows$ is also a group-like homomorphic expansion.
\end{enumerate}
\end{prop}

\begin{proof}
First assume that $G\in \Aut(\arrows)$ and that $Z: \wf \to \arrows$ is a group-like homomorphic expansion. We will show that $G\circ Z$ is also a group-like homomorphic expansion.

Since $G$ and $Z$ are both circuit algebra homomorphisms which commute with the auxiliary operations, it follows that $G\circ Z$ is also a homomorphism which commutes with these operations. We need to verify that $\gr (G\circ Z)= \id_\arrows$. Indeed, \begin{equation*}\gr (G \circ Z)=\gr(G) \circ \gr(Z) = \gr(G) \circ \id_A=\gr(G).\end{equation*} The fact that $\gr(G)= \id_\arrows$ follows from the properties that $G(\rightarrowdiagram)=\rightarrowdiagram$, $G(\vertex)=e^\eta$, $G(\upcap)=e^\gamma$, via the definition of the associated graded map. Hence, $G\circ Z$ is a homomorphic expansion.

The homomorphic expansion $G\circ Z$ is group-like if and only if the $G\circ Z$-values of the generators $\overcrossing$, $\vertex$ and $\upcap$ are group-like (exponentials of primitives). Recall from Proposition~\ref{prop:ExpansionEquations} that 
$$Z(\overcrossing)=R=e^{\rightarrowdiagram}\in\arrows(\uparrow_2), \quad Z(\upcap)=C=e^{c} \in \arrows(\upcap) \quad \text{and}\quad Z(\vertex)=V=e^{v}\in\arrows(\uparrow_2).$$
Therefore, $(G\circ Z )(\overcrossing)= e^{\rightarrowdiagram}$, since $G$ fixes the arrow $\rightarrowdiagram$; this is a group-like element. For the cap, $(G\circ Z)(\upcap)= C\Gamma$ as shown in Figure~\ref{fig:GZ}. This is a group-like element, as products of group-like elements are group-like (product here is understood as stacking in $\arrows(\uparrow)$). Similarly for the vertex, $(G\circ Z)(\vertex)=VN$ is group-like, as it is a product of group-like elements in the Hopf algebra $\arrows(\uparrow_2)$ (using Lemma~\ref{lemma: VI}). This completes the proof of the direction $(1) \Rightarrow (2)$.

For the opposite direction, assume that $G: \arrows \to \arrows$ is a circuit algebra automorphism with the property that, for any group-like homomorphic expansion $Z$, the composite $G\circ Z$ is also a group-like homomorphic expansion.
First note that by the uniqueness of the induced maps $\widehat{Z}$ (Proposition \ref{prop:Zhat}), we have that $$\widehat{G\circ Z}= G\circ \widehat{Z}: \widehat{\wf} \stackrel{\cong}{\longrightarrow} \arrows.$$ 
We need to show that $G$ is filtered, skeleton preserving and commutes with all auxiliary operations. All of these follow from the fact that $\widehat{Z}$ is a skeleton preserving circuit algebra isomorphism which intertwines auxiliary operations. To explicitly demonstrate this for the unzip operation, note that in the diagram:
\[
\begin{tikzcd}
\widehat{\wf}  \arrow[d, "u"]\arrow[r, "\widehat{Z}"]  & \arrows  \arrow[d, "u"]  \arrow[r, "G"] & \arrows  \arrow[d, "u"]
\\
\widehat{\wf} \arrow[r, "\widehat{Z}"] & \arrows \arrow[r, "G"] &\arrows 
\end{tikzcd}
\]
the outer rectangle commutes and that the $\widehat{Z}$ maps are invertible. It follows that the square on the right-hand side commutes. An identical argument works for the other operations. Thus, $G$ is a skeleton preserving circuit algebra automorphism which commutes with unzips and orientation switches. 

The automorphism $G$ is filtered because the composite $G\circ \widehat{Z}$ is filtered and $\widehat{Z}$ is a filtered isomorphism, therefore $G=G\circ \widehat{Z} \circ \widehat{Z}^{-1}$ is a composition of filtered maps.

By Theorem 3.10 of \cite{BND:WKO2}, any homomorphic expansion $Z: \wf \to \arrows$ takes the value $Z(\overcrossing)=e^{\rightarrowdiagram}=R$. Therefore, if $G\circ Z$ and $Z$ are both homomorphic expansions, then $G(\rightarrowdiagram)=\rightarrowdiagram$. Let $G(\vertex)=N$ and $G(\upcap)=\Gamma$. Then $(G\circ Z)(\vertex)=VN$ and $(G\circ Z)(\upcap)=C\Gamma$. By assumption, $V$, $VN$,  $C$ and $C\Gamma$ are group-like.  It follows that $N$ and $\Gamma$ are group-like, and hence can be described as $N=e^\eta$ and $\Gamma=e^\gamma$ as in Definition~\ref{def:aut}. This completes the proof of the $(2) \Rightarrow (1)$ direction.
\end{proof}

\begin{figure}	
	$$G \circ Z(\upcap) = G\left(\raisebox{-0.35\height}{\def\svgscale{0.5}
\begingroup%
  \makeatletter%
  \providecommand\color[2][]{%
    \errmessage{(Inkscape) Color is used for the text in Inkscape, but the package 'color.sty' is not loaded}%
    \renewcommand\color[2][]{}%
  }%
  \providecommand\transparent[1]{%
    \errmessage{(Inkscape) Transparency is used (non-zero) for the text in Inkscape, but the package 'transparent.sty' is not loaded}%
    \renewcommand\transparent[1]{}%
  }%
  \providecommand\rotatebox[2]{#2}%
  \newcommand*\fsize{\dimexpr\f@size pt\relax}%
  \newcommand*\lineheight[1]{\fontsize{\fsize}{#1\fsize}\selectfont}%
  \ifx\svgwidth\undefined%
    \setlength{\unitlength}{38.53507034bp}%
    \ifx\svgscale\undefined%
      \relax%
    \else%
      \setlength{\unitlength}{\unitlength * \real{\svgscale}}%
    \fi%
  \else%
    \setlength{\unitlength}{\svgwidth}%
  \fi%
  \global\let\svgwidth\undefined%
  \global\let\svgscale\undefined%
  \makeatother%
  \begin{picture}(1,1.45970929)%
    \lineheight{1}%
    \setlength\tabcolsep{0pt}%
    \put(0,0){\includegraphics[width=\unitlength,page=1]{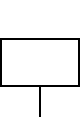}}%
    \put(0.38878495,0.54217816){\color[rgb]{0,0,0}\makebox(0,0)[lt]{\lineheight{1.25}\smash{\begin{tabular}[t]{l}$e^c$\end{tabular}}}}%
    \put(0,0){\includegraphics[width=\unitlength,page=2]{Composition.pdf}}%
  \end{picture}%
\endgroup%
}\right)=\raisebox{-0.4\height}{\def\svgscale{0.5}
\begingroup%
  \makeatletter%
  \providecommand\color[2][]{%
    \errmessage{(Inkscape) Color is used for the text in Inkscape, but the package 'color.sty' is not loaded}%
    \renewcommand\color[2][]{}%
  }%
  \providecommand\transparent[1]{%
    \errmessage{(Inkscape) Transparency is used (non-zero) for the text in Inkscape, but the package 'transparent.sty' is not loaded}%
    \renewcommand\transparent[1]{}%
  }%
  \providecommand\rotatebox[2]{#2}%
  \newcommand*\fsize{\dimexpr\f@size pt\relax}%
  \newcommand*\lineheight[1]{\fontsize{\fsize}{#1\fsize}\selectfont}%
  \ifx\svgwidth\undefined%
    \setlength{\unitlength}{38.76751607bp}%
    \ifx\svgscale\undefined%
      \relax%
    \else%
      \setlength{\unitlength}{\unitlength * \real{\svgscale}}%
    \fi%
  \else%
    \setlength{\unitlength}{\svgwidth}%
  \fi%
  \global\let\svgwidth\undefined%
  \global\let\svgscale\undefined%
  \makeatother%
  \begin{picture}(1,2.22777126)%
    \lineheight{1}%
    \setlength\tabcolsep{0pt}%
    \put(0,0){\includegraphics[width=\unitlength,page=1]{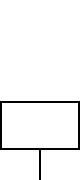}}%
    \put(0.38645387,0.53892759){\color[rgb]{0,0,0}\makebox(0,0)[lt]{\lineheight{1.25}\smash{\begin{tabular}[t]{l}$e^c$\end{tabular}}}}%
    \put(0,0){\includegraphics[width=\unitlength,page=2]{Composition2.pdf}}%
    \put(0.29460928,1.30788528){\color[rgb]{0,0,0}\makebox(0,0)[lt]{\lineheight{1.25}\smash{\begin{tabular}[t]{l}$\Gamma$\end{tabular}}}}%
  \end{picture}%
\endgroup%
}$$
	\caption{The $G\circ Z$ value of the cap.}\label{fig:GZ}
\end{figure}

Any skeleton preserving automorphism $G \in \Aut(\arrows)$ is determined by its values on the generators of $\arrows$: the arrow $\rightarrowdiagram$, the cap $\upcap$, and the vertex $\vertex$.  Proposition~\ref{prop:AutAndExpansions} implies that any such $G$ fixes the generator $\rightarrowdiagram$, which means that $G$ is determined by
$$G(\upcap):=\Gamma=e^\gamma \quad \text{and} \quad G(\vertex):=N=e^\eta.$$

These values cannot be arbitrarily chosen, but must satisfy the equations arising from the relations of $\arrows$.   To determine these equations, recall from Proposition \ref{prop:ExpansionEquations}, that any group-like homomorphic expansion is determined by its values on the generators of $\wf$:
$$Z(\overcrossing)=R=e^{\rightarrowdiagram}\in\arrows(\uparrow_2), \quad Z(\upcap)=C=e^{c} \in \arrows(\upcap) \quad \text{and}\quad Z(\vertex)=V=e^{v}\in\arrows(\uparrow_2),$$ 
subject to the equations \eqref{R4}, \eqref{U} and \eqref{C}. Combining this with Proposition~\ref{prop:AutAndExpansions} we arrive at the following classification of automorphisms $G\in\Aut(\A)$.  

\begin{prop}\label{prop:SimplifiedEqns} The values $G(\upcap)=\Gamma=e^\gamma $ and $G(\vertex)=N=e^\eta$  generate an automorphism in $\Aut(\arrows)$ if and only if they satisfy the following equations with $R=e^{\rightarrowdiagram}\in\arrows(\uparrow_2)$:
\begin{gather} 
N^{12}R^{(12)3}(N^{12})^{-1}=R^{(12)3} \quad \text{ in } \quad \arrows(\uparrow_3) \tag{R4'}  \label{R4'}\\
N\cdot A_1A_2 (N)=1 \quad \text{ in } \quad \arrows(\uparrow_2)  \tag{U'} \label{U'}\\
N^{12}\Gamma^{12} =\Gamma^1\Gamma^2  \quad \text{ in } \quad \arrows(\upcap_2) \tag{C'} \label{C'}.
\end{gather}

\end{prop}

\begin{proof}
Assume that $G\in \Aut(\arrows)$. As in the previous proof, we use that $(G\circ Z)(\upcap)= C\Gamma$ and $(G\circ Z)(\vertex)= VN$. 
Because $G\circ Z$ is a homomorphic expansion, the values $(G\circ Z)(\upcap)=C\Gamma$ and $(G\circ Z)(\vertex)=VN$ satisfy the equations \eqref{R4}, \eqref{U} and \eqref{C} of Proposition~\ref{prop:ExpansionEquations}. 

The equation \eqref{R4} applied to the composite gives the equation:
\begin{equation}\label{R4 combined} V^{12}N^{12}R^{(12)3}=R^{23}R^{13}V^{12}N^{12}.\end{equation}
Applying the \eqref{R4} equation on the right of \eqref{R4 combined} results in 
\begin{equation*}V^{12}N^{12}R^{(12)3}=V^{12}R^{(12)3}N^{12}.\end{equation*} After some rearranging, this is equivalent to the simplified equation \eqref{R4'}, as stated.

Similarly, if we apply \eqref{U} to the composite $G\circ Z$ we get the equation:
\begin{equation}\label{U composite}VN\cdot A_1A_2(VN)=1.\end{equation} 
The adjoint operation reverses the multiplication order (as it is an orientation switch). Multiplying by $V^{-1}$ on the left of \eqref{U composite} and $A_1A_2(V)^{-1}$ on the right of \eqref{U composite}, we obtain
\begin{equation*} N\cdot A_1A_2(N)=V^{-1}\cdot A_1A_2(V)^{-1}=1,\end{equation*} where the second equality holds because of the \eqref{U} equation for $V$. 

Finally, the Cap equation \eqref{C} for the composite $G\circ Z$ results in the equation:
\begin{equation}\label{C combined} V^{12}N^{12}C^{12}\Gamma^{12}=C^1 \Gamma^1 C^2 \Gamma^2.\end{equation}
First, observe that $N^{12}$ commutes with $C^{12}$: this is by the $VI$ relation, as shown in Figure~\ref{fig: V and C}. Furthermore, all instances of $\Gamma$ and $C$ commute with each other by the $TC$ relation. This leads to:
$$V^{12}C^{12}N^{12}\Gamma^{12}=C^1C^2\Gamma^1\Gamma^2.$$
Now note that the original Cap equation \eqref{C} says that on two capped strands $V^{12}C^{12}=C^1C^2$. Applying $G$ on both sides then gives $V^{12}C^{12}\Gamma^1\Gamma^2=C^1C^2\Gamma^1\Gamma^2$. Therefore, we can replace the right hand side above with $V^{12}C^{12}\Gamma^1\Gamma^2$.
$$V^{12}C^{12}N^{12}\Gamma^{12}=V^{12}C^{12}\Gamma^1\Gamma^2 .$$ 
Using that $\arrows(\upcap_2)$ has a left $\arrows(\uparrow_2)$-module by stacking -- see Example~\ref{ex: action on cyc 2} -- we can then left multiply by $(V^{12}C^{12})^{-1}$ to obtain the (\ref{C'}) equation.

Conversely, the equations \eqref{R4'}, \eqref{U'} and \eqref{C'} for $G\circ Z$ imply the equations \eqref{R4}, \eqref{U} and \eqref{C} for $Z$, and therefore by Proposition~\ref{prop:AutAndExpansions},  $G\in \Aut{\arrows}$.
\end{proof}

\begin{figure}
\begin{tikzpicture}[x=0.75pt,y=0.75pt,yscale=-1,xscale=1]

\draw    (131.94,98.39) -- (86.52,60.51) ;
\draw    (38,98.39) -- (86.52,60.51) ;
\draw [color={rgb, 255:red, 38; green, 10; blue, 248 }  ,draw opacity=1 ][line width=2.25]    (86.18,13) -- (86.52,60.51) ;
\draw  [color={rgb, 255:red, 208; green, 2; blue, 27 }  ,draw opacity=1 ][fill={rgb, 255:red, 255; green, 255; blue, 255 }  ,fill opacity=1 ][line width=1.5]  (69.31,34.32) -- (105.45,34.32) -- (105.45,55.69) -- (69.31,55.69) -- cycle ;
\draw    (465.57,98.39) -- (420.14,60.51) ;
\draw    (371.62,98.39) -- (420.14,60.51) ;
\draw [color={rgb, 255:red, 38; green, 10; blue, 248 }  ,draw opacity=1 ][line width=2.25]    (419.8,13) -- (420.14,60.51) ;
\draw  [color={rgb, 255:red, 208; green, 2; blue, 27 }  ,draw opacity=1 ][fill={rgb, 255:red, 255; green, 255; blue, 255 }  ,fill opacity=1 ][line width=1.5]  (378.85,70.33) -- (458.34,70.33) -- (458.34,92.29) -- (378.85,92.29) -- cycle ;
\draw    (133.15,282.58) -- (87.72,244.7) ;
\draw    (39.2,282.58) -- (87.72,244.7) ;
\draw [color={rgb, 255:red, 38; green, 10; blue, 248 }  ,draw opacity=1 ][line width=2.25]    (87.38,197.19) -- (87.72,244.7) ;
\draw  [color={rgb, 255:red, 208; green, 2; blue, 27 }  ,draw opacity=1 ][fill={rgb, 255:red, 255; green, 255; blue, 255 }  ,fill opacity=1 ][line width=1.5]  (45.23,256.97) -- (127.13,256.97) -- (127.13,275.26) -- (45.23,275.26) -- cycle ;
\draw    (301.77,283.8) -- (256.34,245.92) ;
\draw    (207.82,283.8) -- (256.34,245.92) ;
\draw [color={rgb, 255:red, 38; green, 10; blue, 248 }  ,draw opacity=1 ][line width=2.25]    (256,198.41) -- (256.34,245.92) ;
\draw  [color={rgb, 255:red, 208; green, 2; blue, 27 }  ,draw opacity=1 ][fill={rgb, 255:red, 255; green, 255; blue, 255 }  ,fill opacity=1 ][line width=1.5]  (225.53,251) -- (287,251) -- (287,272.6) -- (225.53,272.6) -- cycle ;
\draw    (160.85,48.38) -- (347.54,49.59) ;
\draw  [color={rgb, 255:red, 208; green, 2; blue, 27 }  ,draw opacity=1 ][fill={rgb, 255:red, 255; green, 255; blue, 255 }  ,fill opacity=1 ][line width=1.5]  (207.03,276.48) -- (303.38,276.48) -- (303.38,296) -- (207.03,296) -- cycle ;
\draw    (470.39,280.14) -- (424.96,242.26) ;
\draw    (376.44,280.14) -- (424.96,242.26) ;
\draw [color={rgb, 255:red, 38; green, 10; blue, 248 }  ,draw opacity=1 ][line width=2.25]    (424.62,194.75) -- (424.96,242.26) ;
\draw  [color={rgb, 255:red, 208; green, 2; blue, 27 }  ,draw opacity=1 ][fill={rgb, 255:red, 255; green, 255; blue, 255 }  ,fill opacity=1 ][line width=1.5]  (393.15,249.65) -- (450.32,249.65) -- (450.32,267.94) -- (393.15,267.94) -- cycle ;
\draw  [color={rgb, 255:red, 208; green, 2; blue, 27 }  ,draw opacity=1 ][fill={rgb, 255:red, 255; green, 255; blue, 255 }  ,fill opacity=1 ][line width=1.5]  (377.65,272.82) -- (474,272.82) -- (474,292.34) -- (377.65,292.34) -- cycle ;
\draw [->]   (87,109) -- (87,172) ;
\draw [->]   (423,109) -- (423,172) ;
\draw  [color={rgb, 255:red, 208; green, 2; blue, 27 }  ,draw opacity=1 ][fill={rgb, 255:red, 255; green, 255; blue, 255 }  ,fill opacity=1 ][line width=1.5]  (70.52,238.67) -- (105.45,238.67) -- (105.45,220.37) -- (70.52,220.37) -- cycle ;
\draw [->]   (139.17,242) -- (198.19,242) ;
\draw[->]    (305.38,242) -- (364.4,242) ;

\draw (248.82,24.92) node [anchor=north west][inner sep=0.75pt]    {$=$};
\draw (249.61,61.51) node [anchor=north west][inner sep=0.75pt]    {$VI$};
\draw (80,37.47) node [anchor=north west][inner sep=0.75pt]  [font=\normalsize]  {$C$};
\draw (412,73.82) node [anchor=north west][inner sep=0.75pt]  [font=\normalsize]  {$C^{^{12}}$};
\draw (80,223) node [anchor=north west][inner sep=0.75pt]  [font=\normalsize]  {$C$};
\draw (80,259.01) node [anchor=north west][inner sep=0.75pt]  [font=\small]  {$N^{12}$};
\draw (243,253) node [anchor=north west][inner sep=0.75pt]  [font=\normalsize]  {$C^{^{12}}$};
\draw (243,279.4) node [anchor=north west][inner sep=0.75pt]  [font=\small]  {$N^{12}$};
\draw (412,253) node [anchor=north west][inner sep=0.75pt]  [font=\small]  {$N^{12}$};
\draw (412,274) node [anchor=north west][inner sep=0.75pt]  [font=\normalsize]  {$C^{^{12}}$};
\draw (156.82,222.92) node [anchor=north west][inner sep=0.75pt]    {$=$};
\draw (325.82,226.92) node [anchor=north west][inner sep=0.75pt]    {$=$};
\draw (67,132) node [anchor=north west][inner sep=0.75pt]    {$G$};
\draw (430,132) node [anchor=north west][inner sep=0.75pt]    {$G$};

\end{tikzpicture}

\caption{For any $G\in \Aut(\arrows)$ and $C\in \arrows(\uparrow)$, $N$ commutes with $C^{12}$.
}\label{fig: V and C}
\end{figure}

\begin{remark}
It is also possible, and requires approximately the same effort, to deduce the equations of Proposition~\ref{prop:SimplifiedEqns} directly from the relations of $\arrows$. However, it would require significantly more work to show that this is a complete set of equations without relying on the \cite{BND:WKO2} characterisation of expansions and Proposition~\ref{prop:AutAndExpansions}.
\end{remark}

Recall from Lemma~\ref{lemma: VI} (based on \cite[Lemma 4.6]{BND:WKO2}) that there is an isomorphism of Hopf algebras $$ \arrows(\uparrow_2)\stackrel{\Upsilon}\cong \widehat{U}(\cyc_2\rtimes(\tder_2\oplus \mathfrak{a}_2)),$$ where $\mathfrak{a}_2$ is central, and $\cyc_2$ is commutative. 

\begin{notation}\label{def:UpsilonN}
For $N=G(\vertex)$, we write $\Upsilon(N)=\Upsilon (e^\eta)=:e^we^{n+a}$ for the image of $N$ in $\widehat{U}(\cyc_2\rtimes(\tder_2\oplus \mathfrak{a}_2)),$ where $w\in \cyc_2$, $n\in \tder_2$, and $a\in \mathfrak{a}_2$. 
\end{notation}

In Section~\ref{sec: expansions} we saw that the set of solutions to the Kashiwara--Vergne conjecture are in one-to-one correspondence with the $v$-small homomorphic expansions. To establish the $v$-small condition in the $\krv$ context, we make the following definition. 

\begin{definition}\label{def:vsmallG}
An automorphism $G\in \Aut(\arrows)$ is {\em v-small} if the projection of $\log \Upsilon(N)$ onto the central subalgebra $\mathfrak{a}_2$ is zero. In other words, for a v-small $N=G(\vertex)$, we have $\Upsilon(N)=e^we^n $, with $w\in \cyc_2$ and $n\in \tder_2$. Let $\Aut_v(\arrows)$ denote the subgroup of v-small elements of $\Aut(\arrows).$
\end{definition}

The modified equation \eqref{U'} is identical to the equation \eqref{U}, but for $N$ in place of $V$. 
Therefore, the following result from \cite[Proof of Theorem 4.9 in Section 4.4]{BND:WKO2} applies directly: 

\begin{prop}\label{prop:U'}
The equation \eqref{U'} is satisfied if and only if $\mathcal{J}(e^n)=e^{-2w}$.
\end{prop}

The following is a straightforward implication of Proposition~\ref{prop:AutAndExpansions} and~\ref{prop:SimplifiedEqns}:

\begin{prop} \label{prop:AutvAndExpansions}For a circuit algebra automorphism $G: \arrows \to \arrows$, the following are equivalent:
\begin{enumerate}
\item $G\in \Aut_v(\arrows)$.
\item For any v-small group-like homomorphic expansion $Z: \wf \to \arrows$, the composition $G\circ Z: \wf \rightarrow \arrows$ is also a v-small group-like homomorphic expansion.
\item Denoting $G(\vertex)=N$ and $G(\upcap)=\Gamma$, the projection of $\log \Upsilon(N)$ to $\mathfrak{a}_2$ is zero, and $N$ and $\Gamma$ satisfy the equations \eqref{R4'}, \eqref{U'} and \eqref{C'} with $R=e^{\rightarrowdiagram}$.
\end{enumerate} \qed
\end{prop}

We're now ready to state and prove the main theorem of this section:

\begin{theorem}\label{thm: aut(A)=KRV}
There is an isomorphism of groups $\Aut_v(\arrows)\cong\krv$. 
\end{theorem}

\begin{proof}
Let $G$ be an element in $\Aut_v(\arrows)$ with \[G(\vertex)=N=e^\eta=\Upsilon^{-1}(e^we^n) \quad \text{and} \quad G(\upcap)=\Gamma=e^\mu=e^{\kappa^{-1}(\gamma)},\]  where we recall from Lemma~\ref{hopf_vertex} that  $\gamma=\kappa(\mu) \in x^2\mathbb Q[[x]]$. Similarly, by Lemma~\ref{lemma: VI}, $e^\eta=\Upsilon^{-1}(e^we^n)$, with $n\in\tder_2$ and $w\in \cyc_2$. We define a map 
\begin{equation}\label{eq:antiiso}\tag{$\overline{\Theta}$}\begin{tikzcd} \overline{\Theta}: \Aut_v(\arrows) \arrow[r] & \krv \end{tikzcd}\end{equation} by setting $\overline{\Theta}(G)=(\alpha,s)\in \krv$ with  $\alpha=e^{n}$ and $s=2\gamma$. We will show that $\overline{\Theta}$ is an anti-isomorphism, and then compose with inversions to obtain an isomorphism.


The main difficulty is in showing that the map $\overline{\Theta}$ is well-defined, that is, that the pair $(\alpha,s)$ satisfies the defining equations of the group $\krv$. To do this, we use that the isomorphism $\Upsilon$ is compatible with the adjoint action of $\TAut_2$ on $\exp(\lie_2)$ in the following sense. Let $(\Upsilon^{-1})^{12}$ denote the composition \[(\Upsilon^{-1})^{12}:\TAut_2 \xrightarrow{\Upsilon^{-1}} \arrows(\uparrow_2) \hookrightarrow \arrows(\uparrow_3),\] where the inclusion $\arrows(\uparrow_2)\hookrightarrow\arrows(\uparrow_3)$ is the natural inclusion of arrow diagrams by adding an empty third strand.  
As we saw in Remark~\ref{ex: action of arrow diag on lie}, there is an inclusion $\log \iota: \lie_2 \hookrightarrow \arrows(\uparrow_3)$ defined on basis elements by $\log\iota(x)=a_{13}$, $\log\iota(y)=a_{23}$, where $a_{ij}$ denotes the horizontal arrow from strand $i$ to strand $j$, and the bracket in the Lie algebra of primitive elements of $\arrows(\uparrow_3)$ is the algebra commutator. The map $\iota: \exp(\lie_2) \to \arrows(\uparrow_3)$ is defined to be $e^{\log \iota}$. 
With this notation, we recall from Remark~\ref{ex: action of arrow diag on lie} that the following diagram commutes:

\begin{equation}\label{eq:TAutAction}
\begin{tikzcd}
\TAut_2 \times \exp(\lie_2)  \arrow[d, hookrightarrow, swap, "(\Upsilon^{-1})^{12}\times \iota"]\arrow[rr, "\ad"]  && \exp(\lie_2) \arrow[d, hookrightarrow, "\iota"]
\\
\arrows(\uparrow_3)\times \arrows(\uparrow_3) \arrow[rr, "\text{conj}"] && \arrows(\uparrow_3).
\end{tikzcd}
\end{equation} 

Now, applying $\Upsilon$ to the \eqref{R4'} equation, we obtain:
\begin{equation}\label{Upsilon R4'}e^we^ne^{x+y}e^{-n}e^{-w}=e^{x+y}.\end{equation}
Multiplying \eqref{Upsilon R4'} by $e^{-w}$ on the left and by $e^{w}$ on the right, and using that $[w,x]=[w,y]=0$, as $\cyc_2$ acts trivially on $\lie_2$, this simplifies to 
\begin{equation*} e^ne^{x+y}e^{-n}=e^{x+y}. \end{equation*}
Therefore, $e^{x+y}=e^{-n}e^{x+y}e^n=\alpha(e^{x+y})$ by the definition of the action of $\alpha$.
Thus, we have shown $\alpha (e^{x+y})=e^{x+y}$: the first defining equation of the group $\krv$.

To verify that the pair $(\alpha, s)$ satisfies the second equation of the group $\krv$, recall from Lemma~\ref{hopf_vertex} that there are linear isomorphisms $\kappa: \arrows(\upcap)
\xrightarrow{\cong} \cyc_1/\langle\xi\rangle$ and $\kappa: \arrows(\upcap_2)\xrightarrow{\cong}\cyc_2/\langle \xi, \zeta \rangle$, between arrow diagrams on capped strands and degree 
completed cyclic words with no single-letter terms (see also Example 2.2 and 2.3 of \cite{AT12}).  Thus, we will 
write $\kappa (Z(\upcap))=\kappa (C)=e^{\trace(c(\xi))}$, and $\kappa (G(\upcap))=\kappa(\Gamma)=e^{\trace(\gamma(\xi))}$, where $c, \gamma\in \mathbb Q[[\xi]]$ have no linear terms.  It follows that $
\kappa (C^{12})=e^{\trace( c(\xi + \zeta))}$, and $\kappa(\Gamma^{12})=e^{\trace(\gamma(\xi + \zeta))}$, where $\trace:\widehat{\mathfrak{ass}}_2 \to \cyc_2$ is the linear trace operator.

Recall from Remark~\ref{ex: action on cyc 2} that $\arrows(\upcap_2)$ is a left $\arrows(\uparrow_2)$-module and the following diagram commutes:
\begin{equation}\label{eq:TAutOnWheels}
\begin{tikzcd}
\TAut_2 \times \exp(\cyc_2/\langle\xi,\zeta\rangle) \arrow[d, hookrightarrow, swap, "(\Upsilon^{-1})\times \kappa^{-1}"]\arrow[rr, "\cdot"]  && \exp(\cyc_2/\langle\xi,\zeta\rangle) \arrow[d, rightarrow, "\kappa^{-1}"]
\\
\arrows(\uparrow_2)\times \arrows(\upcap_2) \arrow[rr, "\text{stack}"] && \arrows(\upcap_2)
\end{tikzcd}
\end{equation} where the top map is the basis conjugating action of $\TAut_2$ on $\exp(\cyc_2)$ as in Example \ref{ex:tdercyc} and Remark~\ref{ex: action on cyc 2}.

Putting this all together, we see that applying $\kappa$ to the equation \eqref{C'} yields the following equation in $\cyc_2/\langle\xi,\zeta \rangle$:
\begin{equation}\label{eq:algebraC'}e^w e^n\cdot (e^{\trace(\gamma(\xi+\zeta))})=e^{\trace(\gamma(\xi))}e^{\trace(\gamma(\zeta))}.\end{equation}

Now, since the action of $\TAut_2$ on $\cyc_2$ comes from the action of $\TAut_2$ on $\exp(\lie_2)$ via taking traces, we have that $e^n\cdot e^{\xi+\zeta}=e^{\xi+\zeta}$ by the \eqref{R4'} equation (and noticing that if $e^{-n}$ fixes $(\xi +\zeta)$ then so does $e^n$). Therefore, $$e^n\cdot (e^{\trace(\gamma(\xi+\zeta))})=e^{\trace(\gamma(\xi+\zeta))},$$ 
and equation \eqref{eq:algebraC'} simplifies to
$$e^w =e^{\trace(\gamma(\xi))}e^{\trace(\gamma(\zeta))}e^{-\trace(\gamma(\xi+\zeta))}=e^{\trace(\gamma(\xi)+\gamma(\zeta)-\gamma(\xi+\zeta))}.$$

Proposition~\ref{prop:U'} says that since $G$ respects the \eqref{U'} equation, we know that $\mathcal{J}(e^n)=e^{-2w}$ and thus $e^w=(\mathcal J(e^n))^{-1/2}=\mathcal J(\alpha)^{-1/2}$. 
Thus,
\begin{equation*}\log \mathcal J(\alpha)=J(\alpha)=\trace(2\gamma(\xi+\zeta)-2\gamma(\xi)-2\gamma(\zeta)),\end{equation*}
and therefore $s=2\gamma$ satisfies the second defining equation of $\krv$ with $\alpha=e^n$. This completes the proof that $\overline{\Theta}$ is well-defined.

It remains to show that the map $\overline{\Theta}$ is anti-multiplicative: given $G_1, G_2 \in \Aut_v(\arrows)$ with $G_1(\vertex)=N_1$ and $G_2(\vertex)=N_2$, then $(G_2\circ G_1)(\vertex)=N_1N_2$, where the multiplication is performed in $\arrows(\uparrow_2)$ (using Lemma~\ref{lemma: VI}). Similarly, if $G_1(\upcap)=e^{\kappa^{-1}(\gamma_1(x))}$ and $G_2(\upcap)=e^{\kappa^{-1}(\gamma_2(x))}$ then $(G_2\circ G_1)(\upcap)= e^{\kappa^{-1}((\gamma_1+\gamma_2)(x))}$ (We have omitted the trace notation in order to reduce clutter.). It follows that $\overline{ \Theta}: \Aut_v(\arrows) \to \krv $ is a group anti-homomorphism.  

\medskip 

To show that $\overline{\Theta}$ is bijective, we will construct its inverse $\overline{\Theta}^{-1}: \krv \to \Aut_v(\arrows)$ as follows. Given a pair $(\alpha,s)\in \krv$, define an automorphism $G=\overline{\Theta}^{-1}(\alpha,s)\in \Aut_v(\arrows)$ by setting:
\[N=G(\vertex)=\Upsilon^{-1}\left((\mathcal J(\alpha))^{-1/2}\alpha\right) \quad \text{and}  \quad \Gamma=G(\upcap)=e^{\kappa^{-1}(s/2)}.\]
Recall that in $N$, since $\alpha\in \TAut_2$,  therefore $\Upsilon^{-1}(\alpha)$ is an exponential of tree diagrams, and since $\mathcal J(\alpha)\in \cyc_2$, $\Upsilon^{-1}((\mathcal J(\alpha))^{-1/2})$ is an exponential of wheels.  

The map $\overline{\Theta}^{-1}$ is clearly inverse to $\overline{\Theta}$, and so it remains to show that it is well-defined.  By Proposition~\ref{prop:AutvAndExpansions} this amounts to verifying that $N$ and $\Gamma$ satisfy equations \eqref{R4'}, \eqref{U'} and \eqref{C'} and therefore define an automorphism of $\Aut_v(\arrows)$. Of these, \eqref{U'} is automatic by construction and by Proposition~\ref{prop:U'}. Equation \eqref{R4'} is satisfied by the commutativity of diagram \eqref{eq:TAutAction} and the first defining equation of $\krv$. Equation \eqref{C'} reduces to the second defining equation of $\krv$, using the commutativity of the diagram \eqref{eq:TAutOnWheels}.

To complete the proof and show we have a group isomorphism, we compose $\overline{\Theta}$ with the group inversion in $\krv$:
\begin{equation}\label{iso}\tag{$\Theta$}\begin{tikzcd} {\Theta}: \Aut_v(\arrows) \arrow[r, "\overline{\Theta}"] & \krv \arrow[r, "(-)^{-1}"] &\krv \end{tikzcd}\end{equation}
The map $\Theta$ is well-defined since $\krv$ is a group, and it is an isomorphism because it is a composite of two anti-isomorphisms. Explicitly, we have $\Theta(G)=(\alpha:=e^{-n},s:=-2\gamma)$, where $G$ is as in the beginning of the proof.
\end{proof}

\medskip 

We show that the isomorphism from Theorem~\ref{thm: aut(A)=KRV} is compatible with the action of $\krv$ on $\text{SolKV}$. In order to do this we briefly summarise \cite[Theorem 4.9]{BND:WKO2} which identifies 
the set $\text{SolKV}$ with the set of $v$-small group-like homomorphic expansions of $\wf$ as follows. 

A $v$-small homomorphic expansion $Z$ evaluated at the vertex $\vertex$ determines a value $\Upsilon(Z(\vertex))=e^be^\nu$ with $b\in \cyc_2$ and $\nu\in \tder_2$, so $e^\nu \in \TAut_2$. The corresponding KV solution $F\in\text{SolKV}$ is given by $F_Z:=e^{-\nu}$.\footnote{The paper \cite{BND:WKO2} follows the conventions of \cite{AT12}, where KV solutions are the inverses of the elements defined to be in SolKV in \cite{ATE10}. Hence the negative in the exponent does not appear in \cite{BND:WKO2}.} Conversely, given an $F\in \text{SolKV}$, the corresponding homomorphic expansion $Z_F:\wf\to \arrows$ is defined by $Z_F(\vertex)=\Upsilon^{-1}((\mathcal J(F^{-1}))^{-1/2}F^{-1})$ and $Z_F(\upcap)$ is uniquely determined by $Z_F(\vertex)$. 

Let $\mathcal{HE}$ denote the set of $v$-small, group-like homomorphic expansions $Z:\wf \to \arrows$ and let $\mathcal W$ denote the \cite{BND:WKO2} identification $\mathcal{HE} \longrightarrow \text{SolKV}$ sending $Z\mapsto F_Z:=e^{-\nu}$. 

\begin{prop} The action of $\Aut_v(\arrows)$ on $\mathcal{HE}$ via post-composition with the inverse, $Z\cdot G=G^{-1}\circ Z$, is compatible with the action of $\krv$ on $\textnormal{SolKV}$. That is, the following square commutes: 
\[
\begin{tikzcd}
\mathcal{HE} \times \Aut_v(\arrows)  \arrow[d, "\mathcal{W}\times \Theta", swap]  \arrow[r, "\cdot"] 
& \mathcal{HE} \arrow[d, "\mathcal{W}"]
\\
\textnormal{SolKV} \times \krv  \arrow[r, "\cdot"] & \textnormal{SolKV}
\end{tikzcd}
\]  
\end{prop}

\begin{proof}
Let $(Z,G)\in \mathcal{HE} \times \Aut_v(\arrows)$, with $Z(\vertex)=V$, $G(\vertex)=N$, $\Upsilon(V)=e^be^\nu$, $\Upsilon(N)=e^we^n$. Then $\mathcal{W}(Z)=e^{-\nu}\in \text{SolKV}$, and $\Theta(G)=e^{-n}\in \krv$. Finally, the right action of KRV on SolKV is, as discussed above, $e^{-\nu}\cdot e^{-n}=e^ne^{-\nu}\in \text{SolKV}$.

On the other hand, $Z\cdot G=G^{-1}\circ Z$, and $G^{-1}\circ Z(\vertex)= VN^{-1}$. The latter statement uses the fact that $G(\rightarrowdiagram)=\rightarrowdiagram$, which in turn implies that $G^{-1}(\vertex)=N^{-1}$. Since $\Upsilon(VN^{-1})=e^be^\nu e^{-n}e^{-w}$, we have $\mathcal{W}(G^{-1}\circ Z)=(e^\nu e^{-n})^{-1}=e^n e^{-\nu}$. This agrees with the above, completing the proof.
\end{proof}

\subsection{The group $\kv$}\label{subsec:kv}

In this section, we describe the symmetry group $\kv$ as a group of automorphisms of the completed circuit algebra of welded foams $\hatwf$. The structure of $\hatwf$ is considerably more complicated than its associated graded $\arrows$.  This makes direct computation of the  group of automorphisms of $\hatwf$ difficult. As such, we rely on Theorem~\ref{thm: expansions are isomorphisms}, which states that homomorphic expansions $Z$ of $w$-foams induce circuit algebra isomorphisms $\widehat{Z}: \hatwf \to \arrows$.  

\begin{definition}
Let $\Aut(\hatwf)$ denote the group of circuit algebra automorphisms $g:\hatwf\rightarrow \hatwf$ which preserve the set of group-like homomorphic expansions under pre-composition. Let $\Aut_v(\hatwf)$ denote the group of circuit algebra automorphisms of $\hatwf$ which preserve the set of group-like homomorphic expansions, and the $v$-small property, under pre-composition. 
\end{definition}

\begin{prop}\label{prop:conjbyZ}
Any choice of homomorphic expansion $Z:\wf \to \arrows$ induces an isomorphism $\theta_Z: \Aut(\hatwf) \to \Aut(\arrows)$ given by $\theta_Z(g)=\widehat{Z}g\widehat{Z}^{-1}$, for $g\in \Aut(\hatwf)$. For a homomorphic expansion $Z$ satisfying the $v$-small condition, this restricts to an isomorphism $\theta_Z: \Aut_v(\hatwf)\to \Aut_v(\arrows)$.
\end{prop}

\begin{proof}
This follows from the fact that $\widehat{Z}$ is a circuit algebra isomorphism which also intertwines the auxiliary operations (Proposition~\ref{prop:Zhat}). Note that by \cite{BND:WKO2} and \cite{AT12}, the existence of a KV solution implies that at least one homomorphic expansion exists.
\end{proof}

Recall that the groups $\kv$ and $\krv$ act freely and transitively on $\text{SolKV}$ on the left and right respectively, as follows: for $a \in \kv$, $F \in \text{SolKV}$ and $\alpha \in \krv$, we have $a \cdot F:= F \circ a^{-1}$, and  $F \cdot \alpha:=\alpha^{-1}\circ F$, where composition takes place in $\TAut_2$. In particular, given a choice of $F\in \text{SolKV}$, there is an isomorphism\footnote{It is illuminating to check directly that $FaF^{-1}\in \krv$. The first defining equation is straightforward; the second is an argument similar to the first half of the proof of Proposition 35 in \cite{ATE10}.} $T_F: \kv \to \krv$ given by $T_F(a)=FaF^{-1}$ (\cite[after Proposition 8]{ATE10}). 

In summary, there are group isomorphisms \[\krv \stackrel{\Theta}{\cong} \Aut_v(\arrows) \quad \text{and} \quad  \Aut_v(\arrows)\stackrel{\theta_{Z}}{\cong} \Aut_v(\hatwf).\] and there is also an isomorphism $\kv\stackrel{T_F}{\cong} \krv$ dependent on a choice of an element $F\in \text{SolKV}$. Putting these isomorphisms in the commutative diagram:  
\begin{equation}\label{eq:KVKRV}
\begin{tikzcd}
\Aut_v(\widehat{\wf})  \arrow[d, dashed, "\Psi_Z", swap]  \arrow[r, "\theta_Z"] 
& \Aut_v(\arrows) \arrow[d, "\Theta"]
\\
\kv  \arrow[r, "T_{F_Z}"] & \krv 
\end{tikzcd}
\end{equation} we obtain an isomorphism $\kv \cong \Aut_v(\hatwf)$ for any choice of $Z\in \mathcal{HE}$ and corresponding $F_Z\in \text{SolKV}$. 


\begin{theorem}\label{thm: ZGZ is in KV}
 The isomorphism
$\Psi_Z=T_{F_Z}^{-1}\circ\Theta\circ\theta_Z$ identifies the left action of $\kv$ on \textnormal{SolKV} with the action of $\Aut_v(\hatwf)$ on homomorphic expansions by pre-composition with the inverse. In other words, the following diagram:
\[
\begin{tikzcd}
\Aut_v(\hatwf)\times \mathcal{HE} \arrow[d, "\Psi_Z \times \mathcal{W}", swap]  \arrow[r, "\cdot"] 
& \mathcal{HE} \arrow[d, "\mathcal{W}"]
\\
\kv \times \textnormal{SolKV} \arrow[r, "\cdot"] & \textnormal{SolKV} 
\end{tikzcd}
\]commutes. 
 \end{theorem}

\begin{proof}
Let $g\in \Aut_v(\widehat{\wf})$, then $g=\widehat{Z}^{-1}G\widehat{Z}$ for some $G\in \Aut_v(\arrows)$. Let $Z':\wf \to \arrows$ be an arbitrary homomorphic expansion. Set \[Z(\vertex)=V, \ \Upsilon(V)=e^be^\nu \quad \text{and} \quad G(\vertex)=N, \ \Upsilon(N)=e^we^n \quad \text{and} \quad Z'(\vertex)=V', \ \Upsilon(V')=e^{b'}e^{\nu'}.\]

The automorphism $g$ acts on $Z'$ by pre-composition with the inverse: $$g\cdot Z'= Z' \circ (\widehat{Z}^{-1}G^{-1}\widehat{Z}).$$ In order to find the corresponding element of $\text{SolKV}$, we need to determine the value $Z' \widehat{Z}^{-1}G^{-1}\widehat{Z}(\vertex)$. Since we can always multiply $VN^{-1}$ by $V^{-1}V$ (Figure~\ref{fig:vertexproof}) we have $$Z' \widehat{Z}^{-1}G^{-1}\widehat{Z}(\vertex)=VN^{-1}V^{-1}V'.$$ 

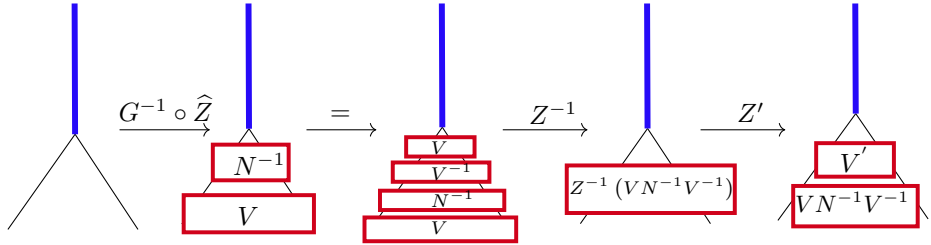
\begin{figure}[h]

\begin{tikzpicture}[x=0.75pt,y=0.75pt,yscale=-.75,xscale=.75]

\draw    (206.63,203.37) -- (164.33,139.3) ;
\draw    (119.16, 203.37) -- (164.33,139.3) ;
\draw [color={rgb, 255:red, 38; green, 10; blue, 248 }  ,draw opacity=1 ][line width=2.25]    (164.01,55) -- (164.33,139.3) ;
\draw  [color={rgb, 255:red, 208; green, 2; blue, 27 }  ,draw opacity=1 ][fill={rgb, 255:red, 255; green, 255; blue, 255 }  ,fill opacity=1 ][line width=1.5]  (140,150) -- (191,150) -- (191,174) -- (140,174) -- cycle ;
\draw  [color={rgb, 255:red, 208; green, 2; blue, 27 }  ,draw opacity=1 ][fill={rgb, 255:red, 255; green, 255; blue, 255 }  ,fill opacity=1 ][line width=1.5]  (120.41,183.99) -- (208,183.99) -- (208,209) -- (120.41,209) -- cycle ;
\draw    (474.64,203.18) -- (432.34,139.11) ;
\draw    (387.16,203.18) -- (432.34,139.11) ;
\draw [color={rgb, 255:red, 38; green, 10; blue, 248 }  ,draw opacity=1 ][line width=2.25]    (432.02,55) -- (432.34,139.11) ;
\draw  [color={rgb, 255:red, 208; green, 2; blue, 27 }  ,draw opacity=1 ][fill={rgb, 255:red, 255; green, 255; blue, 255 }  ,fill opacity=1 ][line width=1.5]  (378,164) -- (491,164) -- (491,196) -- (378,196) -- cycle ;
\draw  [->]  (202.99,140) -- (250,140) ;
\draw    (89.68,207.07) -- (47.38,143) ;
\draw    (2.2,207.07) -- (47.38,143) ;
\draw [color={rgb, 255:red, 38; green, 10; blue, 248 }  ,draw opacity=1 ][line width=2.25]    (47.06,55) -- (47.38,143) ;
\draw  [->]  (77.29,140) -- (138,140) ;
\draw    (621,200) -- (572.34,129.11) ;
\draw    (520,201) -- (572.34,129.11) ;
\draw [color={rgb, 255:red, 38; green, 10; blue, 248 }  ,draw opacity=1 ][line width=2.25]    (572.02,55) -- (572.34,129.11) ;
\draw  [color={rgb, 255:red, 208; green, 2; blue, 27 }  ,draw opacity=1 ][fill={rgb, 255:red, 255; green, 255; blue, 255 }  ,fill opacity=1 ][line width=1.5]  (545.72,149) -- (598,149) -- (598,171.55) -- (545.72,171.55) -- cycle ;
\draw  [color={rgb, 255:red, 208; green, 2; blue, 27 }  ,draw opacity=1 ][fill={rgb, 255:red, 255; green, 255; blue, 255 }  ,fill opacity=1 ][line width=1.5]  (530.28,177.8) -- (614,177.8) -- (614,205) -- (530.28,205) -- cycle ;
\draw  [->]  (334.99,140) -- (392,140) ;
\draw    (336.63,202.37) -- (294.33,138.3) ;
\draw    (249.16,202.37) -- (294.33,138.3) ;
\draw [color={rgb, 255:red, 38; green, 10; blue, 248 }  ,draw opacity=1 ][line width=2.25]    (294.01,55) -- (294.33,138.3) ;
\draw  [color={rgb, 255:red, 208; green, 2; blue, 27 }  ,draw opacity=1 ][fill={rgb, 255:red, 255; green, 255; blue, 255 }  ,fill opacity=1 ][line width=1.5]  (270,145) -- (316,145) -- (316,158) -- (270,158) -- cycle ;
\draw  [color={rgb, 255:red, 208; green, 2; blue, 27 }  ,draw opacity=1 ][fill={rgb, 255:red, 255; green, 255; blue, 255 }  ,fill opacity=1 ][line width=1.5]  (252.41,180.99) -- (336,180.99) -- (336,194) -- (252.41,194) -- cycle ;
\draw  [color={rgb, 255:red, 208; green, 2; blue, 27 }  ,draw opacity=1 ][fill={rgb, 255:red, 255; green, 255; blue, 255 }  ,fill opacity=1 ][line width=1.5]  (261,162) -- (326,162) -- (326,175) -- (261,175) -- cycle ;
\draw  [color={rgb, 255:red, 208; green, 2; blue, 27 }  ,draw opacity=1 ][fill={rgb, 255:red, 255; green, 255; blue, 255 }  ,fill opacity=1 ][line width=1.5]  (242,199) -- (344,199) -- (344,215) -- (242,215) -- cycle ;
\draw  [->]  (467.99,140) -- (525,140) ;

\draw (152,153.95) node [anchor=north west][inner sep=0.75pt]  [font=\normalsize]  {$N^{-1}$};
\draw (155,190) node [anchor=north west][inner sep=0.75pt]  [font=\normalsize]  {$V$};
\draw (378,170) node [anchor=north west][inner sep=0.75pt]  [font=\scriptsize]  {$Z^{-1}\left( VN^{-1}V^{-1}\right)$};
\draw (217.2,125) node [anchor=north west][inner sep=0.75pt]    {$=$};
\draw (560,150) node [anchor=north west][inner sep=0.75pt]  [font=\normalsize]  {$V^{'}$};
\draw (530,180) node [anchor=north west][inner sep=0.75pt]  [font=\small]  {$VN^{-1}V^{-1}$};
\draw (285,200) node [anchor=north west][inner sep=0.75pt]  [font=\scriptsize]  {$V$};
\draw (285,161) node [anchor=north west][inner sep=0.75pt]  [font=\scriptsize]  {$V^{-1}$};
\draw (285,179) node [anchor=north west][inner sep=0.75pt]  [font=\scriptsize]  {$N^{-1}$};
\draw (285,146) node [anchor=north west][inner sep=0.75pt]  [font=\scriptsize]  {$V$};
\draw (75,115) node [anchor=north west][inner sep=0.75pt]   [align=left] {$G^{-1}\circ \widehat{Z}$};
\draw (351,120) node [anchor=north west][inner sep=0.75pt]   [align=left] {$Z^{-1}$};
\draw (491,120) node [anchor=north west][inner sep=0.75pt]   [align=left] {$Z'$};

\end{tikzpicture}
\caption{The value $Z' \widehat{Z}^{-1}G^{-1}\widehat{Z}$ at the vertex.
}\label{fig:vertexproof}
\end{figure} 

The corresponding KV solution is the inverse of the $\TAut_2$-component of $$\Upsilon(VN^{-1}V^{-1}V')=e^be^\nu e^{-n}e^{-w}e^{-\nu}e^{-b}e^{b'}e^{\nu'}.$$ After commuting the $\cyc_2$-components in the semi-direct product to the left, we obtain \begin{equation*}(e^\nu e^{-n}e^{-\nu}e^{\nu'})^{-1}=e^{-\nu'}e^\nu e^{n}e^{-\nu} \in \text{SolKV}. \end{equation*} 

 On the other hand, $\Psi_Z(g)=T^{-1}_{F_Z}(\Theta(G))=e^\nu e^{-n} e^{-\nu}\in \kv$. In turn, $\kv$ acts on $\text{SolKV}$ by composition on the right with the inverse, therefore \begin{equation*}e^\nu e^{-n} e^{-\nu} \cdot e^{-\nu'}= e^{-\nu'}e^{\nu}e^n e^{-\nu},\end{equation*} completing the proof.
\end{proof}

\section{The image of $\mathsf{GRT}_1$ in $\krv$}\label{sec:krvgrt}
The Grothendieck--Teichm\"uller groups $\gt$ and $\grt_1$, defined by Drinfeld in \cite{Drin90}, are known to act freely and transitively on the set of all Drinfeld associators. It was shown by Bar-Natan \cite[Proposition 4.5, Proposition 4.8]{BNGT} that $\grt_1$ is isomorphic to the automorphisms of parenthesised chord diagrams; and later understood that this same result could be concisely presented using operads in \cite{tamarkin1998proof} and \cite[Theorem 10.3.6]{FresseVol1}. In their paper \cite{AT12},  Alekseev and Torossian establish a close relationship between the graded Grothendieck--Teichm\"uller group and the group $\krv$,  conjecturing that $\krv\cong \mathbb{Q}t\times \grt_1$ (\cite[Section 4; Remark 9.14]{AT12}).  In this section we demonstrate how, using the relationship between $\grt_1$ and $\krv$ established by Alekseev and Torossian, we can describe $\grt_1$ as certain automorphisms of arrow diagrams.

\subsection{$\grt_1$ as automorphisms of an operad} 
We briefly recall the description of the Grothendieck-Teichm\"{u}ller group $\grt_1$ as automorphisms of the operad of parenthesised chord diagrams. The interested reader can find a more in-depth description in \cite{BNGT}, \cite{FresseVol1} or \cite{merkGT}.  The Lie algebra of \emph{infinitesimal braids} on $n$ strands, $\mathfrak{t}_n$, $n\geq 2$, is the quotient of the free Lie algebra $\lie_{\frac{n(n-1)}{2}}$ generated by the symbols $\{t_{i,j} = t_{j,i}\}_{1\leq i< j\leq n}$ subject to the relations \[ [t_{i,j} , t_{k,l}] = 0  \quad \text{and} \quad[t_{i,j} , t_{i,k} + t_{k,j}] = 0 \] whenever $i,j,k,l$ are distinct. 

The completed universal enveloping algebra, $\widehat{U}(\mathfrak{t}_n)$, can be viewed as a category of \emph{chord diagrams} on $n$ strands, $\mathsf{CD}(n)$. This category has one object and is enriched in (completed, filtered, coassociative) coalgebras (e.g. \cite[Section 6.4]{merkGT}). Morphisms of $\mathsf{CD}(n)$ are depicted as chord diagrams on $n$ vertical strands, where $t_{i,j}$ is represented by a chord between strands $i$ and $j$. Composition of morphisms is depicted by ``stacking'' the chord diagrams. The collection of categories $\mathsf{CD}:=\{\mathsf{CD}(n)\}$, $n\geq 2$, forms an operad in which operadic composition is defined by a cabling operation (Figure~\ref{operad composition CD}).  Every operation in $\mathsf{CD}$ can be generated by operadic and categorical compositions using a single generator $t_{1,2}:=\chord\in\mathsf{CD}(2)$. 
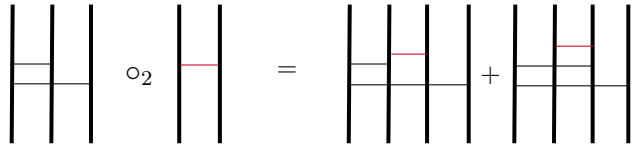
\begin{figure}[h]

\begin{tikzpicture}[x=0.75pt,y=0.75pt,yscale=-.5,xscale=.5]

\draw [line width=1.5]    (31,60) -- (30,200) ;
\draw [line width=1.5]    (71,60) -- (70,200) ;
\draw [line width=1.5]    (110,60) -- (110,200) ;
\draw [line width=1.5]    (199,60) -- (199,200) ;
\draw [line width=1.5]    (240,60) -- (240,200) ;
\draw     (29,120) -- (69,120) ;
\draw    (30,140) -- (109,140) ;
\draw [color={rgb, 255:red, 208; green, 2; blue, 27 }  ,draw opacity=1 ]    (200,121) -- (238.18,121) ;
\draw [line width=1.5]    (371,60) -- (370,200) ;
\draw [line width=1.5]    (411,60) -- (410,200) ;
\draw [line width=1.5]    (449,60) -- (449,200) ;
\draw   (371,120) -- (411,120) ;
\draw [line width=1.5]    (538,60) -- (537,200) ;
\draw [line width=1.5]    (578,60) -- (577,200) ;
\draw [line width=1.5]    (616,60) -- (616,200) ;
\draw    (537,122) -- (617,122) ;
\draw     (538,142) -- (653,142) ;
\draw [line width=1.5]    (491,60) -- (491,200) ;
\draw [line width=1.5]    (652,60) -- (652,200) ;
\draw [color={rgb, 255:red, 208; green, 2; blue, 27 }  ,draw opacity=1 ]    (413,110) -- (449,110) ;
\draw [color={rgb, 255:red, 208; green, 2; blue, 27 }  ,draw opacity=1 ]    (579,102) -- (617,102) ;
\draw (371,141) -- (491,141) ;

\draw (142,120) node [anchor=north west][inner sep=0.75pt]  [font=\Large]  {$\circ _{2}$};
\draw (500,120) node [anchor=north west][inner sep=0.75pt]    {$+$};
\draw (295,120) node [anchor=north west][inner sep=0.75pt]    {$=$};

\end{tikzpicture}
\caption{An example of composition in the operad $\mathsf{CD}$.}\label{operad composition CD}
\end{figure}

\medskip 

The category of \emph{parenthesised chord diagrams}, $\mathsf{PaCD}(n)$, is obtained by formally replacing the object of the category $\mathsf{CD}(n)$ with the set of parenthesised words of length $n$ with distinct letters. Let $M(n)$ denote the subset of parenthesised monomials in $\{1, ..., n\}$ where each element of the set occurs exactly once.\footnote{Equivalently, $M(n)$ may denote the set of binary rooted trees, leaves labelled from the set $\{1, ..., n\}$, where each element of the set occurs exactly once.} For example, $(12)(3(45))$ is an element in $M(5)$. The symmetric group $\calS_n$ acts on the right of $M(n)$ by permuting the elements of the set $\{1, ..., n\}$ and there is an obvious map $u : M(n)\rightarrow \ast$ which collapses every monomial to a point (e.g. \cite[Section 10.3.1]{FresseVol1}).

The category $\mathsf{PaCD}(n)$ is the pullback of the category $\mathsf{CD}(n)$ along the map $u$.  In other words, the set of objects of $\mathsf{PaCD}(n)$ is $M(n)$. Morphisms between any two $p_i,p_j\in M(n)$ are defined as $$\Hom_{\mathsf{PaCD}(n)}(p_i,p_j) := \Hom_{\mathsf{CD}(n)}(\ast,\ast) = \widehat{U}(\mathfrak{t}_n).$$ Categorical composition makes $\PaCD(n)$ into a Hopf groupoid\footnote{Technically, in the \cite{FresseVol1} presentation, the morphisms of $\PaCD(n)$ are only the \emph{group-like} elements of $ \widehat{U}(\mathfrak{t}_n)$.}  \cite[Section 10.3.2]{FresseVol1}.

\medskip 

The collection of all parenthesised chord diagrams, $\mathsf{PaCD}:=\{\mathsf{PaCD}(n)\}_{n\geq 2}$, forms an operad in Hopf groupoids. At the level of objects, operadic composition is given by replacing letters by monomials.\footnote{Alternatively, operadic composition is given by grafting binary trees.} For example, we have the following composition $(1(34))(25)\circ_4 (13)2 = (1(3((46)5)))(27)$. The operadic composition of morphisms is the operadic composition of $\mathsf{CD}$. See \cite{BNGT} or Chapter 10.3.1~\cite{FresseVol1} for full details. 

Every operation in the operad $\mathsf{PaCD}$ can be written as a combination of operadic and categorical compositions of three generators: \[\virtualcrossing\in \Hom_{\mathsf{PaCD}(2)}((12), (21)), \quad \chord\in  \Hom_{\mathsf{PaCD}(2)}((12), (12))\quad \text{and} \quad \Associator\in \Hom_{\mathsf{PaCD}(3)}((12)3,1(23)).\] These generating morphisms are considered up to the \emph{pentagon} and  two \emph{hexagon} relations (notation explained below): 
\begin{equation}\label{PaCD pentagon}  d_1\Associator \cdot d_3 \Associator= d_4 \Associator \cdot d_2 \Associator \cdot d_0 \Associator \end{equation} 
\begin{equation}\label{PaCD hexagon} d_1\virtualcrossing =  \Associator \cdot d_0\virtualcrossing \cdot ((23) (\Associator)^{-1})\cdot (23)d_3\virtualcrossing\cdot (321)\Associator\end{equation} 
\begin{equation}\label{PaCD hexagon2} d_1(\virtualcrossing)^{-1} =  \Associator \cdot d_0(\virtualcrossing)^{-1} \cdot ((23) (\Associator)^{-1})\cdot (23)d_3 (\virtualcrossing)^{-1}\cdot (321)\Associator\end{equation}
Here, the operation $d_0$ adds an "empty" strand to the left of the chord diagram, and $d_{n+1}$ adds an empty strand to the right. For $1\leq i \leq n$, the operation $d_i$ {\em doubles} the $i$th strand of a chord diagram (similar to the {\em unzip} operation of arrow diagrams). This is described on the generators $t_{j,k}$ as follows (see \cite[Definition 2.9]{BNGT}):\[ d_i(t_{j,k})=\begin{cases} t_{j+1,k+1} \ \text{if} \  i < j < k \\ t_{j,k+1} + t_{j+1,k+1} \ \text{if} \ i=j<k \\ t_{j,k+1} \ \text{if} \ j < i < k \\ t_{j,k} + t_{j,k+1} \ \text{if} \ j < i = k \\ t_{j,k} \ \text{if} \ j < k < i .\end{cases}\]  

\begin{definition}\cite[Section 4.2]{AT12}
The graded Grothendieck--Teichm\"uller group $\grt_1$ is the exponentiation of the Lie algebra $\mathfrak{grt}_1$ spanned by derivations $\psi=(0,\psi)\in\tder_2$ which satisfy the following equations: 
\begin{equation}\label{inversion}
\psi(x,y)=-\psi(y,x),
\end{equation} 
\begin{equation}\label{hexagon} 
\psi(x,y)+\psi(y,z) +\psi(z,x) =0 \quad \text{ whenever } x+y+z=0
\end{equation} 
\begin{equation}\label{pentagon}  
\psi(t_{1,2},t_{2,34})+\psi(t_{12,3},t_{3,4}) =\psi(t_{2,3},t_{3,4})+\psi(t_{1,23},t_{23,4})+\psi(t_{1,2},t_{2,3}).
\end{equation} 
 \end{definition} 
In equation \eqref{pentagon}, $t_{i,j}\in\mathfrak{t}_4$ and the notation $t_{ij,k}$ is shorthand for $t_{i,k}+t_{j,k}$. Similarly, $t_{i,jk}=t_{i,j}+t_{i,k}$.  
\medskip 

Let $\Aut(\mathsf{PaCD})$ denote the group of automorphisms of the operad $\mathsf{PaCD}$ which fix objects and the elements $\virtualcrossing\in \Hom_{\mathsf{PaCD}(2)}((12), (21))$ and $\chord\in  \Hom_{\mathsf{PaCD}(2)}((12), (12))$. The following theorem can be found in \cite[Proposition 4.5, Proposition 4.8]{BNGT} or \cite[Theorem 10.3.6]{FresseVol1}. See also \cite{tamarkin1998proof} or \cite[Section 6.7]{merkGT}.

\begin{theorem}\label{thm: GRT is Aut} 
There is an isomorphism of groups $\Aut(\mathsf{PaCD})\cong\grt_1$. 
\end{theorem}

We do not reproduce the proof of Theorem~\ref{thm: GRT is Aut} in this paper, but we will use the following key points.  Since every element of $\PaCD$ can be written via operadic and categorical compositions of $\chord$, $\virtualcrossing$ and $\Associator$, it follows that any automorphism  $G\in\Aut(\PaCD)$ is uniquely determined by its value $G(\Associator)= \Psi(t_{1,2}, t_{2,3})$, where $\Psi=e^\psi$ is an element of $\widehat{U}(\lie_2)$. Since $G$ is an automorphism of $\PaCD$, $G$ respects the pentagon \eqref{PaCD pentagon} and hexagon \eqref{PaCD hexagon}, \eqref{PaCD hexagon2} relations, which means that $\Psi=e^{\psi}$ satisfies equations \eqref{inversion}, \eqref{hexagon} and \eqref{pentagon} and is therefore in $\grt_1$.

\subsection{$\grt_1$ as automorphisms of arrow diagrams} 
 Alekseev and Torossian \cite[Theorem 4.6]{AT12}) provide an injective Lie algebra homomorphism $\varrho: \mathfrak{grt}_1 \rightarrow \mathfrak{krv}_2$ given by
 \[ \psi \mapsto (\psi(-x-y,x), \psi(-x-y,y)). \] Combining this result with those of Theorem~\ref{thm: aut(A)=KRV} and the now classical Theorem~\ref{thm: GRT is Aut}, we have following commutative diagram of groups: \[\begin{tikzcd}
\Aut(\PaCD) \arrow[d, "\cong", swap]\arrow[r, dotted, "\tilde{\varrho}"] & \Aut_v(\arrows) \arrow[d, "\cong"] \\
\grt_1 \arrow[r, "\exp(\varrho)"] & \krv
 \end{tikzcd}\] In order to describe the graded Grothendieck-Teichm\"uller group $\grt_1$ as automorphisms of arrow diagrams, we first describe the image of $\tilde{\varrho}$ in $\Aut_{v}(\arrows).$

\subsubsection{The relationship between $\PaCD$ and $\arrows$}  The algebraic structures of $\PaCD$ and $\arrows$ are different: operads do not embed directly into circuit algebras. Nonetheless, we can describe the image of the underlying collection of Hopf groupoids $\bigcup_{n}\PaCD(n)$ inside the circuit algebra $\arrows$.  

For any $n\geq 2$, there is an inclusion of Lie algebras $\mathfrak{t}_n\hookrightarrow \tder_n$, where $\mathfrak{t}_n$ is isomorphic to a Lie subalgebra of $\tder_n$ spanned by tangential derivations of the form $t^{i,j}=(0,\ldots, x_j,\ldots, x_i,\ldots 0)$,  with non-vanishing components $x_i$ at the $j$th place and $x_j$ at the $i$th place \cite[Proposition 3.11]{AT12}.  The inclusion $\mathfrak{t}_n\hookrightarrow \tder_n$ extends to a Hopf algebra homomorphism \[\begin{tikzcd}\Hom_{\mathsf{CD}(n)}(\ast,\ast):=\widehat{U}(\mathfrak{t}_n)\arrow[r] & \widehat{U}(\tder_n\oplus \mathfrak{a}_n \ltimes \cyc_n)  \arrow[r, "\cong"', "\Upsilon^{-1}"] &  \arrows(\uparrow_n).\end{tikzcd}\]

\begin{lemma}\label{Hopf CD lemma}
For each $n\geq 2$ there exists an inclusion of Hopf algebras \[\begin{tikzcd}\Hom_{\mathsf{CD}(n)}(\ast,\ast)\arrow[r, "\varepsilon"]& \arrows(\uparrow_n).\end{tikzcd}\]
\end{lemma} 

\begin{proof}
We start by defining a linear map $\begin{tikzcd}\Hom_{\mathsf{CD}(n)}(\ast,\ast):=\widehat{U}(\mathfrak{t}_n)\arrow[r, "\varepsilon"]& \arrows(\uparrow_n).\end{tikzcd}$ The generating element $t_{i,j}$ in $\mathfrak{t}_{n}$ is represented as a chord diagram in $\widehat{U}(\mathfrak{t}_n)$ which has a single chord connecting the $i$th vertical strand to the $j$th vertical strand (e.g. $t_{1,2}=\chord$). We set $\varepsilon(t_{i,j})$ to be the sum of two arrow diagrams:  the first a single arrow going from strand $i$ to $j$ and the second a single arrow going from $j$ to $i$. In other words, we send the chord diagram $t_{i,j}$ to the arrow diagram $\Upsilon^{-1}(t^{i,j})$. In particular, $\varepsilon(t_{1,2})=\varepsilon(\chord)=\rarrow + \larrow$. By \cite[Proposition 3.11]{AT12}, this map is well-defined as the $t^{i,j}$s satisfy the relations of $\mathfrak{t}_n$. 

The coproduct on $\mathsf{CD}(n)$ is defined on generators as $\Delta_{\mathsf{CD}}(t_{i,j})=t_{i,j}\otimes 1 + 1\otimes t_{i,j}$. It is now straightforward to check that $\varepsilon(\Delta_{\mathsf{CD}}(t_{i,j}))=\Delta_{\arrows}(\varepsilon(t_{i,j}))$, where the latter is defined in Definition~\ref{def: coprod}. Similarly, the product commutes with $\varepsilon$ since the categorical composition of $\mathsf{CD}(n)$ corresponds to circuit algebra composition given by the wiring diagram $\WD_s$ (Figure~\ref{WD for stacking}) which stacks arrow diagrams on $n$ strands. 
\end{proof}

Next, we extend Lemma~\ref{Hopf CD lemma} to an inclusion of Hopf groupoids $\begin{tikzcd}\PaCD(n)\arrow[r, "\varepsilon"]& \arrows\end{tikzcd}$ by working ``one skeleton at a time''. For any two objects $p_1,p_2$ in $\PaCD(n)$, the empty chord diagram in $\Hom_{\PaCD(n)}(p_1,p_2)$ has an underlying permutation $p_2^{-1}p_1\in\mathcal{S}_n$ (here $p_1$, $p_2 \in M(2)$ are considered as permutations interpreted in the one-line notation). As we saw in Example~\ref{ex: w-tangle skeleton}, there is a corresponding skeleton in $\calS$. We have already implicitly used this idea when we write $\virtualcrossing$ for the empty chord diagram in $\Hom_{\PaCD(2)}((12),(21))$ which has underlying permutation $(21)\in\calS_2$.  The empty chord diagram in $\Hom_{\PaCD(n)}(p_1,p_2)$ carries the additional structure of the parenthesisations $p_1$ and $p_2$ which we can encode by ``closing up'' the skeleton permutation $p_2^{-1}p_1$ with binary trees at the bottom and top\footnote{As a general convention we draw diagrams "bottom to top". That is, $p_1$ is shown at the bottom and morphisms travel up to $p_2$.}
 according to the parenthesisations $p_1$ and $p_2$. We will denote this skeleton element by $c_{p_1p_2}$ in $\arrows$. For example, the closure of $\Associator\in\Hom_{\PaCD(3)}((12)3, 1(23))$ is the skeleton element $\bubble\in \arrows$ depicted in Figure~\ref{def: bubble}. 

\begin{figure}[h]
\begin{tikzpicture}[x=0.75pt,y=0.75pt,yscale=-.5,xscale= .8]

\draw   (458,154.5) .. controls (458,120.53) and (476.13,93) .. (498.5,93) .. controls (520.87,93) and (539,120.53) .. (539,154.5) .. controls (539,188.47) and (520.87,216) .. (498.5,216) .. controls (476.13,216) and (458,188.47) .. (458,154.5) -- cycle ;
\draw    (466,191) -- (530,117) ;
\draw    (498.5,216) -- (499,244) ;
\draw    (498.5,93) -- (498,61) ;
\draw    (99,99.81) -- (98,234.81) ;
\draw    (162,101.81) -- (161,236.81) ;
\draw [->]   (212,170) -- (417,170) ;
\draw    (107,235) .. controls (114,196) and (147,151) .. (150,102) ;

\draw (35,235) node [anchor=north west][inner sep=0.75pt]    {$p_{1} \ =$};
\draw (35,90) node [anchor=north west][inner sep=0.75pt]    {$p_{2} \ =$};
\draw (282,145) node [anchor=north west][inner sep=0.75pt]    {$c_{p_{1}p_{2}}$};
\draw (85,85) node [anchor=north west][inner sep=0.75pt]   [align=left] {(};
\draw (138,85) node [anchor=north west][inner sep=0.75pt]   [align=left] {(};
\draw (80,226) node [anchor=north west][inner sep=0.75pt]   [align=left] {(};
\draw (113,226) node [anchor=north west][inner sep=0.75pt]   [align=left] {)};
\draw (165,226) node [anchor=north west][inner sep=0.75pt]   [align=left] {)};
\draw (165,85) node [anchor=north west][inner sep=0.75pt]   [align=left] {)};
\draw (170,85) node [anchor=north west][inner sep=0.75pt]   [align=left] {)};
\draw (85,226) node [anchor=north west][inner sep=0.75pt]   [align=left] {(};

\end{tikzpicture}
\caption{The closure of $\Associator$ in $\arrows$. }\label{def: bubble} 
\end{figure}

We define the map
\begin{equation} \begin{tikzcd} \Hom_{\PaCD(n)}(p_1,p_2)\arrow[r, "\varepsilon"] &\arrows(c_{p_1p_2})\end{tikzcd}\label{linear map PaCD to A}\tag{$\star$}\end{equation} 
as the composition of the old  $\varepsilon$ (as in Lemma~\ref{Hopf CD lemma}) applied to the underlying morphism of $\mathsf{CD}(n)$, followed by the skeleton closure. Letting $p_i,p_j$ run over all objects in $\PaCD(n)$, we have the following theorem. 

\begin{theorem}\label{inclusion of Hopf groupoids} The categorical composition of $\PaCD(n)$ can be realised using the operations in $\bigcup_{c_{p_ip_j}} \arrows(c_{p_ip_j})$, and this elevates $\varepsilon$ to an inclusion of Hopf groupoids
\[\begin{tikzcd} \PaCD(n)\arrow[r, "\varepsilon"]& \bigcup_{c_{p_ip_j}} \arrows(c_{p_ip_j}).\end{tikzcd}\] 
\end{theorem}

We split the proof of this theorem into three lemmas: first, we show that the categorical composition of $\PaCD(n)$ can be realised in $\bigcup_{c_{p_ip_j}} \arrows(c_{p_ip_j})$ to make $\varepsilon$ multiplicative. Then we show that $\varepsilon$ is well-defined, that is, it respects the pentagon and hexagon relations.

\begin{lemma}\label{lem:emultiplicative}
The categorical composition of $\PaCD(n)$ can be realised using the operations of $\bigcup_{c_{p_ip_j}} \arrows(c_{p_ip_j})$ to make $\varepsilon$ multiplicative.
\end{lemma}

\begin{proof} We need to define a composition operation on $\bigcup_{c_{p_ip_j}} \arrows(c_{p_ip_j})$ in a way that makes the maps \eqref{linear map PaCD to A} compatible with categorical composition in $\PaCD(n)$. This is summarised in the following commutative diagram, for any $p_1,p_2,p_3\in\PaCD(n)$:
\begin{equation}\label{stacking}
\begin{tikzcd} \Hom_{\PaCD(n)}(p_1,p_2)\otimes \Hom_{\PaCD(n)}(p_2,p_3)\arrow[d, "\circ", swap]\arrow[r, "\varepsilon\times \varepsilon"]& \arrows(c_{p_1p_2}) \otimes \arrows (c_{p_2p_3}) \arrow[d, "u^{n-1}\WD_s"] \\
\Hom_{\PaCD(n)}(p_1,p_3) \arrow[r, "\varepsilon"]& \arrows(c_{p_1p_3}) 
\end{tikzcd}
\end{equation}
In the diagram \eqref{stacking}, the left vertical map is categorical composition in $\PaCD$ and the right vertical map is given by first "stacking" (a circuit algebra operation using a wiring diagram as in Figure~\ref{WD for stacking}), and then applying the unzip operation $n-1$ times to eliminate the vertices in the middle, as in Figure~\ref{fig:ClosureComp}.  
\end{proof}

Recall that every element of the operad $\PaCD$ can be obtained by iterated categorical and operadic compositions of the morphisms $\chord$, $\Associator$ and $\virtualcrossing$. Moreover, these generating morphisms satisfy the pentagon and hexagon relations \eqref{PaCD pentagon}, \eqref{PaCD hexagon}, and \eqref{PaCD hexagon2}. The operadic composition of $\PaCD$ is not preserved by $\varepsilon$ (that is, not realised by circuit algebra operations, unzips and orientation switches). However, operadic composition is not required for stating the pentagon and hexagon relations. This, and the multiplicativity of $\varepsilon$ helps show below that these relations are preserved by 

\begin{equation}\label{eq:ue}\tag{$\star\star$}
\begin{tikzcd}
\bigcup\limits_{n\geq 2}\PaCD(n) \arrow[r, "\cup\varepsilon"]& \arrows.
\end{tikzcd}
\end{equation}

\begin{figure}
\[
\begin{tikzpicture}[x=0.75pt,y=0.75pt,yscale=-1,xscale=1]

\draw   (293,90.53) .. controls (293,64.55) and (307.76,43.48) .. (325.97,43.48) .. controls (344.17,43.48) and (358.93,64.55) .. (358.93,90.53) .. controls (358.93,116.51) and (344.17,137.58) .. (325.97,137.58) .. controls (307.76,137.58) and (293,116.51) .. (293,90.53) -- cycle ;
\draw    (299.51,118.45) -- (351.6,61.84) ;
\draw [color={rgb, 255:red, 0; green, 0; blue, 0 }  ,draw opacity=1 ]   (325.97,137.58) -- (326.37,159) ;
\draw    (325.97,43.48) -- (325.56,19) ;
\draw   (294.29,223.49) .. controls (294.29,198.24) and (308.77,177.78) .. (326.64,177.78) .. controls (344.51,177.78) and (359,198.24) .. (359,223.49) .. controls (359,248.73) and (344.51,269.19) .. (326.64,269.19) .. controls (308.77,269.19) and (294.29,248.73) .. (294.29,223.49) -- cycle ;
\draw    (304.08,257.98) -- (355.21,202.99) ;
\draw    (326.64,269.19) -- (327.04,290) ;
\draw [color={rgb, 255:red, 0; green, 0; blue, 0 }  ,draw opacity=1 ]   (326.64,177.78) -- (326.24,154) ;
\draw  [draw opacity=0] (353.25,114.21) .. controls (349.08,111.46) and (346,101.9) .. (346,90.53) .. controls (346,78.68) and (349.34,68.79) .. (353.79,66.53) -- (355.72,90.53) -- cycle ; \draw   (353.25,114.21) .. controls (349.08,111.46) and (346,101.9) .. (346,90.53) .. controls (346,78.68) and (349.34,68.79) .. (353.79,66.53) ;
\draw  [draw opacity=0] (300.81,193.73) .. controls (305.43,197.33) and (308.13,206.24) .. (307.14,216.36) .. controls (306.07,227.24) and (301.1,236.03) .. (295.35,238.08) -- (294.47,215.29) -- cycle ; \draw   (300.81,193.73) .. controls (305.43,197.33) and (308.13,206.24) .. (307.14,216.36) .. controls (306.07,227.24) and (301.1,236.03) .. (295.35,238.08) ;
\draw   (462,154.09) .. controls (462,110.1) and (484.83,74.44) .. (513,74.44) .. controls (541.17,74.44) and (564,110.1) .. (564,154.09) .. controls (564,198.08) and (541.17,233.74) .. (513,233.74) .. controls (484.83,233.74) and (462,198.08) .. (462,154.09) -- cycle ;
\draw    (469,193) -- (541,88) ;
\draw    (513,233.74) -- (513.63,270) ;
\draw    (513,74.44) -- (512.37,33) ;
\draw    (484,220) -- (558,115) ;
\draw [->]   (367,153) -- (436,153) ;
\draw    (62,40) -- (62,140) ;
\draw    (82,40) -- (82,140) ;
\draw    (102,40) -- (102,140) ;
\draw    (122,40) -- (122,140) ;
\draw    (62,180) -- (62,285) ;
\draw    (82,180) -- (82,285) ;
\draw    (102,180) -- (102,285) ;
\draw    (123,180) -- (123,285) ;
\draw [->]   (162,151) -- (252,151) ;
\draw  [color={rgb, 255:red, 208; green, 2; blue, 27 }  ,draw opacity=1 ][fill={rgb, 255:red, 255; green, 255; blue, 255 }  ,fill opacity=1 ] (56,218) -- (126,218) -- (126,237) -- (56,237) -- cycle ;
\draw  [color={rgb, 255:red, 208; green, 2; blue, 27 }  ,draw opacity=1 ][fill={rgb, 255:red, 255; green, 255; blue, 255 }  ,fill opacity=1 ] (56,78) -- (126,78) -- (126,97) -- (56,97) -- cycle ;
\draw  [color={rgb, 255:red, 208; green, 2; blue, 27 }  ,draw opacity=1 ][fill={rgb, 255:red, 255; green, 255; blue, 255 }  ,fill opacity=1 ] (293,80) -- (358,80) -- (358,96) -- (293,96) -- cycle ;
\draw  [color={rgb, 255:red, 208; green, 2; blue, 27 }  ,draw opacity=1 ][fill={rgb, 255:red, 255; green, 255; blue, 255 }  ,fill opacity=1 ] (295,214) -- (359,214) -- (359,230) -- (295,230) -- cycle ;
\draw  [color={rgb, 255:red, 208; green, 2; blue, 27 }  ,draw opacity=1 ][fill={rgb, 255:red, 255; green, 255; blue, 255 }  ,fill opacity=1 ] (463,165) -- (563,165) -- (563,186) -- (463,186) -- cycle ;
\draw  [color={rgb, 255:red, 208; green, 2; blue, 27 }  ,draw opacity=1 ][fill={rgb, 255:red, 255; green, 255; blue, 255 }  ,fill opacity=1 ] (463,125) -- (563,125) -- (563,145) -- (463,145) -- cycle ;

\draw (395,135) node [anchor=north west][inner sep=0.75pt]    {$u^{3}$};
\draw (207,135) node [anchor=north west][inner sep=0.75pt]    {$c$};
\draw (50,175) node [anchor=north west][inner sep=0.75pt]  [font=\footnotesize]  {$( ($};
\draw (94,175) node [anchor=north west][inner sep=0.75pt]  [font=\footnotesize]  {$($};
\draw (94,133) node [anchor=north west][inner sep=0.75pt]  [font=\footnotesize]  {$($};
\draw (50,133) node [anchor=north west][inner sep=0.75pt]  [font=\footnotesize]  {$($};
\draw (45,278) node [anchor=north west][inner sep=0.75pt]  [font=\footnotesize]  {$((($};
\draw (83,175) node [anchor=north west][inner sep=0.75pt]  [font=\footnotesize]  {$)$};
\draw (124,175) node [anchor=north west][inner sep=0.75pt]  [font=\footnotesize]  {$)$};
\draw (124,133) node [anchor=north west][inner sep=0.75pt]  [font=\footnotesize]  {$)$};
\draw (83,133) node [anchor=north west][inner sep=0.75pt]  [font=\footnotesize]  {$)$};
\draw (128,133) node [anchor=north west][inner sep=0.75pt]  [font=\footnotesize]  {$)$};
\draw (53,133) node [anchor=north west][inner sep=0.75pt]  [font=\footnotesize]  {$($};
\draw (128,175) node [anchor=north west][inner sep=0.75pt]  [font=\footnotesize]  {$)$};
\draw (125,278) node [anchor=north west][inner sep=0.75pt]  [font=\footnotesize]  {$)$};
\draw (83,278) node [anchor=north west][inner sep=0.75pt]  [font=\footnotesize]  {$)$};
\draw (103,278) node [anchor=north west][inner sep=0.75pt]  [font=\footnotesize]  {$)$};
\draw (124,34) node [anchor=north west][inner sep=0.75pt]  [font=\footnotesize]  {$)))$};
\draw (92,34) node [anchor=north west][inner sep=0.75pt]  [font=\footnotesize]  {$($};
\draw (73,34) node [anchor=north west][inner sep=0.75pt]  [font=\footnotesize]  {$($};
\draw (53,34) node [anchor=north west][inner sep=0.75pt]  [font=\footnotesize]  {$($};

\end{tikzpicture}
\]
\caption{Closure and composition of chord diagrams.}\label{fig:ClosureComp}
\end{figure}
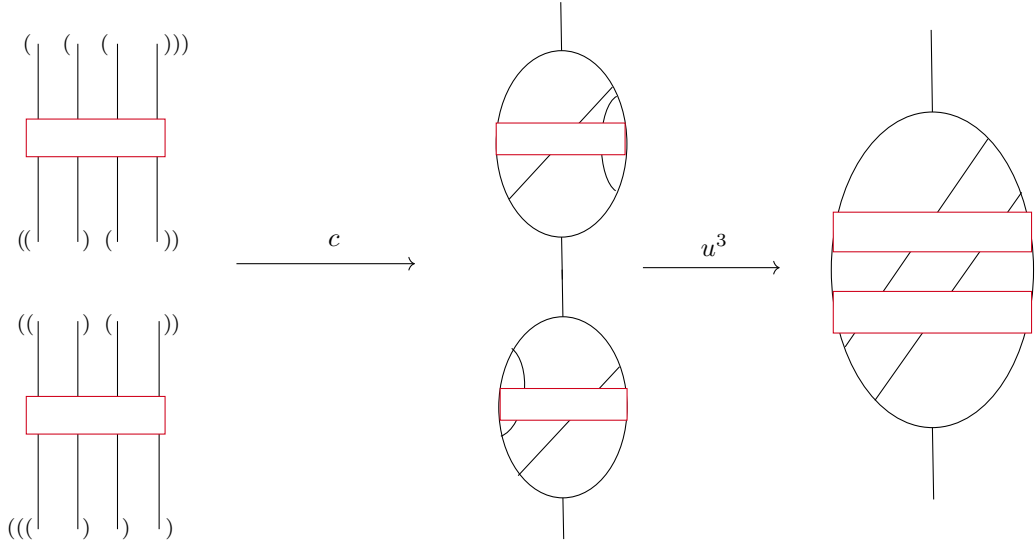

\begin{lemma}\label{lem: pentagon in arrows}
The map  \eqref{eq:ue} respects the pentagon relation. That is, the $\varepsilon$-images of the two sides of the pentagon relation coincide in $\arrows$. 
\end{lemma}

\begin{proof} 
Recall that the pentagon equation \eqref{PaCD pentagon} in $\PaCD$ is 
\begin{equation*} 
d_1\Associator \cdot d_3 \Associator = d_4 \Associator \cdot d_2 \Associator \cdot d_0 \Associator.
\end{equation*}
We apply the closure operation to each of the $d_i\Associator$ to get the corresponding skeleton element $d_i\bubble$, $0\leq i\leq 4$.\footnote{The notation $d_i\bubble$ is only convenient shorthand for the closure of $d_i\Associator$, as the strand doubling operation $d_i$ does not translate directly to a circuit algebra or auxiliary operation, though it can be described in diagrammatic terms as a version of strand doubling.}
We then multiply these elements in $\arrows$ using the operation $u^{n-1}\WD_s$, as in the diagram \eqref{stacking}.
In particular, the image of $d_1\Associator \cdot d_3 \Associator $ in $\arrows$ is the composite of the two ``bubbles'' pictured on the left in Figure~\ref{LHS pentagon}. Completely unzipping the interior of this composite results in the third ``bubble'' on the left in Figure~\ref{LHS pentagon}. Similarly, the element $d_4 \Associator \cdot d_2 \Associator \cdot d_0 \Associator$ corresponds to the stacking of bubbles depicted on the right of Figure~\ref{RHS pentagon}. Completely unzipping the interior results in the third bubble on the right in Figure~\ref{RHS pentagon}. It follows that  \begin{equation}\label{pentagon arrows} d_1\bubble \cdot d_3 \bubble   =d_4 \bubble \cdot d_2 \bubble \cdot d_0 \bubble \end{equation} holds in $\arrows$ as both sides of the equation represent the same skeleton element. 
\end{proof}

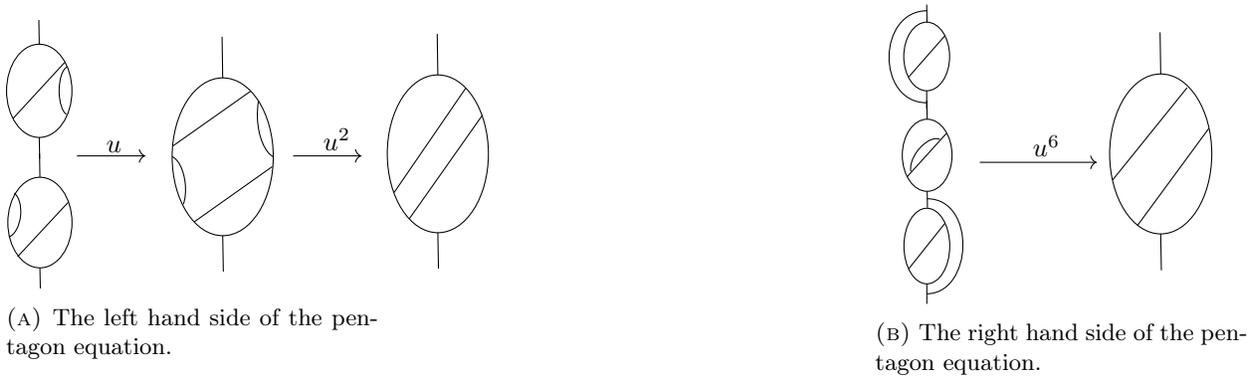
\begin{figure}

\begin{subfigure}{0.3\textwidth}
\begin{tikzpicture}[x=0.75pt,y=0.75pt,yscale=-.5,xscale=.5]

\draw   (58,84.34) .. controls (58,58.36) and (72.76,37.29) .. (90.97,37.29) .. controls (109.17,37.29) and (123.93,58.36) .. (123.93,84.34) .. controls (123.93,110.33) and (109.17,131.39) .. (90.97,131.39) .. controls (72.76,131.39) and (58,110.33) .. (58,84.34) -- cycle ;
\draw    (64.51,112.27) -- (116.6,55.65) ;
\draw [color={rgb, 255:red, 0; green, 0; blue, 0 }  ,draw opacity=1 ]   (90.97,131.39) -- (91.37,152.81) ;
\draw    (90.97,37.29) -- (90.56,12.81) ;
\draw   (59.29,217.3) .. controls (59.29,192.06) and (73.77,171.59) .. (91.64,171.59) .. controls (109.51,171.59) and (124,192.06) .. (124,217.3) .. controls (124,242.54) and (109.51,263) .. (91.64,263) .. controls (73.77,263) and (59.29,242.54) .. (59.29,217.3) -- cycle ;
\draw    (69.08,251.8) -- (120.21,196.8) ;
\draw    (91.64,263) -- (92.04,283.81) ;
\draw [color={rgb, 255:red, 0; green, 0; blue, 0 }  ,draw opacity=1 ]   (91.64,171.59) -- (91.24,147.81) ;
\draw  [draw opacity=0] (118.25,108.03) .. controls (114.08,105.28) and (111,95.71) .. (111,84.34) .. controls (111,72.49) and (114.34,62.61) .. (118.79,60.35) -- (120.72,84.34) -- cycle ; \draw   (118.25,108.03) .. controls (114.08,105.28) and (111,95.71) .. (111,84.34) .. controls (111,72.49) and (114.34,62.61) .. (118.79,60.35) ;
\draw  [draw opacity=0] (65.81,187.55) .. controls (70.43,191.14) and (73.13,200.05) .. (72.14,210.18) .. controls (71.07,221.06) and (66.1,229.84) .. (60.35,231.9) -- (59.47,209.11) -- cycle ; \draw   (65.81,187.55) .. controls (70.43,191.14) and (73.13,200.05) .. (72.14,210.18) .. controls (71.07,221.06) and (66.1,229.84) .. (60.35,231.9) ;
\draw   (225,150.9) .. controls (225,106.91) and (247.83,71.26) .. (276,71.26) .. controls (304.17,71.26) and (327,106.91) .. (327,150.9) .. controls (327,194.89) and (304.17,230.55) .. (276,230.55) .. controls (247.83,230.55) and (225,194.89) .. (225,150.9) -- cycle ;
\draw    (226,139.81) -- (304,84.81) ;
\draw    (276,230.55) -- (276.63,266.81) ;
\draw    (276,71.26) -- (275.37,29.81) ;
\draw    (247,216.81) -- (326,160.81) ;
\draw  [draw opacity=0] (225,150.9) .. controls (231.14,153.15) and (236.7,164) .. (238.16,177.58) .. controls (239.04,185.76) and (238.26,193.33) .. (236.3,198.86) -- (225.01,179.17) -- cycle ; \draw   (225,150.9) .. controls (231.14,153.15) and (236.7,164) .. (238.16,177.58) .. controls (239.04,185.76) and (238.26,193.33) .. (236.3,198.86) ;
\draw  [draw opacity=0] (327,150.9) .. controls (318.53,146.41) and (311.81,130.77) .. (311.02,111.92) .. controls (310.75,105.53) and (311.19,99.45) .. (312.21,94) -- (330.64,111.1) -- cycle ; \draw   (327,150.9) .. controls (318.53,146.41) and (311.81,130.77) .. (311.02,111.92) .. controls (310.75,105.53) and (311.19,99.45) .. (312.21,94) ;
\draw   (442,147.9) .. controls (442,103.91) and (464.83,68.26) .. (493,68.26) .. controls (521.17,68.26) and (544,103.91) .. (544,147.9) .. controls (544,191.89) and (521.17,227.55) .. (493,227.55) .. controls (464.83,227.55) and (442,191.89) .. (442,147.9) -- cycle ;
\draw    (449,186.81) -- (521,81.81) ;
\draw    (493,227.55) -- (493.63,263.81) ;
\draw    (493,68.26) -- (492.37,26.81) ;
\draw    (464,213.81) -- (538,108.81) ;
\draw [->]   (129,150) -- (198,150) ;
\draw[->]    (347,150) -- (416,150) ;

\draw (156,132) node [anchor=north west][inner sep=0.75pt]    {$u$};
\draw (375,120) node [anchor=north west][inner sep=0.75pt]    {$u^{2}$};
\end{tikzpicture}
\caption{The left hand side of the pentagon equation.}\label{LHS pentagon}
\end{subfigure} 
\hfill
\begin{subfigure}{0.3\textwidth}
\begin{tikzpicture}[x=0.75pt,y=0.5pt,yscale=-.75,xscale=.5]

\draw   (102.33,53.16) .. controls (102.33,34.04) and (112.68,18.54) .. (125.45,18.54) .. controls (138.22,18.54) and (148.57,34.04) .. (148.57,53.16) .. controls (148.57,72.27) and (138.22,87.77) .. (125.45,87.77) .. controls (112.68,87.77) and (102.33,72.27) .. (102.33,53.16) -- cycle ;
\draw    (107.18,73.98) -- (143.72,32.33) ;
\draw    (125.45,87.77) -- (125.45,106.06) ;
\draw    (125.74,18.82) -- (125.45,0.81) ;
\draw    (125.45,99.87) .. controls (74.21,99.87) and (75.21,7) .. (126.02,7) ;
\draw   (102.5,244.5) .. controls (102.5,223.43) and (112.84,206.35) .. (125.6,206.35) .. controls (138.36,206.35) and (148.71,223.43) .. (148.71,244.5) .. controls (148.71,265.57) and (138.36,282.65) .. (125.6,282.65) .. controls (112.84,282.65) and (102.5,265.57) .. (102.5,244.5) -- cycle ;
\draw    (107.35,267.45) -- (143.86,221.55) ;
\draw    (125.6,282.65) -- (125.6,302.81) ;
\draw    (125.6,292.73) .. controls (172.38,292.27) and (175.8,196.74) .. (125.75,196.74) ;
\draw   (100.84,152.87) .. controls (100.84,132.78) and (112.07,116.5) .. (125.92,116.5) .. controls (139.77,116.5) and (151,132.78) .. (151,152.87) .. controls (151,172.96) and (139.77,189.25) .. (125.92,189.25) .. controls (112.07,189.25) and (100.84,172.96) .. (100.84,152.87) -- cycle ;
\draw    (105.8,174.46) -- (145.43,130.69) ;
\draw    (125.92,189.25) -- (126.23,205.81) ;
\draw    (125.92,116.5) -- (125.61,97.57) ;
\draw  [draw opacity=0] (109.52,169.1) .. controls (108.26,163.92) and (110.55,156.05) .. (115.92,148.77) .. controls (123.05,139.1) and (133.01,133.98) .. (138.66,136.91) -- (125.9,155.48) -- cycle ; \draw   (109.52,169.1) .. controls (108.26,163.92) and (110.55,156.05) .. (115.92,148.77) .. controls (123.05,139.1) and (133.01,133.98) .. (138.66,136.91) ;
\draw   (310,151.53) .. controls (310,106.95) and (333.06,70.81) .. (361.5,70.81) .. controls (389.94,70.81) and (413,106.95) .. (413,151.53) .. controls (413,196.11) and (389.94,232.25) .. (361.5,232.25) .. controls (333.06,232.25) and (310,196.11) .. (310,151.53) -- cycle ;
\draw    (313,177.81) -- (388,84.81) ;
\draw    (361.5,232.25) -- (362.14,269) ;
\draw    (361.5,70.81) -- (360.86,28.81) ;
\draw    (338,223.81) -- (410,125.81) ;
\draw[->]    (179,160) -- (297,160) ;

\draw (230,130) node [anchor=north west][inner sep=0.75pt]    {$u^{6}$};
\end{tikzpicture}
\caption{The right hand side of the pentagon equation.}\label{RHS pentagon}
\end{subfigure}
\caption{The two sides of the pentagon equation in $\arrows$.}
\end{figure}

\begin{lemma}\label{lem: hexagon in arrows}
The map  \eqref{eq:ue} respects the hexagon relations. That is, the $\varepsilon$-images of the two sides of each hexagon relation coincide in $\arrows$. 
\end{lemma}

\begin{proof} 
Recall that the first hexagon relation \eqref{PaCD hexagon} in $\PaCD$ is 
\begin{equation*} \label{hexagon PaCD} 
d_1\virtualcrossing = \Associator \cdot d_0\virtualcrossing \cdot ((23) \Associator^{-1})\cdot (23)d_3\virtualcrossing\cdot (321)\Associator. 
\end{equation*} 
The element $\virtualcrossing\in\Hom_{\PaCD(2)}((12),(21))$ is the empty chord diagram on the transposition map; $\varepsilon$ maps this to the skeleton element $\twist\in A$. As before, we write $d_i\bubble$ or $d_j\twist$ for the skeleton closures of the morphisms $d_i\Associator$ and $d_j\virtualcrossing$ in $\PaCD$. One can then check, after fully unzipping the composite in Figure~\ref{hexagon pic} that the equation
 \begin{equation}\label{hexagon arrows} d_1\twist = \bubble \cdot d_0\twist \cdot ((23) \bubble^{-1})\cdot (23)d_3\twist \cdot (321)\bubble\end{equation} holds in $\arrows$. The second hexagon equation is similar. 
  \end{proof} 
 
 \begin{figure}[h]
\begin{tikzpicture}[x=0.75pt,y=0.75pt,yscale=-.75,xscale=.75]

\draw   (88.31,185.38) .. controls (61.16,184.2) and (39.83,169.67) .. (40.66,152.93) .. controls (41.49,136.19) and (64.17,123.58) .. (91.32,124.76) .. controls (118.47,125.94) and (139.81,140.47) .. (138.98,157.21) .. controls (138.15,173.95) and (115.46,186.56) .. (88.31,185.38) -- cycle ;
\draw    (40.66,152.93) -- (15.06,152.19) ;
\draw  [color={rgb, 255:red, 0; green, 0; blue, 0 }  ][line width=0.75] [line join = round][line cap = round] (42.37,163.88) .. controls (42.37,163.88) and (42.37,163.88) .. (42.37,163.88) ;
\draw  [draw opacity=0] (172.08,172.34) .. controls (167.26,167.02) and (164.46,160.73) .. (164.45,153.97) .. controls (164.4,134.79) and (186.78,119.18) .. (214.44,119.1) .. controls (242.09,119.02) and (264.55,134.51) .. (264.6,153.69) .. controls (264.61,159.52) and (262.55,165.02) .. (258.91,169.85) -- (214.52,153.83) -- cycle ; \draw   (172.08,172.34) .. controls (167.26,167.02) and (164.46,160.73) .. (164.45,153.97) .. controls (164.4,134.79) and (186.78,119.18) .. (214.44,119.1) .. controls (242.09,119.02) and (264.55,134.51) .. (264.6,153.69) .. controls (264.61,159.52) and (262.55,165.02) .. (258.91,169.85) ;
\draw    (172.59,172.9) .. controls (168.95,107.9) and (264.16,251.11) .. (258.5,170.4) ;
\draw    (172.59,172.9) .. controls (172.83,234.66) and (253.92,118.43) .. (258.5,170.4) ;
\draw    (265,153.81) -- (293.18,153.54) ;
\draw    (138.04,154.38) -- (164.46,154.36) ;
\draw    (59.23,130.86) -- (131.67,172.26) ;
\draw   (340.49,187.24) .. controls (313.57,186.01) and (292.39,170.92) .. (293.18,153.54) .. controls (293.98,136.16) and (316.45,123.07) .. (343.37,124.3) .. controls (370.29,125.53) and (391.47,140.62) .. (390.67,158) .. controls (389.88,175.38) and (367.41,188.47) .. (340.49,187.24) -- cycle ;
\draw    (391,152.81) -- (415,152.81) ;
\draw  [color={rgb, 255:red, 0; green, 0; blue, 0 }  ][line width=0.75] [line join = round][line cap = round] (294.89,164.91) .. controls (294.89,164.91) and (294.89,164.91) .. (294.89,164.91) ;
\draw    (298.68,168.97) -- (383.36,140.93) ;
\draw  [draw opacity=0] (505.92,135.98) .. controls (511.03,140.58) and (514.03,146.1) .. (514.06,152.05) .. controls (514.15,168.28) and (492.17,181.56) .. (464.96,181.71) .. controls (437.76,181.86) and (415.63,168.83) .. (415.54,152.6) .. controls (415.5,146.99) and (418.11,141.73) .. (422.65,137.25) -- (464.8,152.32) -- cycle ; \draw   (505.92,135.98) .. controls (511.03,140.58) and (514.03,146.1) .. (514.06,152.05) .. controls (514.15,168.28) and (492.17,181.56) .. (464.96,181.71) .. controls (437.76,181.86) and (415.63,168.83) .. (415.54,152.6) .. controls (415.5,146.99) and (418.11,141.73) .. (422.65,137.25) ;
\draw    (422.27,137.63) .. controls (423.63,85.96) and (508.12,191.88) .. (507.15,137.15) ;
\draw    (422.27,137.63) .. controls (420.38,193.12) and (503.67,82.54) .. (507.15,137.15) ;
\draw    (513.28,150.19) -- (532.52,150.08) ;
\draw   (579.02,180.85) .. controls (554.55,180.63) and (534.53,166.03) .. (534.29,148.24) .. controls (534.05,130.46) and (553.69,116.22) .. (578.15,116.45) .. controls (602.61,116.68) and (622.64,131.28) .. (622.88,149.06) .. controls (623.12,166.85) and (603.48,181.08) .. (579.02,180.85) -- cycle ;
\draw    (654,148.81) -- (622.88,149.06) ;
\draw  [color={rgb, 255:red, 0; green, 0; blue, 0 }  ][line width=0.75] [line join = round][line cap = round] (534.14,158.65) .. controls (534.14,158.65) and (534.14,158.65) .. (534.14,158.65) ;
\draw    (547.09,125.19) -- (620,159.81) ;

\end{tikzpicture}
\caption{The right hand side of the image of the hexagon relation.
}\label{hexagon pic}
\end{figure}
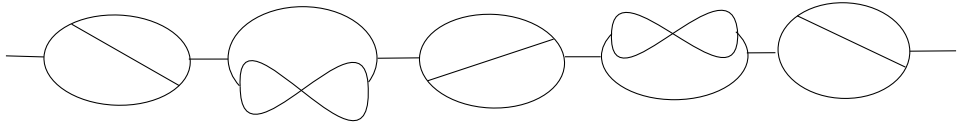

{\em Proof of Theorem \ref{inclusion of Hopf groupoids}.} By Lemmas~\ref{lem: pentagon in arrows}~and~\ref{lem: hexagon in arrows}, $\varepsilon$ is well-defined. By definition, every $\Hom_{\PaCD(n)}(p_1, p_2)$ is a (completed, filtered coassocaitive) coalgebra so $\varepsilon$ is a map of (complete filtered, coassociative) coalgebras. Finally, by Lemma~\ref{lem:emultiplicative}, $\varepsilon$ is multiplicative.
\qed

\subsubsection{The image of $\grt_1$ in $\Aut_v(\arrows)$}
Recall that the Alekseev--Torossian map of Lie algebras $\varrho:\mathfrak{grt}_1\rightarrow \mathfrak{krv}_2$ induces a group homomorphism $\tilde{\varrho}$:

\begin{equation}\label{eq:rho}
\begin{tikzcd}
\Aut(\PaCD) \arrow[d, "\cong", swap]\arrow[r, dotted, "\tilde{\varrho}"] & \Aut_v(\arrows) \arrow[d, "\cong"] \\
\grt_1 \arrow[r, hookrightarrow, "\exp(\varrho)"] & \krv. 
 \end{tikzcd}
 \end{equation}
 
 The fact that $\tilde{\varrho}$ is a group homomorphism is automatic since $\tilde{\varrho}$ is a composite of group homomorphisms. Moreover, $\tilde{\varrho}$ is injective because $\varrho$ is injective. The image of $\Aut(\PaCD)$ in $\Aut_{v}(\arrows)$ is a subgroup of $\krv$ isomorphic to $\grt_1$. In Theorem~\ref{thm: GRT as Aut(A)} we give a direct description of $\grt_1$ as automorphisms of arrow diagrams whose value on $\bubble$ satisfies an additional condition. We start with two useful lemmas.

 \begin{lemma}\label{lem:bub}
 	For any $G_1, G_2 \in \Aut_{v}(\arrows)$, if $G_1(\bubble)=G_2(\bubble)$ then $G_1=G_2$.
 \end{lemma}  

\begin{proof} 
By Theorem \ref{thm: aut(A)=KRV}, an automorphism $G\in \Aut_v(\arrows)$ is determined by $G(\vertex)=N$. In particular, $\Upsilon(N)=e^we^n$, and $w$ is uniquely determined by $n$, so $G$ is determined by $n$.

For $i=1,2$, set $\Upsilon(G_i(\vertex))=e^{w_i}e^{n_i}$, and it is enough to prove that $n_1=n_2$. Since the $G_i$ are circuit algebra maps, we have
$$\Upsilon(G_i(\bubble))=\Upsilon((G_i(\vertex)^{12,3})^{-1}(G_i(\vertex)^{1,2})^{-1} G_i(\vertex)^{2,3}G_i(\vertex)^{1,23}).$$
The $\TAut_3$ (``tree'') component of this expression is:
$$(e^{-n_i})^{12,3}(e^{-n_i})^{1,2}) (e^{n_i})^{2,3}(e^{n_i})^{1,23} \in \TAut_3. $$     
Then, expanding the $\TAut_3$ component of $G_1(\bubble)G_2(\bubble)^{-1}=1$ degree by degree, using the above equation, we obtain that $n_1=n_2$, which implies the result.
\end{proof}

Recall that for $x\in \arrows(\uparrow_n)$ we denote $x^*=A_1A_2...A_n(x)$, where $A_i$ is the adjoint (direction switch) operation on strand $i$. The following fact from \cite{BND:WKO2} helps prove equalities between elements in $\arrows(\uparrow_n)$:
 
 \begin{lemma}\label{lem:TreeLemma}
 \cite[Lemma 4.16]{BND:WKO2}
 Two group-like elements $a$ and $b$ in $\arrows(\uparrow_n)$ are equal if and only if their ``tree components'' -- that is, their projections to $\exp(\tder_n \oplus \mathfrak{a}_n)$ -- are equal, and in addition $aa^*=bb^*$.  
 \end{lemma}

\begin{theorem}\label{thm: GRT as Aut(A)}
	The (isomorphic) image of $\grt_1$ in $\Aut_v(\arrows)$ is the subgroup $$\EB=\{\tilde{G} \in \Aut_v(\arrows)~:~\tilde{G}(\bubble)=\varepsilon(G(\Associator)) \text{ for some (unique) }G\in\Aut(\PaCD)\}.$$
\end{theorem}

\begin{proof} By the diagram \eqref{eq:rho}, the isomorphic image of $\grt_1$ in $\Aut_v(A)$ is $\tilde{\varrho}(\Aut(\PaCD))$. Therefore, we need to show that $\tilde{\varrho}(\Aut(\PaCD))=\EB$.
	
First, for an element $G \in \Aut(\PaCD)$, we show that $\tilde{G}:=\tilde{\varrho}(G) \in \EB$.
In the proof of Theorem~\ref{thm: GRT is Aut}, the isomorphism $\Aut(\PaCD)\cong \grt_1$ comes from the identification of $G(\Associator)=\Psi(t_{1,2},t_{2,3})$ where $\Psi=e^{\psi}$ for some $\psi\in\mathfrak{grt}_1$.  Applying $\exp(\varrho)$, we get an element $\exp(\varrho(\psi)):=\alpha \in \krv$. By Proposition 9.12 of \cite{AT12}, we have\footnote{Note that there is an inversion compared to \cite{AT12} as we follow the notation of \cite{ATE10}.}
\begin{equation}\label{eq:apsi}  \Psi(t^{1,2},t^{2,3})=\alpha^{12,3}\alpha^{1,2}(\alpha^{2,3})^{-1}(\alpha^{1,23})^{-1}\end{equation}
Recall that the element $\alpha \in \krv$ uniquely determines the corresponding Duflo series $s$. By Theorem~\ref{thm: aut(A)=KRV}, $\tilde{G}=\Theta^{-1}(\alpha)\in\Aut_v(\arrows)$ is given by
\[ 
\tilde{G}(\vertex):=\Upsilon^{-1}((\mathcal{J}(\alpha^{-1}))^{-\frac{1}{2}}\alpha^{-1}) \quad \text{and}\quad \tilde{G}(\upcap):=e^{-\kappa^{-1}(s)/2}.
\] 

Since $\tilde{G}$ is an automorphism of circuit algebras, we have (explained in more detail below):
\begin{equation}\label{eq:Gtilde}
\tilde{G}(\bubble)=(\tilde{G}(\vertex)^{12,3})^{-1}(\tilde{G}(\vertex)^{1,2})^{-1} \tilde{G}(\vertex)^{2,3}\tilde{G}(\vertex)^{1,23}.
\end{equation}
Here the cosimplicial notation on $\tilde{G}(\vertex)$ is the same as in Section~\ref{sec: cosimplicial structure}, keeping in mind that $\tilde{G}(\vertex)\in \arrows(\uparrow_2)$ by Lemma \ref{lemma: VI}. For example, $(\tilde{G}(\vertex)^{12,3})^{-1}$ is the value of the bottom vertex in $\bubble$. It is inverted as it is a vertex of the opposite orientation, see \cite[Proof of Theorem 4.9]{BND:WKO2}. A priori, this value is an arrow diagram on the top two strands of the bottom vertex. However, by the VI relation it can be ``pushed up'' to the middle three strands of the skeleton $\bubble$, this effectively doubles the first strand, hence the superscript ``$12,3$''.

Denote by $a$ the right hand side of Equation~\eqref{eq:Gtilde}, and by $b$  the $\Upsilon^{-1}$ value of the right hand side of equation \eqref{eq:apsi}. We use Lemma~\ref{lem:TreeLemma} to show that $a=b$. Observe that the tree components of $a$ and $b$ are the same (by definition of $\tilde{G}$). Furthermore, $aa^*=1$ by the \eqref{U'} equation. On the other hand, $bb^*=1$, as follows. The isomorphism $\Upsilon^{-1}$ maps $t^{i,j}$ to a sums of two horizontal arrows, and hence $(\Upsilon^{-1}(t^{ij}))^*= - \Upsilon^{-1}(t^{i,j})$. The group-like property $\Psi=e^{\psi}$ then implies that $\Upsilon^{-1}(\Psi(t^{1,2},t^{2,3}))\left(\Upsilon^{-1}(\Psi(t^{1,2},t^{2,3}))\right)^*=bb^*=1$. Therefore, $a=b$. In other words, this means that the ``wheel components'' of the vertex values in Equation~\eqref{eq:Gtilde} all cancel.

Applying $\Upsilon^{-1}$ to Equation \ref{eq:apsi}, note that $\Upsilon^{-1}\Psi(t^{1,2},t^{2,3})=\varepsilon(G(\Associator))$ by the definition of the map $\varepsilon$. Therefore, 
$\tilde{G}(\bubble)=\varepsilon(G(\Associator)),$ so $\tilde{G}\in \EB$.

Conversely, assume $\tilde{G}\in \EB$; we show that $\tilde{G}$ is in the image of $\tilde{\varrho}$. 
The statement $\tilde{G}\in \EB$ means that $\tilde{G} \in \Aut_v(\arrows)$ and $\tilde{G}(\bubble)=\varepsilon(G(\Associator))$ for some $G\in\Aut(\PaCD)$. 

Such a $G$ is unique, because $\varepsilon$ is injective and the value $G(\Associator)$ determines $G$. We will show that $\tilde{G}=\tilde{\varrho}(G)$, so $\tilde{G} \in \tilde{\varrho}(\Aut(\PaCD))$.
By the construction above, we know that $\tilde{\varrho}(G) \in \Aut_v(\arrows)$ satisfies $\tilde{\varrho}(G)(\bubble)=\varepsilon(G(\Associator))$. Therefore,
$$\tilde{\varrho}(G)(\bubble)=\tilde{G}(\bubble).$$
Applying Lemma \ref{lem:bub}, we get $\tilde{G}=\tilde{\varrho}(G)$.
\end{proof}

\bibliographystyle{amsalpha}
\bibliography{kv}
\end{document}